\documentclass{amsart}[12pt,letter]

\usepackage{color,amssymb,enumerate}
\usepackage{mathabx}
\usepackage{thmtools}
\usepackage{thm-restate}
\newsavebox\tempbox
\let\svwidetilde\widetilde
\renewcommand\widetilde[1]{\sbox\tempbox{$#1$}\svwidetilde{\usebox{\tempbox}}}
\usepackage{mathtools}

\usepackage{tikz}
\usetikzlibrary{arrows}

\usepackage{comment}

\usepackage[small,nohug,heads=vee]{diagrams}
\diagramstyle[labelstyle=\scriptstyle]

\usepackage{hyperref}
\usepackage[all]{hypcap}
\usepackage[alphabetic]{amsrefs}
\numberwithin{figure}{section}
\numberwithin{equation}{section}

\newtheorem{theorem}{Theorem}[section]
\newtheorem{corollary}[theorem]{Corollary}
\newtheorem{lemma}[theorem]{Lemma}
\newtheorem{proposition}[theorem]{Proposition}

\theoremstyle{definition}
\newtheorem{definition}[theorem]{Definition}
\newtheorem{definition*}{Definition}
\theoremstyle{remark}
\newtheorem{remark}[theorem]{Remark}

\newtheorem*{claim}{Claim}

\newcommand{\defin}[1]{\textit{#1}}

\newcommand{\Ind}{\operatorname{Ind}}
\newcommand{\codim}{\operatorname{codim}}

\newcommand{\R}{\mathbb{R}}
\newcommand{\C}{\mathbb{C}}
\newcommand{\Q}{\mathbb{Q}}
\newcommand{\CP}{\mathbb{CP}}

\newcommand{\Z}{\mathbb{Z}}

\newcommand{\Id}{\operatorname{Id}}
\newcommand{\id}{\operatorname{id}}

\newcommand{\Y}{\mathcal{Y}}
\newcommand{\X}{\mathcal{X}}
\newcommand{\A}{\mathcal{A}}

\newcommand{\pb}{{}^{*}}

\newcommand{\T}{{\,T}}

\newcommand{\M}{\mathcal{M}}
\newcommand{\U}{\mathcal{U}}
\newcommand{\V}{\mathcal{V}}
\newcommand{\MM}{\widetilde{ \M }}
\newcommand{\e}{\operatorname{e}}

\newcommand{\loc}{\text{loc}}

\newcommand{\CZ}{\text{CZ}}

\newcommand{\tensor}{\otimes}

\DeclareMathOperator{\gw}{GW}             %
\DeclareMathOperator{\Crit}{Crit}             %
\DeclareMathOperator{\ind}{ind}             %
\DeclareMathOperator{\supp}{supp}             %
\DeclareMathOperator{\rot}{rot}         %
\DeclareMathOperator{\domain}{domain}         %
\DeclareMathOperator{\target}{target}         %

\newcommand{\CC}{\mathcal C}

\newcommand{\F}{\mathcal{F}}
\newcommand{\ev}{\operatorname{ev}}

\newcommand{\wc}[1]{\widecheck {#1}}            %
\newcommand{\wh}[1]{\widehat {#1}}              %

\newcommand{\ee}{\mathbf{e}}
\newcommand{\ff}{\mathbf{f}}
\renewcommand{\gg}{\mathbf{g}}

\newcommand{\moduliFloerCyl}{\mathcal M}
\newcommand{\topo}{\mathrm{topo}}
\newcommand{\geom}{\mathrm{geom}}

\begin{document}

\title{Symplectic homology of complements of smooth divisors}
\date{\today}
\author{Lu\'is Diogo}
\author{Samuel T. Lisi}

\begin{abstract}
    If $(X, \omega)$ is a closed symplectic manifold, and $\Sigma$ is a
    smooth symplectic submanifold Poincar\'e dual to a positive multiple of $\omega$, 
    then $X \setminus \Sigma$ can be completed to a Liouville manifold $(W, d\lambda)$.

    Under monotonicity assumptions on $X$ and on $\Sigma$, 
    we construct a chain complex whose homology computes the symplectic homology
    of $W$. We show the differential is given in terms of Morse contributions, 
    terms computed from Gromov--Witten invariants of $X$ relative to $\Sigma$ 
    and terms computed from the Gromov--Witten invariants of $\Sigma$.

    We use a Morse--Bott model for symplectic homology. Our proof involves
    comparing Floer cylinders with punctures to pseudoholomorphic curves in
    the symplectization of the unit normal bundle to $\Sigma$.
\end{abstract}

\maketitle
\tableofcontents

\section{Introduction} \label{Sec:Introduction}

Symplectic homology 
is a version of Hamiltonian Floer homology for a class of open symplectic 
manifolds with contact boundary, including Liouville manifolds \cites{SymplecticHomology,ViterboSH}. 
Some of the applications of symplectic homology include special cases of 
the Weinstein conjecture \cite{ViterboSH}, and a proof that there are 
infinitely many distinct symplectic structures on $\R^{2n}$, for $n\geq 4$ \cite{McLean1}.

One of the interests of symplectic homology is the
fact that it relates to other important tools of symplectic
topology, including contact homology \cite{BOExactSequence},
Rabinowitz Floer homology \cite{CieliebakFrauenfelderOancea} and the Fukaya category
\cites{AbouzaidGeneration,GanatraThesis}, as well as to string
topology \cite{AbbondandoloSchwarzRing}.  In this paper, we explore a
relation between symplectic homology and quantum cohomology. For other
relations between these two invariants, see Seidel \cite{SeidelLFIII} 
as well as Borman and Sheridan \cite{BormanSheridan}.

Symplectic homology has a wealth of algebraic structures, including a product
and a Batallin--Vilkovisky operator, as observed by Seidel
\cite{SeidelBiasedView}. This paper will not address the computation of these
algebraic structures, but we intend to explore that direction in future work
using the methods developed here.   

Despite its great relevance, symplectic homology is often very hard to
compute. Among the few known results is the fact that it vanishes for
subcritical Weinstein manifolds (including $(\R^{2n},\omega_{std})$)
\cite{CieliebakHandleAttaching} and 
for flexible Weinstein domains \cite{BEEsurgery}.
Symplectic homology of the cotangent
bundle of a closed, orientable spin manifold $M$ is isomorphic
as a ring to the Chas--Sullivan string topology ring of $M$
\cites{AbbondandoloSchwarzRing,AbouzaidViterbo}.  There is also a
surgery formula for the symplectic homology of a Weinstein manifold
\cites{BEEsurgery,BEEProduct}, which has been used to compute new examples
\cites{EkholmNg,EtguLekiliKoszul}.

In this paper, we explore the problem of computing the symplectic homology
of a special class of symplectic manifolds, namely complements of certain
symplectic hypersurfaces $\Sigma$ in closed symplectic manifolds $X$. An
important source of examples comes from smooth ample divisors in projective
varieties.  Allowing for $\Sigma$ to be a normal crossings divisor would
be a significant improvement in generality, but is currently beyond
our reach. Important progress has been made in that direction, including
\cites{Pascaleff,NguyenThesis,GanatraPomerleano1,GanatraPomerleano2,McLeanMBSequence}.

Under these assumptions on $X$ and $\Sigma$, the complement
$X\setminus\Sigma$ is the interior of a Liouville domain, hence its
symplectic homology can be defined.  The goal of this paper is to
compute the symplectic homology groups of such divisor complements, in terms of counts of rigid
pseudoholomorphic spheres in $\Sigma$ and $X$. These counts can sometimes be expressed as
absolute Gromov--Witten invariants of $(\Sigma,\omega_\Sigma)$ and relative Gromov--Witten
invariants of $(X,\Sigma,\omega)$. 

The following is the simplest formulation of our main result. We call a Morse function {\em perfect} 
if the corresponding Morse differential vanishes.  

\begin{theorem}
If $\Sigma$ and $X \setminus \Sigma$ admit perfect Morse functions, then
there is a chain complex computing the symplectic homology of $X\setminus
\Sigma$, whose differential is expressed in terms of Gromov--Witten invariants of
$(\Sigma,\omega_\Sigma)$, relative Gromov--Witten invariants of $(X,\Sigma,\omega)$ and
the Morse differentials in the contact boundary $Y$ and in $X\setminus \Sigma$.  
\end{theorem}

For more general and precise statements, see Theorem \ref{T:differential}
and Lemmas \ref{lacunary GW} and \ref{augmentation rel GW}. 
The strategy of the proof is inspired by
\cite{BOExactSequence}, which relates the symplectic homology of a symplectic manifold with contact type boundary with the 
contact homology of the boundary. The idea consists of stretching the neck along the
boundary $Y$ of a tubular neighborhood of $\Sigma$ in $X$, and keeping track of 
the degenerations of Floer cascades that contribute to the differential.
The limit configurations will include pieces in the symplectization of
the contact manifold $Y$, and pieces in (the completion of) $X\setminus
\Sigma$. The former project to pseudohomolorphic spheres
in $\Sigma$. 
These are related to Gromov--Witten invariants of $(\Sigma,\omega_\Sigma)$.
The curves in $X \setminus \Sigma$ 
are related to pseudoholomorphic spheres in $X$ with tangency
constraints in $\Sigma$. Under suitable hypotheses, they are counted by relative Gromov--Witten invariants of $(X,\Sigma,\omega)$. 

\begin{remark}

The theorem also holds if $X$ and $\Sigma$ are semipositive
\cite{McDuffSalamon}*{Definition 6.4.1} with {\em negative} monotonicity
constants.
Such symplectic manifolds have no non-constant 
$J$-holomorphic spheres for regular $J$ \cite{McDuffSalamonOld}*{proof of Lemma 5.1.3},
and hence the symplectic homology differential is obtained from purely Morse
theoretic terms.
(This was pointed out to us by Fran\c cois Charest).

\end{remark}

An important step in our argument is the proof of a correspondence between moduli spaces of solutions to Floer's equation in $\R\times Y$ and moduli spaces of punctured pseudoholomorphic spheres in $\R\times Y$. This is done by showing that the {\em difference} $\bf e$ (in a sense we make precise) between a Floer solution and a corresponding pseudoholomorphic curve solves a certain PDE, and then proving that solutions to this PDE exist and are unique (see Sections \ref{S:ansatz_Floer_to_holo} and \ref{S:ansatz_holo_to_Floer}). 

\begin{remark}
 As mentioned earlier, this paper drew much inspiration from
 \cite{BOExactSequence} and its use of neck-stretching. We should
 point out some of the main differences between the two papers. In
 \cite{BOExactSequence}, the authors consider a more general
 situation, in which the symplectic manifold is not necessarily obtained
 from the complement of a smooth symplectic divisor.  For this reason, the
 contact boundaries $Y$ they consider are not necessarily prequantization
 bundles over $\Sigma$, as in our case.  The fact that our $Y$ are so
 special has a number of advantages which we use,
 notably in Part \ref{Ansatz}. 
 This enables
 us to obtain a bijection between moduli spaces of Floer trajectories
 and of pseudoholomorphic curves (rather than just a continuation map relating symplectic homology and a non-equivariant version of contact homology),
 which is then used to compute the symplectic homology differential
 explicitly. 
 As studied in our paper \cite{DiogoLisiSplit}, this also allows us to achieve
 transversality by geometric arguments involving monotonicity and automatic
 transversality.
 For more on the relation between this work and \cite{BOExactSequence}, see Remark \ref{Cone}.
\end{remark}

\begin{remark}
The contact boundaries $Y$ of the Liouville domains $X\setminus
\Sigma$ are such that the Reeb flow is a free $S^1$-action. As explained in
\cite{SeidelBiasedView}*{(3.2)}, there is a spectral sequence converging
to symplectic homology, whose first page consists of homology groups of
$Y$ and $X\setminus \Sigma$. One can think of this paper as computing
the differential on that page of the spectral sequence (and there happen to
be no higher order differentials). This is also related to Remark \ref{Cone}.   
\end{remark}
 
This paper is complemented by \cite{DiogoLisiSplit} and is organized as follows. 
Part 1 provides more details about
our geometric setup and the almost complex structures that we use. 
Part 2 describes the chain complexes that compute symplectic
homology, before and after stretching the neck. We use a Morse--Bott
approach, so the differentials count {\em Floer cylinders with cascades}. The full discussion of the stretched  
case (which we refer to as {\em split}) is given in \cite{DiogoLisiSplit}, including the relevant transversality results
and an explanation of how our monotonicity assumptions imply that the Floer cylinders with cascades 
that can contribute to the split differential can only be of four simple types. 
Part 3 is the most technical of this paper, where we explain why counting punctured 
Floer cylinders in $\R\times Y$ is equivalent to counting punctured 
pseudohomolorphic cylinders in $\R\times Y$, and how the latter can be
expressed in terms of counts of pseudoholomorphic spheres in $\Sigma$
(which are in turn related to Gromov--Witten invariants).  Finally, in Part 4, 
we illustrate this computation scheme by calculating the symplectic homology of $T^*S^2$ 
(which, as we mentioned earlier, is isomorphic to the well-known homology of the free loop 
space of $S^2$, see for instance \cite{CohenJonesYan}).

Besides the references mentioned throughout the paper, this article shares 
some important ideas with several works in the literature. A few of these are 
\cites{BourgeoisThesis,GHK,EliashbergPolterovichQuasi,RitterNegativeBundles,
FabertII,Uebele1,Uebele2,Fauck,Kerman,AlbersGuttHein,KimKwonLee}.

\subsection*{Acknowledgements}
Our work has benefitted greatly from help and guidance we have received from
many people. We thank especially Yasha Eliashberg. We are also grateful for
helpful discussions with and feedback from
Mohammed Abouzaid, 
Paul Biran,
Strom Borman,
Fr\'ed\'eric Bourgeois,
Kai Cieliebak,
Oliver Fabert,
Joel Fish,
Eleny Ionel,
Janko Latschev,
Melissa Liu,
Dusa McDuff,
Mark McLean,
Leonid Polterovich,
Daniel Pomerleano,
Dietmar Salamon,
Mohammad Tehrani,
Jean-Yves Welschinger, and 
Aleksey Zinger.
Thanks are also due to the anonymous referee for a careful reading of the text and helpful suggestions. 

S.L.~was partially supported by the ERC Starting Grant of Fr\'ed\'eric Bourgeois
StG-239781-ContactMath and also by Vincent Colin's ERC Grant geodycon.

L.D.~thanks Stanford University, ETH Z\"urich, Columbia University and Uppsala University 
for excellent working conditions. L.D.~was partially supported by the Knut and Alice Wallenberg Foundation.

\part{Setup, almost complex structures and neck-stretching}

\section{Almost complex structures and neck-stretching}

In this section, we consider a local model of the divisor, and identify a class
of almost complex structures that are well-behaved in this local model.

\subsection{Symplectic hyperplane sections}

We adopt some of our terminology from \cite{BiranKhanevsky}*{Section 2.2}, 
with minor modifications.
A variant on this construction is also explicitly described in
\cite{Albers_Frauenfelder_Negative_Line_Bundles}*{Section 3.1}.

Let  $(X,\omega)$ be a closed symplectic manifold such that
$[K\omega]$ admits a lift to $H^2(X;\Z)$ for some $K>0$.
We refer to a closed codimension 2 symplectic submanifold
$\Sigma^{2n-2}\stackrel{\iota}{\subset} X$ as a \defin{symplectic
divisor}.  Write $\omega_\Sigma= \iota^*\omega$.

Let $\pi \colon E\to \Sigma$ be a complex line bundle over $\Sigma$
such that $c_1(E)= [K\omega_\Sigma]$ in $H^2(\Sigma;\R)$, and  
choose a Hermitian metric on $E$. 
Then, there exists a connection 1-form $\Theta\in
\Omega^1(E\setminus \Sigma)$ such that $d\Theta = - K \pi^*\omega_\Sigma$. 
\footnote{In \cite{BiranKhanevsky}, $\Theta$ is referred to as
the global angular 1-form $\alpha^{\nabla}$.} 
 If the coordinate $\rho$ stands for the norm in $E$ measured with
respect to the Hermitian metric, and $f \colon [0, +\infty) \to (0, 1]$
is a smooth function with $f(0) = 1$, $f' < 0$, then 
\begin{equation} \label{E:connectionForm}
\omega_E:= -d \left( \frac{f( \rho^2)}{K} \Theta \right) = -\frac{2 \rho f'( \rho^2)}{K} d\rho \wedge \Theta + f(\rho^2) \pi^* \omega_\Sigma 
\end{equation}
is a symplectic form on $E\setminus \Sigma$ that extends smoothly to all of $E$. 
Its restriction to the $0$-section is $\omega_\Sigma$. Following  
\cite{BiranKhanevsky}, we will explicitly use $f(\rho^2) = \e^{-\rho^2}$. 
By construction, this symplectic form is exact on $E \setminus 0$.
The Liouville vector field dual to the primitive $-\frac{f( \rho^2)}{K} \Theta$ is 
\[
    \frac{f(\rho^2)}{2\rho f'(\rho^2)} \frac{\partial}{\partial \rho},
\]
which points inwards along the fibres. 
This then induces a contact form on the $\rho_0$--sphere bundle, and we denote this space 
by $Y$. This contact form is 
\[
    \alpha_{\rho_0} = -\frac{f(\rho_0^2)}{K} \Theta|_Y.
\]
Observe that by following the flow of the Liouville vector field, we obtain that 
$E\setminus \Sigma$ 
equipped with the symplectic form $\omega_E$ is then symplectomorphic 
to a piece of symplectization $\left ( (-\infty, -\ln f(\rho_0^2)) \times Y,
d(\e^r \alpha_{\rho_0}) \right )$. 

With a choice of contact form coming from this construction, 
$(Y,\alpha)$ is a {\em prequantization bundle} associated to $(\Sigma, K\omega_\Sigma)$. 
Different choices of $\rho_0$ give diffeomorphic sphere bundles (by scaling in the
fibre direction), but change the scaling of the contact form. It will be convenient to
take $\alpha = - \frac{1}{K} \Theta|_Y$. This corresponds to a real blow-up
of the $0$-section (i.e.~a limit as $\rho_0 \to 0$).

Say that a {\em symplectic tubular neighborhood} of the symplectic divisor $\Sigma \stackrel{\iota}{\subset} X$ is an extension of $\iota$ to a
symplectomorphism $\varphi \colon \U \to X$, where $\U$ is a disk neighborhood of the zero section on a complex line bundle $E$ over $\Sigma$, equipped with a Hermitian metric and a symplectic form $\omega_E$ as above.  
It is well known that every symplectic divisor has a symplectic tubular
neighborhood (see \cite{McDuffSalamonIntro}), where $E$ is the symplectic normal bundle to $\Sigma$ in $X$.

\begin{definition}
The symplectic divisor $\Sigma$ is a \defin{symplectic hyperplane section}
of $X$ if $[\Sigma]$ is Poincar\'e dual (over $\Q$) to $[K\omega]$ for some $K >0$.

\label{hyperplane}
\end{definition}

The condition of $\Sigma$ being a symplectic hyperplane section
sometimes includes the assumption that the complement of a neighborhood
of $\Sigma$ is a Weinstein domain. We will not need this assumption for
most of the paper, and will point out the one point where it is used,
in the statement of Theorem \ref{T:differential}.

Sources of examples of symplectic hyperplane sections 
include smooth ample divisors in smooth complex projective
varietes \cite{BiranBarriers} and Donaldson divisors in symplectic
manifolds with rational $[\omega]$ \citelist{\cite{GirouxOpenBook}
\cite{Giroux_remarks_Donaldson}}. Note that
this latter class of examples typically does not satisfy the
additional monotonicity requirements we impose on the divisor.

We now observe that the complement of the neighbourhood of the 
hyperplane section is a Liouville domain. In different language, this
is \cite{Opshtein_Polarizations}*{Lemma 2.1}. See also
\cite{Giroux_remarks_Donaldson}*{Proposition 5}.

\begin{lemma} \label{L:WisLiouville}

    Let $\Sigma \subset X$ be a symplectic hyperplane section, as in
    Definition \ref{hyperplane}, and let $E \to \Sigma$ be the normal bundle
    to $\Sigma$, endowed with a Hermitian metric and symplectic form
    $\omega_E$ as in Equation \eqref{E:connectionForm}.

    Then, 
    \begin{enumerate}[(i)]
        \item the symplectic form $\omega|_{X \setminus \Sigma}$ is exact on
            $X \setminus \Sigma$; 
        \item there exist a primitive $\lambda$ on $X \setminus \Sigma$,
            an open $\rho_0$-disk bundle $\U \subset E$ and a
            symplectic embedding $\varphi \colon \U \to X$ such that $\varphi$
            restricted to the $0$-section is the inclusion and 
            \[
                \varphi^*\lambda = - \frac{f(\rho^2)}{K} \Theta
            \]
            on $\U \setminus \Sigma$. 
    \end{enumerate}

    In particular then, $X\setminus \Sigma$ is symplectomorphic to
    $\overline{W} \setminus \partial \overline{W}$ where $\overline{W}$ is a
    Liouville domain with boundary $Y$, where $Y$ is diffeomorphic to the 
    $\rho_0$-circle bundle in $E$. 
    The primitive $\lambda$ extends to the
    boundary and induces the contact form 
    \[
        \alpha = - \frac{1}{K} \Theta|_Y
    \]
\end{lemma}
\begin{proof}

First, the assumption that $\Sigma$ is Poincar\'e dual to a multiple of
$\omega$ implies 
    that $[\omega|_{X \setminus \Sigma}] = 0 \in H^2(X \setminus
    \Sigma; \R)$. Thus, 
    $\omega|_{X \setminus \Sigma} = d\lambda_0$ for some 1-form
    $\lambda_0$ on $X \setminus \Sigma$. 

Now, by the symplectic neighbourhood theorem (see, for instance,
\cite{McDuffSalamonIntro}*{Theorem 3.30}), there exists a neighbourhood of the
$0$-section $\U$ (which we may take to be a $\rho_0$-disk bundle) and a symplectic
embedding 
\[
    \varphi_0 \colon \U \to X,
\]
inducing the inclusion on the $0$--section.

Writing $\lambda_E = -\frac{f(\rho^2)}{K} \Theta$, we have $\omega_E = d\lambda_E$.

We now have 
\[
    d \left ( \varphi_0^*\lambda_0 - \lambda_E \right ) = 0
\]
on $\U \setminus \Sigma$.

The deleted neighbourhood $\U \setminus \Sigma$ retracts on to a $\rho$-circle
bundle $Y$ (for any $0 < \rho < \rho_0$). 
The Gysin sequence applied to the circle bundle $Y \to \Sigma$ gives
\[
        0 \to H^1(\Sigma; \R) \xrightarrow{\pi^*} H^1(Y; \R) \xrightarrow{\pi_*}
        H^0(\Sigma; \R) \xrightarrow{\cup [K\omega_\Sigma]} H^2(\Sigma; \R).
    \]
    The last map is an injection from $H^0(\Sigma; \R)$ 
    to $H^2(\Sigma; \R)$ and so $\pi^*$ is an isomorphism. 

It follows then that 
\[
    \varphi_0^*\lambda_0 - \lambda_E = \pi^*\mu + dg
\]
for some closed 1-form $\mu$ on $\Sigma$ and some function $g \colon \U \to \R$. 

Extend $g\circ \varphi_0^{-1}$, defined on $\varphi_0(\U)$, to a function on $X$
that we also denote by $g$. 
The desired primitive is then $\lambda = \lambda_0 - dg$.
In particular then, \[
    \varphi_0^*\lambda = -\frac{f(\rho^2)}{K} \Theta + \pi^*\mu.
\]

We will now construct a symplectomorphism $\phi$ of $E$, sending the
$0$-section to the $0$-section, so that $\phi^*\varphi^*\lambda = \lambda_E$.

    Let $V$ be a vector field on $\Sigma$ defined by $\omega_\Sigma(V, \cdot)
    =  \mu$. The flow of $V$ induces a symplectomorphism since $\mu$ is
    closed. Now, construct a horizontal lift $\tilde V$ to $E$, with the
    properties that $\rho^2 \Theta(\tilde V) = 0$ and $\rho d\rho(\tilde V) = 0$. 
    In particular, $\tilde V$ is defined on the total space $E$ and its
    restriction to the $0$-section is $V$.

Let $\phi^t$ be the
    flow of the vector field $-\frac{1}{f(\rho^2)} \tilde V$. 
    Notice that this flow preserves the
    levels $\rho = \text{constant}$.
    
    Also notice that $\pi^*\mu(\tilde V) =
    \mu(V) =  \omega_\Sigma(V, V) = 0$ and $\omega_E(\tilde V, \cdot) =
    f(\rho^2) \pi^*\omega_\Sigma(V, \cdot) = f(\rho^2)\pi^*\mu$.
    
    Finally, we calculate:
    \begin{align*}
        \frac{d}{dt} (\phi^t)^*\left( \lambda_E - t\pi^*\mu \right ) 
        &= (\phi^t)^* \left( \frac{1}{f(\rho^2)}L_{\tilde V} \lambda_E - \pi^*\mu \right ) \\
        &= (\phi^t)^* \left( \e^{\rho^2} \omega_E(\tilde V, \cdot) - \pi^*\mu \right )\\
        &= (\phi^t)^* \left( \pi^*\omega_\Sigma( \tilde V, \cdot) - \pi^*\mu \right ) \\
        &= 0.  
    \end{align*}

    Thus, by precomposing with the time-1 map $\phi^1$, we then have 
    $(\phi^1)^*\varphi^*\lambda|_Y = \lambda_E|_Y = \alpha_{\rho_0}$, as required.

    Notice now that 
    \begin{align*}
        [0, +\infty) &\to (-\infty, 0]\\
        \rho &\mapsto \log f(\rho^2)
    \end{align*} 
    is a diffeomorphism, using the fact that $f(0) = 1$ and $f' < 0$. 
    Let $\rho(r)$ denote the inverse function.
    In particular then, $\e^r = f( \rho(r)^2 )$.
    
    Then, if $Y$ is the $\rho_0$-circle bundle in $E$, and $\alpha = -
    \frac{1}{K}\Theta|_Y$, we have a diffeomorphism
    \begin{align*}
        (\log f(\rho_0^2), 0) \times Y \to (\U \setminus \Sigma)\\
        (r, y) \mapsto \frac{\rho(r)}{\rho_0}y
    \end{align*}
    Notice that $| \frac{\rho(r)}{\rho_0}y| = \rho(r)$, so this pulls-back the
    1-form $\lambda_E$ to $-\frac{e^r}{K}\Theta|_Y = \e^r \alpha$.
    
    This therefore shows that $\varphi(\U) \setminus \Sigma$ is
    symplectomorphic to $(\log( f(\rho_0^2) ), 0) \times Y$ with symplectic
    form given by $d(\e^r \alpha)$. We may now compactify by gluing in
    $\partial \overline{W} = \{0 \} \times Y$. 

	Note that this implies that the vector field dual to $\lambda$
	points outwards at the boundary, and is thus a Liouville vector field
	and $\lambda$ is a Liouville form for $\omega$.  
        \footnote{This fact was also pointed out to us by McLean, using a different
        argument.}

\end{proof}

Fix now the data of the complex line bundle $E$ with Hermitian metric and
connection form, together with a neighbourhood of the $0$-section $\U$ and
symplectic neighbourhood $\varphi \colon \U \to X$. 
We will also henceforward assume that $f(\rho^2) = \e^{-\rho^2}$ for
concreteness.
Assume furthermore that $\U$ is chosen so that $\varphi$ extends to an
embedding $\varphi \colon \overline{\U} \to X$ of the closure of $\U$. 
We will allow our almost complex structures on $X$ to be perturbed in $\V = X \setminus \varphi(\overline \U)$, which is non-empty.

\subsection{Almost complex structures}
\label{sec:almost complex structures}

Since the Hermitian connection on the complex line bundle $E$ above defines a horizontal distribution,
any almost complex structure $J_\Sigma$ on 
$(\Sigma, \omega_\Sigma)$ may be lifted to an almost complex structure $J_E$ 
by requiring that $J_E$ preserve the horizontal and vertical subspaces,
is $i$ on the vertical subspaces and is the lift of $J_\Sigma$ to the horizontal subspaces. In particular, if $\pi \colon E \to \Sigma$ is the bundle projection map, we have $d\pi \cdot J_E = J_\Sigma \cdot d\pi$.

\begin{definition} \label{admissible J}
An almost complex structure $J_X$ on $(X, \omega)$, 
compatible with $\omega$, is called
\defin{admissible} if, in the fixed neighbourhood $\U$ of $\Sigma$,  
$J_X = \varphi_*J_E$, where $J_E$ is the lift of an almost complex structure $J_\Sigma$ on $\Sigma$.
In particular, $\Sigma$ has $J_X$ invariant tangent spaces 
(i.e.~is $J_X$-holomorphic).  
\end{definition}

In order to compute the symplectic homology of $X \setminus \Sigma$, we will need
to study Floer trajectories in the completion $W$ of $X \setminus \Sigma$ (which we construct below).
This will involve SFT-style splitting (neck-stretching) along a contact-type hypersurface
in $W$. We will construct a suitable family of almost complex structures on $W$ to implement this process.

\begin{remark}
The idea of this construction is very simply illustrated by an example where this 
becomes trivial: let $X = S^2$ with area 1, and let $\Sigma$ be a point.
It follows that $X \setminus \Sigma$ is symplectomorphic to the open unit disk. 
Its completion involves attaching $[0, +\infty) \times S^1$ to this open disk
to give $\C$ with its standard symplectic form.

However, $X$ can be identified with $\CP^1$
and thus $X \setminus \Sigma$ is, holomorphically, the complex plane, which, as we argued, is the completion of the disk.

More generally, this same phenomenon occurs. First, $X \setminus \Sigma$ must be completed
by attaching a cylindrical end. This completion is then biholomorphic (but not symplectomorphic) to $X \setminus \Sigma$.
Our construction is to start with an admissible almost complex structure $J_X$ on $X$, complete $X \setminus \Sigma$ and equip it with a family 
of almost complex structures that allow us to simultaneously define
symplectic homology and stretch the neck. In the trivial example of $X = \CP^1$, we 
stretch along a copy of $S^1$. The resulting split
almost complex manifold (as in \cite{SFTcompactness}) will consist of two levels: 
a lower level that is $\C$ and an upper level that is $\R \times S^1$. 
The lower level will then be biholomorphic to $X\setminus \Sigma$ with $J_X$, and the upper level will be biholomorphic to the normal 
bundle of $\Sigma$ with the 0-section removed.
\end{remark}

Let $Y \subset E$ be the $\rho_0$ circle bundle with respect to the
Hermitian metric on $E$. Let 
$\alpha = -\frac{1}{K}\Theta$. Recall that the induced contact form on $Y$,
seen as the $\rho_0$-circle bundle, is $\alpha_{\rho_0} = \e^{-\rho_0^2} \alpha$.

As observed earlier, there is an exact symplectomorphism 
\begin{align*}
    \psi_1 \colon (E \setminus \Sigma, \omega_E) &\to \left ( (-\infty, 0) \times Y, d(\e^{r} \alpha) \right ) \\
    v &\mapsto \big (- |v|^2 , \frac{\rho_0 v}{|v|} \big )
\end{align*}
In particular then, we have $X \setminus \Sigma$ (with primitive $\lambda$) 
has a neighbourhood of $\Sigma$ exact symplectomorphic 
(by composing $\varphi^{-1}$ with $\psi_1$) to $(-\rho_0^2, 0) \times Y$.

Since $\Sigma$ has a neighborhood in $X$ modelled after a neighborhood
of the zero section in $E$, we use $\psi_1$ to glue a copy of
$\big([0,\infty)\times Y,d(\e^r \alpha)\big)$ to 
$(X\setminus \Sigma, \omega)$. We then obtain a symplectic manifold we denote
by $W$, and call it the \defin{completion} of $X \setminus \Sigma$.
By Lemma \ref{L:WisLiouville}, there is a global primitive $\lambda$ of the
symplectic form that agrees with $\e^r \alpha$ on the cylindrical end $[0,
\infty) \times Y$.

\begin{remark}
Observe that $J_1\coloneq (\psi_1)_*J_E$ becomes singular as $r \to 0$. Indeed, if $R$ is
the Reeb vector field on $Y$ for the contact form $\alpha = - \frac{1}{K}\Theta$, we obtain:
\[
J_1 \partial_r = -\frac{1}{4r \pi K}R.
\]
Consider now the diffeomorphism  (that is not symplectic!)
\begin{align*}
\psi_2 \colon E\setminus \Sigma &\to \R\times Y \\
v &\mapsto (- \frac{1}{2\pi K} \ln|v|, \frac{\rho_0 v}{|v|})  
\end{align*}
Let $J_Y \coloneqq (\psi_2)_*J_E$ be the pushforward complex structure. 
Observe that \[
(\psi_2)_*( -2\pi K \rho \partial_\rho )= \partial_r.\]
We also have $\Theta( J_E \rho \partial_\rho ) = \frac{1}{2\pi}$. 
Hence, $\Theta( J_Y \partial_r ) = -K$. 
It follows then that $J_Y \partial_r = R$. 
Furthermore, $\psi_2$ commutes with the projection maps $\pi \colon E
\to \Sigma$ and $\pi_\Sigma \colon Y \to \Sigma$. In particular then, this
gives the translation invariance of $J_Y|_{\ker \alpha}$ and that $J_Y$
preserves the contact structure. 
This verifies that $J_Y$ is a cylindrical almost complex structure on $\R
\times Y$, 
adapted to the contact form $\alpha$.
Notice that 
$J_Y$ is both $\R$--translation
invariant and Reeb-flow invariant since these two actions correspond to 
the $\C^*$ action on 
the Hermitian line bundle $E$. Even though $\psi_2$ is not a symplectomorphism,
$J_Y$ is compatible with the symplectic form $d(\e^r \alpha)$.
\end{remark}

We now want to define an almost complex structure $J_W$ on $W$ that agrees
with $J_X$ on a compact subset and is cylindrical in the complement of a
slightly larger compact subset. To achieve this, we will use an interpolation
between the complex structures $J_1$ and $J_Y$.

Fix a small number $0 < \epsilon \le \rho_0^2$ such that $(-\epsilon,\epsilon)\times Y\subset W$. 
The interpolating region will be contained in the piece of the symplectization given
by 
$(-\epsilon/2,-\epsilon/4)\times Y$.
Let $g\colon (-\infty,0) \to \R$ be a diffeomorphism (whose graph is given in Figure \ref{F:g}), such that 
\begin{equation*}
\begin{cases}
g(r) = r $ in $ (-\infty,-\epsilon/2) \\
g(-\epsilon/4) = -\epsilon/4 \\
g' > 0 \\
g'(r) = -\frac{1}{4\pi K r} $ for $r \geq -\epsilon/4.
\end{cases}
\end{equation*}

\begin{figure} 
\begin{tikzpicture}[domain=0:4]
    \draw[->] (-4.2,0) -- (1,0) node[right] {$x$};
    \draw[->] (0,-4.2) -- (0,2.2) node[above] {$g(x)$};
    \draw[line width=1.5pt, color=black] (-4.2,-4.2) -- (-2.8,-2.8);
    \draw[line width=1.5pt,domain=-1.7:-.001,smooth,variable=\x,black] plot ({\x},{-0.5*ln(-\x)+0.5*ln(1.5) - 1.5});
    \draw[line width=1.5pt, color=black] (-2.8,-2.8)  ..  controls(-2.2,-2.2) and (-2.2,-1.733333) .. (-1.7,-1.5626);
    \draw [dashed, color=black] (-3,0.1) -- (-3,-3);
    \draw [dashed, color=black] (-3,-3) -- (0.1,-3);
    \draw [dashed, color=black] (-1.5,0.1) -- (-1.5,-1.5);
    \draw [dashed, color=black] (-1.5,-1.5) -- (0.1,-1.5);
    \node at (-3,.3) {$-\epsilon/2$};
    \node at (.5,-3) {$-\epsilon/2$};
    \node at (-1.5,.3) {$-\epsilon/4$};
    \node at (.5,-1.5) {$-\epsilon/4$};
\end{tikzpicture}
\caption{The function $g$}
\label{F:g}
\end{figure}
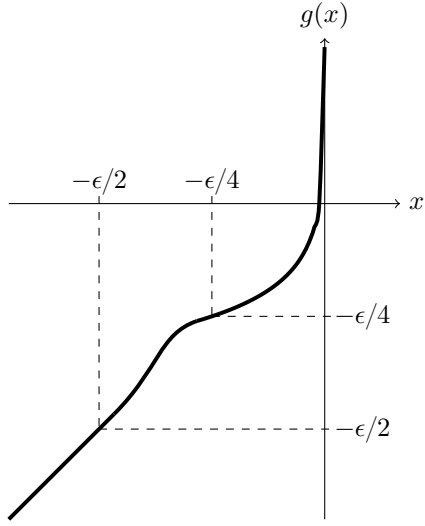

Let $G \colon (-\infty,0)\times Y\to \R\times Y$ be the diffeomorphism $G(r,x)= (g(r),x)$. 
Define the almost complex structure $J_W$ on $W$ as follows:
\begin{equation*}
J_W := 
\begin{cases}
J_X & \text{ on } W\setminus \big([-\epsilon/2,\infty)\times Y\big) \\
G_*J_1 & \text{ on } [-\epsilon/2,-\epsilon/4)\times Y \\
J_Y & \text{ on } [-\epsilon/4,\infty)\times Y
\end{cases}
\end{equation*}

Notice that we may use $G$ to construct a diffeomorphism $\hat G \colon X
\setminus \Sigma \to W$ as follows
\begin{align*}
    &\hat G(x) = x \quad \text{ for } x \in X \setminus \U \\
    &\hat G(\varphi \circ \psi^{-1}(r, y)) = G(r, y) \quad \text{ for } (r, y)
    \in (-\rho_0^2, 0) \times Y.
\end{align*}
Then, it follows immediately that $\hat G_*J_X = J_W$.

The following result compares $J_W$-holomorphic planes with
$J_X$-holomorphic spheres, and will be useful to relate the symplectic
homology differential with relative Gromov--Witten invariants. 

\begin{lemma} \label{planes = spheres}
There is a diffeomorphism $\psi \colon W \to X \setminus \Sigma$ such that $\psi_*J_W = J_X$. This map defines a bijection between $J_W$-holomorphic planes in $W$ asymptotic to Reeb orbits of multiplicity $k$ in $Y$ and $J_X$-holomorphic spheres in $X$ intersecting $\Sigma$ at precisely one point, with order of contact $k$.
\end{lemma}
\begin{proof}
Define $\psi$ to be the inverse of the diffeomorphism $\hat G$ defined above. 
It follows from our definitions that $\psi_*J_W = J_X$. 

The $J_W$-plane then corresponds to a punctured sphere in $X$ with finite
energy. Applying Gromov's removal of singularities theorem and observing that
the multiplicity of the Reeb orbit corresponds to windings around $\Sigma$ of
a loop near the puncture. This then gives the claimed order of contact.
\end{proof}

Having defined the class of admissible almost complex structures on $W$,
cylindrical and Reeb--invariant at infinity, we now also describe how we stretch the
neck of $W$ along $\{ 0 \} \times Y$ 
(in the sense of SFT compactness \cite{SFTcompactness}).

\begin{lemma} \label{L:stretch}
	Suppose that $\Sigma$ is a symplectic hyperplane section of $(X, \omega)$
	and let $J_X$ be an admissible almost complex structure on 
	$X$. 

	Then, there exists a 1-parametric family of almost complex structures
	$J_\kappa$ on the completion $W$ of $X \setminus \Sigma$, which are
	cylindrical outside a compact set, so that $(W, J_\kappa)$ converge
	in the SFT sense as $\kappa\to \infty$ to a split almost complex manifold whose 
	bottom level is $(W, J_W)$ and whose upper level is $(\R \times Y,J_Y)$.
\end{lemma}

\begin{proof}

We explicitly define a family $J_\kappa$, for $\kappa>>0$, of almost complex
structures on $W$. Let $J_W$ be the almost complex structure corresponding to
$J_X$ by the construction above, for fixed diffeomorphism $g \colon (-\infty, 0)
\to \R$. 

For each $\kappa \geq \epsilon/4$ take a diffeomorphism 
$$f_\kappa: [-\epsilon/4,\epsilon/4] \to [-\kappa,\kappa]$$
such that $f_\kappa' \equiv 1$ in a neighborhood of $\pm \epsilon/4$ and $f_\kappa'> 0$.
Define the diffeomorphisms 
\begin{align*}
F_\kappa: [-\epsilon/4,\epsilon/4]\times Y &\to [-\kappa,\kappa]\times Y \\
(r,x) &\mapsto (f_\kappa(r),x)
\end{align*}
Let 
\begin{equation}
J_\kappa := 
\begin{cases}
J_W & \text{ on } W\setminus \big([-\epsilon/4,\infty)\times Y\big) \\
(F_\kappa)^*J_Y & \text{ on } [-\epsilon/4,\epsilon/4]\times Y \\
J_Y & \text{ on } (\epsilon/4,\infty)\times Y
\end{cases}
\label{J_kappa}
\end{equation}
See Figure \ref{Jkappa_fig}. Note that if $\kappa=\epsilon/4$ and $f_\kappa(r) = r$, then $J_{\epsilon/4} = J_W$.

\begin{figure}%
  \begin{center}
    \def\svgwidth{0.5\textwidth}

\begingroup
  \makeatletter
  \providecommand\color[2][]{%
    \errmessage{(Inkscape) Color is used for the text in Inkscape, but the package 'color.sty' is not loaded}
    \renewcommand\color[2][]{}%
  }
  \providecommand\transparent[1]{%
    \errmessage{(Inkscape) Transparency is used (non-zero) for the text in Inkscape, but the package 'transparent.sty' is not loaded}
    \renewcommand\transparent[1]{}%
  }
  \providecommand\rotatebox[2]{#2}
  \ifx\svgwidth\undefined
    \setlength{\unitlength}{294.97670279pt}
  \else
    \setlength{\unitlength}{\svgwidth}
  \fi
  \global\let\svgwidth\undefined
  \makeatother
  \begin{picture}(1,1.02731484)%
    \put(0,0){\includegraphics[width=\unitlength]{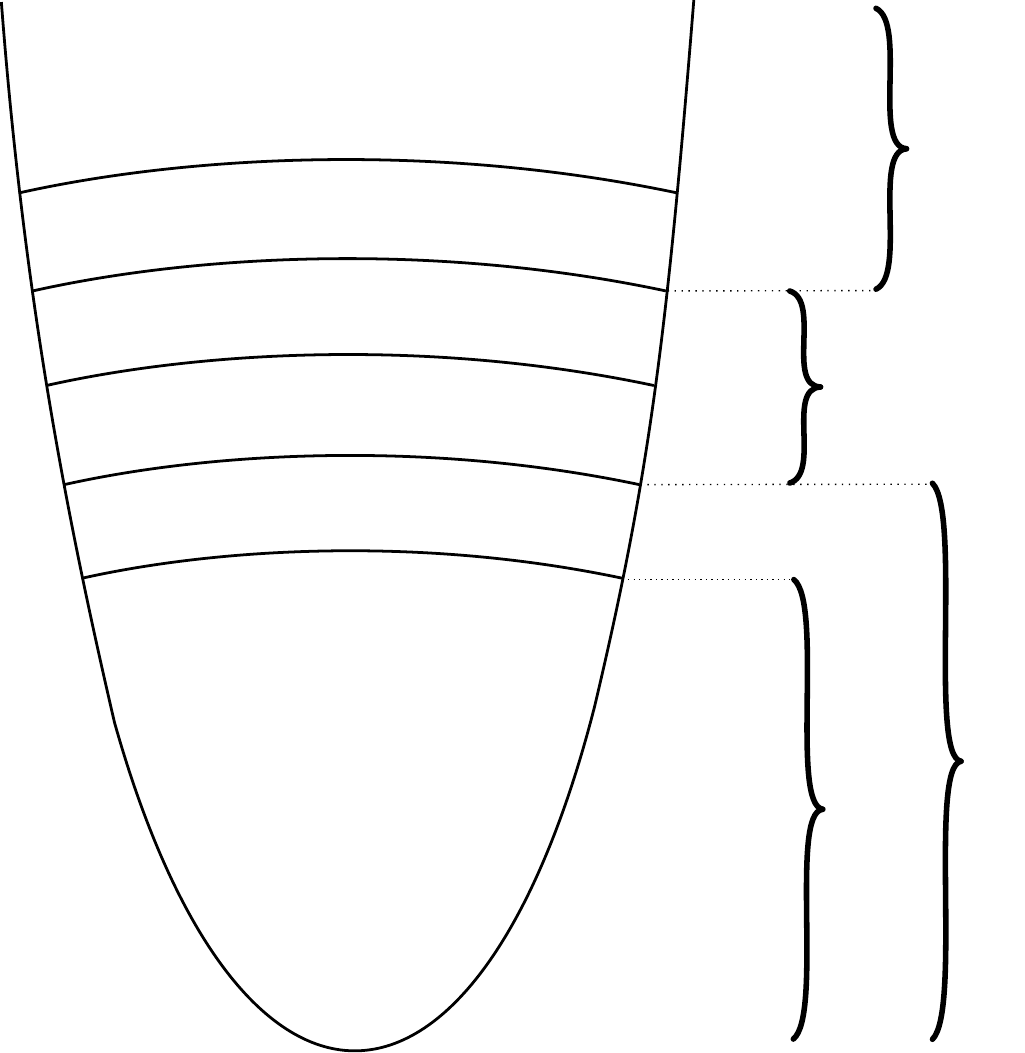}}%
    \put(0.62996209,0.47386319){\color[rgb]{0,0,0}\makebox(0,0)[lb]{\smash{$-\epsilon/2$}}}%
    \put(0.63921907,0.56346689){\color[rgb]{0,0,0}\makebox(0,0)[lb]{\smash{$-\epsilon/4$}}}%
    \put(0.66193603,0.64250162){\color[rgb]{0,0,0}\makebox(0,0)[lb]{\smash{$0$}}}%
    \put(0.67381388,0.75380551){\color[rgb]{0,0,0}\makebox(0,0)[lb]{\smash{$\epsilon/4$}}}%
    \put(0.67735345,0.83118533){\color[rgb]{0,0,0}\makebox(0,0)[lb]{\smash{$\epsilon/2$}}}%
    \put(0.90781135,0.87608029){\color[rgb]{0,0,0}\makebox(0,0)[lb]{\smash{$J_Y$}}}%
    \put(0.95759067,0.27313536){\color[rgb]{0,0,0}\makebox(0,0)[lb]{\smash{$J_W$}}}%
    \put(0.81830753,0.22993992){\color[rgb]{0,0,0}\makebox(0,0)[lb]{\smash{$J_X$}}}%
    \put(0.82329991,0.63681824){\color[rgb]{0,0,0}\makebox(0,0)[lb]{\smash{$(F_\kappa)^*J_Y$}}}%
  \end{picture}%
\endgroup

   \end{center}
  \caption[The almost complex structure $J_\kappa$]{The almost complex structure $J_\kappa$ in
      \protect{\eqref{J_kappa}}}
  \label{Jkappa_fig}
\end{figure}
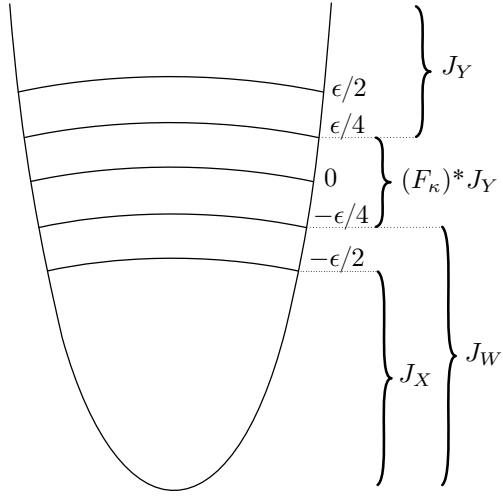

By SFT compactness \cite{SFTcompactness}, a sequence of finite energy $J_{\kappa_i}$-holomorphic curves $u_{\kappa_i}$ in $W$, with $\kappa_i \to \infty$, has a subsequence converging to a pseudoholomorphic building with levels mapping to either $(\R \times Y, J_Y)$ or $(W,J_W)$. 
\end{proof}

Recall that the symplectic homology of $W$ is independent of choice of almost complex structure. 
In our context, it will be useful to use $J_\kappa$, for $\kappa$ large enough.

\part{Moduli spaces and symplectic chain complex}

\section{The chain complex} \label{S:chain complex}

We will describe two chain complexes associated to $W$ whose generators are (essentially) the same, but for which the differentials are \textit{a priori} different.
We first define the group underlying these chain complexes.

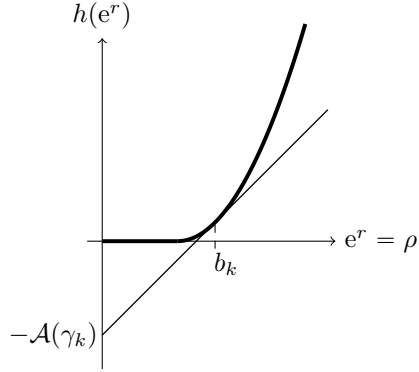
\begin{figure} 
\begin{tikzpicture}[domain=0:6]
    \draw[->] (-0.2,0) -- (3.1,0) node[right] {$\e^r = \rho$};
    \draw[->] (0,-1.7) -- (0,2.7) node[above] {$h(\e^r)$};%
    \draw[line width=1.5pt, color=black] (0,0) -- (1,0);
    \draw[line width=1.5pt,domain=1:2.7,smooth,variable=\x,black] plot ({\x},{(\x-1)*(\x-1)});
    \draw[line width=0.5pt, color=black] (0,-1.25) -- (3,1.75);
    \draw [dashed, color=black] (1.5,-0.1) -- (1.5,0.25);
    \node at (1.65,-.3) {$b_k$};
    \node at (-.65,-1.25) {$-\A(\gamma_k)$};
\end{tikzpicture}
\caption{An admissible Hamiltonian and the graphical procedure for computing the action of a periodic orbit}
\label{F:Jshaped}
\end{figure}

\begin{definition} \label{def:J_shaped}
Consider a function $h \colon (0, +\infty) \to \R$ with the following properties:
\begin{enumerate}[(i)]
\item $h(\rho) = 0$ for $\rho \le 2$;
\item $h'(\rho) > 0$ for $\rho > 2$;
\item $h'(\rho) \to +\infty$ as $\rho \to \infty$;
\item $h''(\rho) > 0$ for $\rho > 2$; 
\end{enumerate}

From this, we obtain a Hamiltonian function $H$ on either $\R \times
Y$ or on $W$ by setting $H(r, y) = h(\e^r)$ on $\R^+ \times Y$ and
extending it by $0$ elsewhere. We will refer to such  Hamiltonians as
\defin{admissible}. See Figure \ref{F:Jshaped}.

\end{definition}

Since $\omega = d(\e^r \alpha)$ on $\R_+\times Y$, the Hamiltonian vector field associated to $H$ is $h'(\e^r)R$, where $R$ is the Reeb vector
field associated to $\alpha$. Recall that $(Y,\alpha)$ is a prequantization bundle over $(\Sigma,\omega_\Sigma)$, and that the corresponding periodic Reeb orbits correspond to covers of the $S^1$-fibres of $Y\to \Sigma$. The periods of these orbits are positive integer multiples of $\frac{1}{K}$, giving the multiplicities of the covers.  The 1-periodic orbits of $H$ are thus of two types:
\begin{enumerate}
 \item constant orbits: one for each point in 
 \[
     W_0 \coloneqq \{w\in W \,|\, (dH)_w=0\} = \overline{W \setminus \supp (dH)};
 \]
 \item non-constant orbits: for each $k\in \Z_+$, there is a $Y$-family of 1-periodic $X_H$-orbits, contained in the level set $Y_k
 \coloneqq \{b_k\}\times Y$, for the unique $b_k>\log 2$ such that $h'(\e^{b_k}) = kK$. Each point in $Y_k$ is the starting point of one such orbit.
\end{enumerate}

\begin{remark}
It is important to observe that these Hamiltonians are Morse--Bott
non-degenerate \textit{except} at the boundary $\partial W_0 = \partial \supp (dH)$. In Section \ref{S:compactness}, we prove that the moduli spaces of curves we consider will not interact with these degenerate, constant orbits.

Recall that a family of periodic Hamiltonian orbits for a time-dependent Hamiltonian vector field
is said to be Morse--Bott non-degenerate if the parametrized 1-periodic orbits form a manifold, 
and the tangent space of the family of orbits at a point is given by the eigenspace of $1$ for the corresponding Poincar\'e return map.
(Morse non-degeneracy requires the return map not to have 1 as an eigenvalue
and hence these periodic orbits must be isolated.)
\end{remark}

We also fix some auxiliary data, consisting of Morse functions and
vector fields.  Fix throughout a Morse function $f_\Sigma \colon
\Sigma \to \R$ and a gradient-like vector field $Z_\Sigma\in\mathfrak X(\Sigma)$, which means that 
$\frac{1}{c} |df_\Sigma|^2 \le df_\Sigma(Z_\Sigma)\le c |df_\Sigma|^2$ for some constant $c>0$.
Denote the time-$t$ flow
of $Z_\Sigma$ by $\varphi^t_{Z_\Sigma}$. 
Given $p\in \Crit (f_\Sigma)$, its stable and unstable manifolds (or ascending and descending manifolds, respectively) are
\begin{equation}\label{(un)stable}
W^s_{\Sigma}(p) := \left\{ q\in \Sigma | \lim_{t\to \infty} \varphi_{Z_\Sigma}^{-t}(q) = p \right\}
 \text{,} \,
W^u_{\Sigma}(p) := \left\{ q\in \Sigma | \lim_{t\to -\infty} \varphi_{Z_\Sigma}^{-t}(q) = p \right\}.
\end{equation}
Notice the sign of time in the flow. 
We further require that $(f_\Sigma,Z_\Sigma)$ be a Morse--Smale pair, i.e.~that all stable and unstable
manifolds of $Z_\Sigma$ intersect transversally.

The contact distribution $\xi$ defines an Ehresmann connection on the
circle bundle $S^1 \hookrightarrow Y\stackrel{\pi_\Sigma}{\to} \Sigma$ (by an abuse of notation, we will sometimes 
also denote the projection $\R\times Y\to \Sigma$ by $\pi_\Sigma$). Denote the horizontal
lift of $Z_\Sigma$ by $\pi_\Sigma^*Z_\Sigma\in \mathfrak X(Y)$.  We fix
a Morse function $f_Y \colon Y \to \R$ %
and a gradient-like vector field $Z_Y\in \mathfrak
X(Y)$ such that $(f_Y,Z_Y)$ is a Morse--Smale pair and the vector field
$Z_Y-\pi_\Sigma^*Z_\Sigma$ is vertical (tangent to the $S^1$-fibres).
Under these assumptions, flow lines of $Z_Y$ project under $\pi_\Sigma$
to flow lines of $Z_\Sigma$.

Observe that critical points of $f_Y$ must lie in the fibres above the critical points 
of $f_\Sigma$ (and these are zeros of $Z_Y$ and $Z_\Sigma$ respectively). 
For notational simplicity, we suppose that $f_Y$ has two critical points in each
fibre. 
In the following, given a critical point for $f_\Sigma$, $p \in \Sigma$, we denote the two critical points in the fibre above $p$ by $\widehat p$ and $\widecheck p$, the fibrewise maximum and fibrewise minimum of $f_Y$, respectively.

We will denote by $M(p)$ the Morse index of a critical point $p \in \Sigma$ of $f_\Sigma$, and by $\tilde M (\tilde p) = M(p) + i(\tilde p)$ the Morse index of the critical point $\tilde p = \widehat p$ or $\tilde p = \widecheck p$ of $f_Y$. The fibrewise index has 
$i(\widehat p) = 1$ and $i(\widecheck p) = 0$.

Since $H$ is admissible, we can identify $W \setminus \overline W$ with
$(\log 2,\infty) \times Y$. Fix also a Morse function $f_W$ and a gradient-like vector field $Z_W$ on $W$, such
that $(f_W,Z_W)$ is a Morse--Smale pair and $Z_W$ restricted to $[-\epsilon/4,\infty)\times Y$ is the constant
vector field $\partial_r$, where $r$ is the coordinate function on the
first factor and $\epsilon>0$ is as in Section \ref{sec:almost complex structures}. 
We denote also by $(f_W,Z_W)$ the Morse--Smale pair that is induced on $X\setminus\Sigma$ by the diffeomorphism 
in Lemma \ref{planes = spheres}.
Denote by $M(x)$ the Morse index of $x\in \Crit(f_W)$ with respect to $f_W$.

We now define the Morse--Bott symplectic chain complex of $W$ and
$H$. Recall that for every $k>0$, each point in $Y_k \coloneqq \{ b_k \} \times Y \subset \R^+ \times Y$ is the starting point of a 
1-periodic orbit of $X_H$, which covers $k$ times the simple Reeb orbit through the corresponding point in $Y$. 
For each critical
point $\tilde p = \widehat p$ or $\tilde p = \widecheck p$ of $f_Y$,
there is a generator corresponding to the pair $(k, \tilde p)$. We will
denote this generator by $\tilde p_k$.  The complex is then given by:
\begin{equation} \label{eqn:chain complex}
SC_*(W,H) = \left(\bigoplus_{k>0} \,
        \bigoplus_{p\in \text{Crit}(f_\Sigma)}\Z
            \langle \widecheck p_k, \widehat p_k \rangle\right) 
        \oplus \left(\bigoplus_{x\in \text{Crit}(f_W)}\Z\langle x \rangle\right) 
\end{equation}

Recall that $\lambda$ is a Liouville form on $W$. 

\begin{definition} \label{def:symplectic action}
The \defin{Hamiltonian action} of a loop $\gamma \colon S^1 \to W$ is 
$$\A( \gamma ) = \int \gamma^*( \lambda - H dt).$$ 
\end{definition}

In particular, for any constant orbit $\gamma \in \overline W$, $\A(\gamma ) = 0$ and for any orbit $\gamma_k \in Y_k$, we have 
\begin{equation} \label{E:HamiltonianAction}
\A( \gamma_k) = \e^{b_k} h'(\e^{b_k}) - h( \e^{b_k})>0,
\end{equation}
where $b_k$ is as above. %
The action of $\gamma_k$ is the negative of the $y$-intersept of the tangent line to the graph of $h$ at $\e^{b_k}$. See Figure \ref{F:Jshaped}. Note that the convexity of $h$ implies that $\A(\gamma_k)$ is monotone increasing in $k$.

We will now define the gradings of the generators. For more details on these formulas, see \cite{DiogoLisiSplit}*{Section 3.1}. 
For a critical point $\widetilde p$ of $f_Y$, 
and a multiplicity $k$, we define
\begin{equation} \label{eqn:grading Y}
| \widetilde p_k | = \tilde M( \widetilde p) + 1-n + 2\frac{\tau_X - K}{K} k \in \R, 
\end{equation}
where we recall that $\tau_X$ is the monotonicity constant of $X$ and $c_1(N\Sigma) = [K \omega_\Sigma]$.

For a critical point $x$ of $f_W$, we define
\begin{equation} \label{eqn:grading W}
|x| = n - M(x).
\end{equation}

\section{Floer moduli spaces before and after splitting}

In this section, we discuss the differential in the Morse--Bott symplectic homology of $W$, before and after stretching the neck. 

\subsection{Morse--Bott symplectic homology}

Consider $H \colon W \to \R$ as defined in Section \ref{S:chain complex}, together with the auxiliary data of $(f_\Sigma, Z_\Sigma)$, $(f_Y, Z_Y)$ and $(f_W, Z_W)$.
Fix an almost complex structure $J_\kappa$ as in \eqref{J_kappa} for a large $\kappa>0$.

\begin{definition} \label{D:Floer cylinder}
A map $\tilde v \colon \R \times S^1 \to W$ is a \defin{Floer cylinder} if
\begin{equation}
\partial_s \tilde v + J_\kappa(\partial_t \tilde v - X_H(\tilde v)) = 0.
\label{Floer eq W}
\end{equation}

Let $\moduliFloerCyl(x_-, x_+)$ denote the space of all such Floer cylinders such that, for fixed 1-periodic orbits $x_\pm$ of $X_H$, $\lim_{s \to \pm \infty} \tilde v(s,.) = x_{\pm}$. 

Let $S_-$ and $S_+$ denote connected spaces of 1-periodic $X_H$-orbits (either $\overline W$ or $Y_k$). Then we define
\[
\moduliFloerCyl(S_-, S_+) = \bigcup_{x_- \in S_-, x_+ \in S_+} \moduliFloerCyl(x_-, x_+).
\]

Given a Floer cylinder $\tilde v$, we denote its asymptotic limits by $\tilde
v(+\infty) = x_+$ and $\tilde v(-\infty) = x_-$.
\end{definition}

Note that $\A(\tilde v(s_0,.)) \leq \A(\tilde v(s_1,.))$ if $s_0 \leq s_1$. 

\begin{definition} \label{D:presplit energy}
The {\em energy} of a Floer cylinder $\tilde v \colon \R \times S^1 \to W$ is given by:
\[
E(\tilde v) = \int_{\R \times S^1} \tilde v^*d\lambda - \tilde v^*dH \wedge dt 
    = \int_{\R\times S^1} || \tilde v_s ||^2_{J_\kappa} \,ds \wedge dt,
\]
where $|| \tilde v_s ||^2_{J_\kappa} = d\lambda(\tilde v_s, J_\kappa \tilde v_s)$.
\end{definition}

Since the symplectic form is exact, a Floer cylinder $\tilde v$ with $\tilde v(\pm\infty) = x_\pm$ has
\[
E(\tilde v) = \A(x_+) - \A(x_-) =  \int x_+^*( \lambda - h(\e^r) dt) - \int x_-^*(\lambda - h(\e^r) dt).
\]
A Floer cylinder is non-trivial if $E(\tilde v) > 0$, or equivalently, if
$\tilde v$ is not of the form $(s, t) \mapsto \gamma(t)$ for some 1-periodic $X_H$-orbit $\gamma$. 

Let us recall some convergence properties of finite energy Floer trajectories. 
This is a combination of statements of several theorems from the literature,
with slightly different hypotheses.
In the Morse--Bott case, we refer to
\citelist{\cite{BOSymplecticHomology} \cite{BourgeoisThesis} \cite{HWZ4}}.
In the non-degenerate case, the relevant ideas have appeared in 
\citelist{\cite{RobbinSalamonAsymptotics} \cite{HWZ1}}, and of course the
original work was done by Floer \cite{Floer88}.
\begin{lemma} \label{L:convergence to orbit}
Suppose that $\tilde v\colon \R \times S^1 \to W$ is a finite energy Floer cylinder contained in a compact subset of $W$.
Then, for any sequence $s_k \to +\infty$ [resp.,  $s_k \to -\infty$] there is a subsequence we also denote $(s_k)_{k=1}^\infty$ and a 1-periodic orbit $\gamma(t)$ of the Hamiltonian vector field so that $\tilde v(s_k, t) \to \gamma(t)$ in $C^\infty$.
\begin{enumerate}
\item If $\gamma$ is a Morse--Bott non-degenerate orbit, then the limit $\gamma$ does not depend on the initial choice of sequence, and furthermore, we have the much stronger result that $\tilde v(s,t)$ converges to $\gamma$ exponentially fast in $s$;
\item If $\gamma$ is a non-isolated degenerate orbit, the limit may depend on the initial sequence $(s_k)_{k=1}^\infty$. Any two limits are connected by a family of periodic orbits of the same action.
	\label{2}
\end{enumerate}
\end{lemma}

Furthermore, if the curve converges exponentially fast, 
at a negative [resp., positive] puncture the rate of 
convergence is governed by the smallest positive [largest negative] eigenvalue of the
corresponding asymptotic operator \cite{SiefringAsymptotics}.

\begin{remark} \label{R:not MB not converge}
There are several constructions of examples where the limit is not
unique in the degenerate case. In particular, Siefring has carefully
constructed such examples in the case of pseudoholomorphic curves in
symplectizations \cite{SiefringMultipleLimits}.  Similarly, a gradient
trajectory of a smooth function that is not Morse--Bott can fail to
converge to a single critical point, and may contain sequences converging
to different critical points.

Even if the limit were unique, these examples illustrate also that convergence
will not typically be at an exponential rate, and thus will not provide a good
Fredholm problem.
\end{remark}

\begin{remark}

In the case of a Morse--Bott limit, a cylinder contained in a compact subset of 
$W$ has finite energy if and only if it converges to 1-periodic
Hamiltonian orbits at its punctures, see
\citelist{
    \cite{SalamonFloerHomology}*{Proposition 1.21}
    \cite{SFTcompactness}*{Proposition 5.13}
}.

 Yoel Groman pointed out to us that the assumption that $\tilde v$
 has finite energy already implies that its image is contained in
 a compact subset of $W$, so the assumptions of Lemmas \ref{L:convergence
 to orbit} and \ref{no + limits in W} can be simplified.  This is
 because the action functional for our Hamiltonians satisfies the
 Palais--Smale condition \cite{GromanOpen}.

\end{remark}

The next result shows that there are no non-trivial Floer cylinders
whose positive asymptote $x_+$ is in $W_0$.  Note that if
such a cylinder has asymptotic limits at $+\infty$ and $-\infty$,
then it is contained in a compact subset of $W$, by the maximum
principle.

\begin{lemma} \label{no + limits in W}
Let $\tilde v\colon \R \times S^1 \to W$ be a solution of \eqref{Floer eq W}
that has finite energy and is contained in a compact subset of $W$. If there is
a sequence $s_k^+\to \infty$ such that the $C^0$-limit $\lim_{k\to \infty}\tilde
v(s_k^+,.) = x_+ \in W_0$, then $\tilde v$ is constant.  
\end{lemma}
\begin{proof}
Take any sequence $s_k^- \to -\infty$. Lemma \ref{L:convergence to orbit} 
implies that there are subsequences (denoted also by $s_k^\pm$) such that
$\tilde v(s_k^\pm,.)$ converge in $C^1$ (actually $C^\infty$) to 1-periodic
$X_H$-orbits $x_\pm$, with $x_+\in W_0$. Then,
\begin{align*}
0 \leq E(\tilde v) &= \lim_{k\to \infty} E(\tilde v|_{[s_k^-,s_k^+]\times S^1}) = \lim_{k\to \infty} \big(\A(\tilde v(s_k^+,.)) - \A(\tilde v(s_k^-,.)) \big) = \\
&= \A(x_+) - \A(x_-) .
\end{align*}
Since $\A(x_+) = 0$ and $\A(\gamma)\geq 0$ for every 1-periodic $X_H$-orbit, we conclude that $E(\tilde v) = 0$. This implies the result. 
\end{proof}

The following is the main definition of this section. 

\begin{definition} \label{def:FloerCylinderWithCascades}
Fix $N \ge 0$. Let $S_0, S_1, \dots, S_{N}$ be a collection of connected spaces
of orbits, which can either be one of the $Y_k$ or $W_0$. 
Let $(f_i, Z_i)$, $i=0, \dots, N+1$ be the pair of Morse function and 
gradient-like vector field of $f_i = f_Y$, $Z_i = Z_Y$ if $S_i = Y_k$ for some
$k$, and $f_i = -f_W$, $Z_i=-Z_W$ if $S_i = W_0$.

Let $x$ be a critical point of $f_0$ and $y$ a critical point of $f_N$ (so $x$ and $y$ are generators of the chain complex (\ref{eqn:chain complex})). 

A {\em Floer cylinder with $0$ cascades} ($N = 0$), with positive end at $y$ and negative end at $x$, consists of a 
positive gradient trajectory $\nu \colon \R \to S_0$, such that $\nu(-\infty) = x$, $\nu(+\infty) = y$ and $\dot \nu =  Z_0(\nu)$.

A {\em Floer cylinder with $N$ cascades}, $N \ge 1$, with positive end at $y$ and negative end at $x$, consists of the following data:
\begin{enumerate}
\renewcommand{\theenumi}{\roman{enumi}}
\item $N-1$ length parameters $l_i > 0, i=1, \dots, N-1$;
\item Two half-infinite gradient trajectories, $\nu_0 \colon (-\infty, 0] \to S_0$ and $\nu_{N} \colon [0, +\infty) \to S_{N}$
with $\nu_0(-\infty) = x$, $\nu_N(+\infty) = y$ and $\dot \nu_i = Z_i(\nu_i)$ for $i=0$ or $N$;
\item $N-1$ gradient trajectories $\nu_i$ defined on intervals of length $l_i$, $\nu_i \colon [0, l_i] \to S_i$ for $i=1, \dots, N-1$ 
such that $\dot \nu_i = Z_i(\nu_i)$;
\item $N$ non-trivial Floer cylinders $\tilde v_i \colon \R \times S^1 \to W$, 
    $i = 1, \dots, N$, satisfying equation (\ref{Floer eq W})
	where (defining $l_0 = 0$)
	\[
            \tilde v_i(+\infty, \cdot) = \nu_i(0) \in S_i, 
                \quad \tilde v_i(-\infty, \cdot) = \nu_{i-1}(l_{i-1}) \in S_{i-1}.
	\]
\end{enumerate}	
\label{Floer cascade}
In the case of a Floer cylinder with $N \ge 1$ cascades, we refer to the 
non-trivial Floer cylinders $\tilde v_i$ as {\em sublevels}. 
See Figure \ref{fig:Floer cylinder with cascades} for a schematic illustration.
\end{definition}

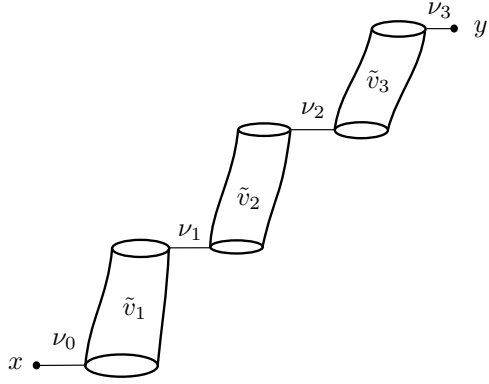
\begin{figure}
  \begin{center}
    \def\svgwidth{0.5\textwidth}

\begingroup
  \makeatletter
  \providecommand\color[2][]{%
    \errmessage{(Inkscape) Color is used for the text in Inkscape, but the package 'color.sty' is not loaded}
    \renewcommand\color[2][]{}%
  }
  \providecommand\transparent[1]{%
    \errmessage{(Inkscape) Transparency is used (non-zero) for the text in Inkscape, but the package 'transparent.sty' is not loaded}
    \renewcommand\transparent[1]{}%
  }
  \providecommand\rotatebox[2]{#2}
  \ifx\svgwidth\undefined
    \setlength{\unitlength}{323.78738945pt}
  \else
    \setlength{\unitlength}{\svgwidth}
  \fi
  \global\let\svgwidth\undefined
  \makeatother
  \begin{picture}(1,0.79842498)%
    \put(0,0){\includegraphics[width=\unitlength]{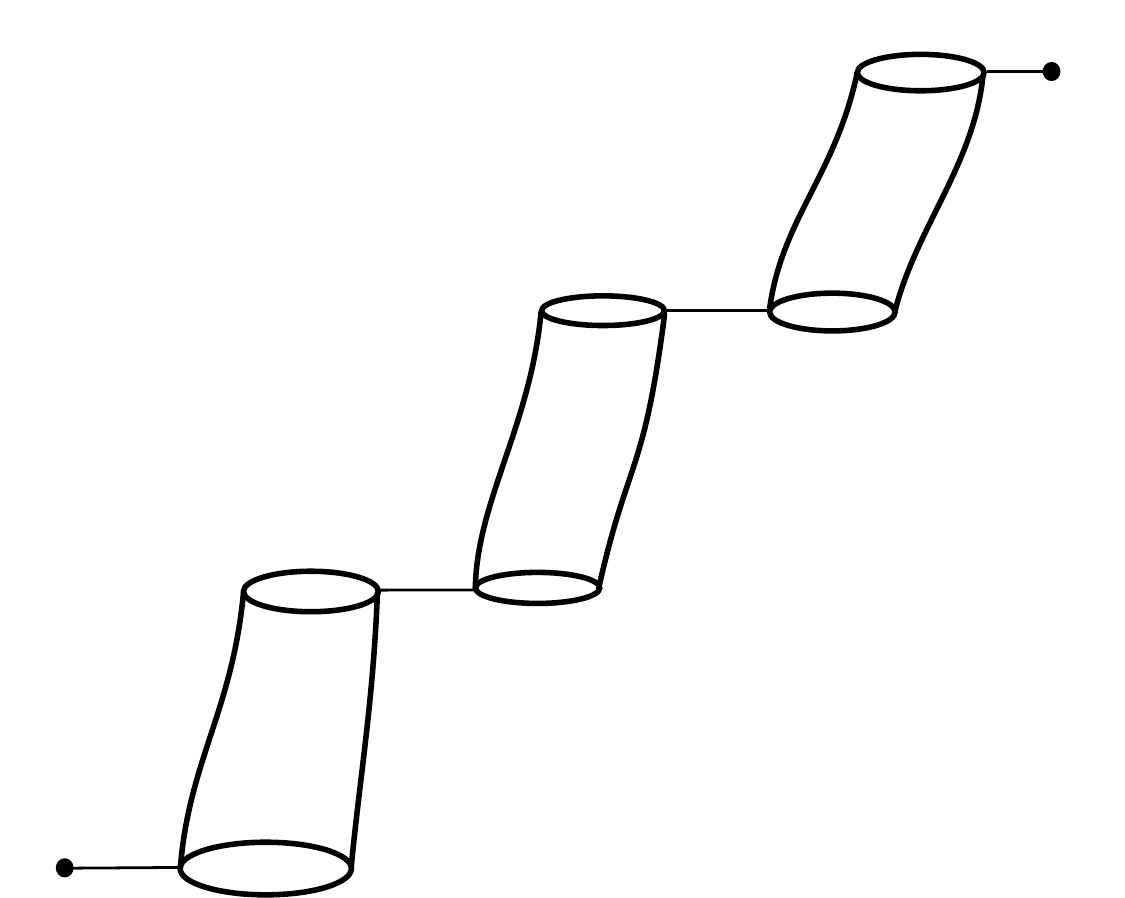}}%
    \put(-0.00164074,0.01853505){\color[rgb]{0,0,0}\makebox(0,0)[lb]{\smash{$x$}}}%
    \put(0.97698142,0.72810202){\color[rgb]{0,0,0}\makebox(0,0)[lb]{\smash{$y$}}}%
    \put(0.23839863,0.13196028){\color[rgb]{0,0,0}\makebox(0,0)[lb]{\smash{$\tilde v_1$}}}%
    \put(0.47843802,0.36408628){\color[rgb]{0,0,0}\makebox(0,0)[lb]{\smash{$\tilde v_2$}}}%
    \put(0.75013097,0.60940121){\color[rgb]{0,0,0}\makebox(0,0)[lb]{\smash{$\tilde v_3$}}}%
    \put(0.09068209,0.06865309){\color[rgb]{0,0,0}\makebox(0,0)[lb]{\smash{$\nu_0$}}}%
    \put(0.35446167,0.29814131){\color[rgb]{0,0,0}\makebox(0,0)[lb]{\smash{$\nu_1$}}}%
    \put(0.61296557,0.55400752){\color[rgb]{0,0,0}\makebox(0,0)[lb]{\smash{$\nu_2$}}}%
    \put(0.87938298,0.76503116){\color[rgb]{0,0,0}\makebox(0,0)[lb]{\smash{$\nu_3$}}}%
  \end{picture}%
\endgroup
   \end{center}
  \caption{A Floer cylinder with 3 cascades, as in Definition \protect{\ref{def:FloerCylinderWithCascades}}, with positive end at $p$ and negative end at $q$.} 
   \label{fig:Floer cylinder with cascades}
\end{figure}

The energy of a Floer cylinder with $N$ cascades is the sum of the energies of each of its $N$ sublevels. 

\begin{lemma} 
Floer cylinders with cascades do not contain sublevels $\tilde v_i$ such 
that $\lim_{s\to+\infty} \tilde v_i(s,.) \in W_0$. If $\lim_{s\to-\infty} \tilde
v_i(s,.) = x_-\in W_0$, then $i=1$ and $x_- \notin \partial W_0$. 
\label{non-MB}
\end{lemma}
\begin{proof}
The case $s\to +\infty$ follows from Lemma \ref{no + limits in W}.  
For the case $s\to -\infty$, observe that $\nu_{i-1}$ is a negative flow line of 
$Z_W$, which agrees with $\partial_r$ on $[-\epsilon/4,\infty) \times Y$ (in
particular on $\partial W_0$). 
If $i>0$, then the sublevel $\tilde v_{i-1}$ must be such that
$\lim_{s\to+\infty} \tilde v_{i-1}(s,.) \in W_0$, which as we saw is impossible. 
If $i=1$, then $\nu_0(-\infty) \in \Crit(f_W)$, and all critical points of $f_W$
are ``below $\partial W_0$''. Therefore, $x_-\notin \partial W_0$. 
\end{proof}

\begin{remark} \label{R:not MB}
The points in $\partial W_0 = \partial \supp dH$, which were ruled out as asymptotic orbits of sublevels 
by the previous Lemma, correspond to constant orbits of $X_H$ that are not of Morse--Bott type. 
Convergence to orbits that are not Morse--Bott non-degenerate would introduce
severe analytical difficulties
(as discussed in Remark \ref{R:not MB not converge}).
\end{remark}

\begin{lemma}
Let $S_0, S_1, \dots, S_N$ be the spaces of orbits of a Floer cylinder with $N$
cascades. Let $k(S_i) = k$ if $S_i = Y_k$ and let $k(S_i) = 0$ if $S_i = W_0$. 
Then, $k(S_i)$ is monotone increasing in $i$.  
\label{monotone multiplicities}
\end{lemma}
\begin{proof}
This follows immediately from the fact that $0 < E(\tilde v_i) = \A( \tilde v_i(+\infty)) - \A( \tilde v_i(-\infty))$ and from the monotonicity of $\A$ in the multiplicity $k$. 
\end{proof}

\begin{remark}
	These lemmas have a number of important consequences. In particular, 
a Floer cylinder with cascades between two critical points of $f_W$ 
must have 0 cascades, which is to say that it consists of a flow line of $-Z_W$. 

Furthermore, a Floer cylinder with cascades whose positive end is on a $Y_k$, $k
\ge 1$, and whose negative end is on $W_0$ must have the first sublevel
connecting an orbit in a $Y_l$ to a (constant) orbit in the interior of $W_0$. 
This is the unique sub-level in the cascade that converges to a point in $W_0$.
\label{0 cascades W}
\end{remark}

We are now able to define a differential on the chain complex (\ref{eqn:chain complex}).
Given generators $x, y$, denote by
$$
\M_{H,N}(x,y;J_\kappa)
$$
the space of Floer cylinders with $N$ cascades from $x$ to
$y$ (i.e. with negative end at $x$ and positive end at $y$).

We then define
\begin{equation} \label{eqn:differential}
    \partial_{\text{pre}} \, y = \sum_{|x|=|y| - 1} \# \left( \M_{H, 0}(x,y; J_\kappa) / \R \right ) x
            + \sum_{|x| = |y| - 1} \sum_{N = 1}^\infty \# \left ( \M_{H, N}(x,y; J_\kappa) / \R^N
            \right ) x.
\end{equation}

Observe that if $N = 0$, the Floer cylinder with $0$ cascades is a flow line, and the $\R$ action is given by its standard reparametrization. If, instead, $N \ge 1$, the $\R^{N}$ action is given by domain
translation on the $N$ sublevels.

We call $\partial_{\text{pre}}$ the {\em presplit} Floer differential, to distinguish it from the {\em split} differential that will be discussed next. 

\begin{remark}
To give a complete definition of $\partial_{\text{pre}}$, we should
discuss the orientations of the moduli spaces involved and the associated
signs. Instead, we will relate the presplit Floer differential with the
split Floer differential, and refer to the discussion of orientations in
\cite{DiogoLisiSplit}*{Section 7}. Orientations for the split differential
will be further studied in Section \ref{compute signs} below.
\end{remark}

\subsection{Morse--Bott split symplectic homology}

\label{cascades after splitting}

Recall the definition of the almost complex structures $J_\kappa$, $J_Y$
and $J_W$ in Section \ref{sec:almost complex structures}. Given a sequence
$\tilde v_{\kappa_n}$ of finite energy $J_{\kappa_n}$-Floer cylinders in
$W$, with $\kappa_n\to \infty$, SFT compactness \cite{SFTcompactness}
implies that there is a subsequence converging to an SFT building.
Observe that the Hamiltonian $H$ is supported in $(\log 2, \infty)\times
Y$, which is above the hypersurface $\{0\}\times Y$ along which we stretch
the neck. This makes it possible to apply the usual SFT compactness
argument.  A gluing argument gives that a transverse SFT-type building
can be glued to obtain a ``presplit'' Floer cylinder, as in the previous
section.

This Hamiltonian induces a function on $\R\times Y$, vanishing on
$\R_-\times Y$, which we also denote by $H$.  The non-constant orbits
again form manifolds $Y_k$, where $k\geq 1$. The critical points of the
Morse function $f_W \colon W\to \R$ are below the neck-stretching region.

This then gives a different description of the differential on the
chain complex (\ref{eqn:chain complex}) (defined above in terms of Floer
cylinders with cascades) by counting {\em split} Floer cylinders with
cascades. In principle, these could become quite complicated. Thanks
to the monotonicity assumptions we impose on $X$ and on $\Sigma$,
the cascades that have the correct Fredholm index to appear in the
differential turn out to be relatively simple. This is analyzed in detail
in \cite{DiogoLisiSplit}, see in particular 
\cite{DiogoLisiSplit}*{Section 6.1}.  
We now provide a brief summary of the cascades that
contribute to the differential in split symplectic homology.

First, we describe the component pieces. The upper level of a cascade in the
differential will consist of a (possibly punctured) Floer cylinder
$\tilde v \colon \R \times S^1 \setminus \Gamma \to \R \times Y$,
where $\Gamma$ is either empty or contains at most one point $P \in \R \times
S^1$. We refer to such a $P$ as an \textit{augmentation puncture}. 
This solves the Floer equation
\begin{equation} \label{Floer eq Y}
\partial_s \tilde v + J_Y (\partial_t \tilde v - X_H(\tilde v)) = 0,
\end{equation}
and has finite \textit{hybrid energy}, given
by
\begin{equation} \label{E:hybridEnergy}
    E( \tilde v ) = \sup \left \{ 
\int_{\R \times S^1} \tilde v \pb \left ( d(  \eta \alpha ) - dH \wedge dt
\right) \, | \, \eta \colon \R \to [0, \infty), \eta' \ge 0, \eta(r) = \e^r
\text{ for } r \ge 0
\right \}.
\end{equation}
Such a cylinder converges to a non-constant closed Hamiltonian orbit at $+\infty$, to a
Reeb orbit in $\{ -\infty \} \times Y$ at the puncture $P$ and, at $-\infty$,
converges either to a non-constant closed Hamiltonian orbit or to a Reeb orbit 
in $\{ -\infty \} \times Y$. 
A cascade may also have a lower level consisting of a $J_W$-holomorphic plane
converging to a Reeb orbit in $\{+\infty \} \times Y$ at its positive puncture.

Notice that by the Reeb invariance of both the Hamiltonian $H$ and of the almost
complex structure $J_Y$, such a (punctured) Floer cylinder projects to a
$J_\Sigma$ holomorphic sphere (with removable singularities at the punctures
$\pm \infty$ and $P$). 

The following result is a restatement of \cite{DiogoLisiSplit}*{Propositions 6.2 and 6.3}. 
It states 
that under our monotonicity assumptions, the cascades contributing to the split Floer differential can be of only four simple types.

\begin{figure}
  \begin{center}
    \def\svgwidth{0.3\textwidth}
\begingroup
  \makeatletter
  \providecommand\color[2][]{%
    \errmessage{(Inkscape) Color is used for the text in Inkscape, but the package 'color.sty' is not loaded}
    \renewcommand\color[2][]{}%
  } 
  \providecommand\transparent[1]{%
    \errmessage{(Inkscape) Transparency is used (non-zero) for the text in Inkscape, but the package 'transparent.sty' is not loaded} 
    \renewcommand\transparent[1]{}%
  } 
  \providecommand\rotatebox[2]{#2} 
  \ifx\svgwidth\undefined 
    \setlength{\unitlength}{153.65167524pt} 
  \else 
    \setlength{\unitlength}{\svgwidth} 
  \fi 
  \global\let\svgwidth\undefined
  \makeatother
  \begin{picture}(1,0.52879344)%
    \put(0,0){\includegraphics[width=\unitlength]{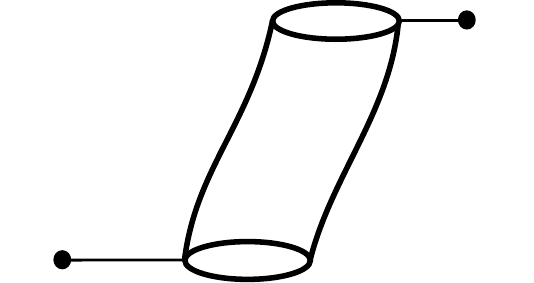}}%
    \put(-0.0634575,0.02310476){\color[rgb]{0,0,0}\makebox(0,0)[lb]{\smash{$\widehat q_{k_-}$}}}%
    \put(0.95149337,0.47648333){\color[rgb]{0,0,0}\makebox(0,0)[lb]{\smash{$\widecheck p_{k_+}$}}}%
    \put(0.52131919,0.24979405){\color[rgb]{0,0,0}\makebox(0,0)[lb]{\smash{$\tilde v$}}}%
    \put(0.902131919,0.164979405){\color[rgb]{0,0,0}\makebox(0,0)[lb]{\smash{$\R\times Y$}}}%
  \end{picture}%
\endgroup

   \end{center}
  \caption[Case 1]{Case 1 in Proposition \protect{\ref{P:split SH cases}}}
  \label{case_1_fig}
\end{figure}

\begin{figure}
  \begin{center}
    \def\svgwidth{0.25\textwidth}

\begingroup
  \makeatletter
  \providecommand\color[2][]{%
    \errmessage{(Inkscape) Color is used for the text in Inkscape, but the package 'color.sty' is not loaded}
    \renewcommand\color[2][]{}%
  }
  \providecommand\transparent[1]{%
    \errmessage{(Inkscape) Transparency is used (non-zero) for the text in Inkscape, but the package 'transparent.sty' is not loaded}
    \renewcommand\transparent[1]{}%
  }
  \providecommand\rotatebox[2]{#2}
  \ifx\svgwidth\undefined
    \setlength{\unitlength}{155.47384928pt}
  \else
    \setlength{\unitlength}{\svgwidth}
  \fi
  \global\let\svgwidth\undefined
  \makeatother
  \begin{picture}(1,1.34900323)%
    \put(0,0){\includegraphics[width=\unitlength]{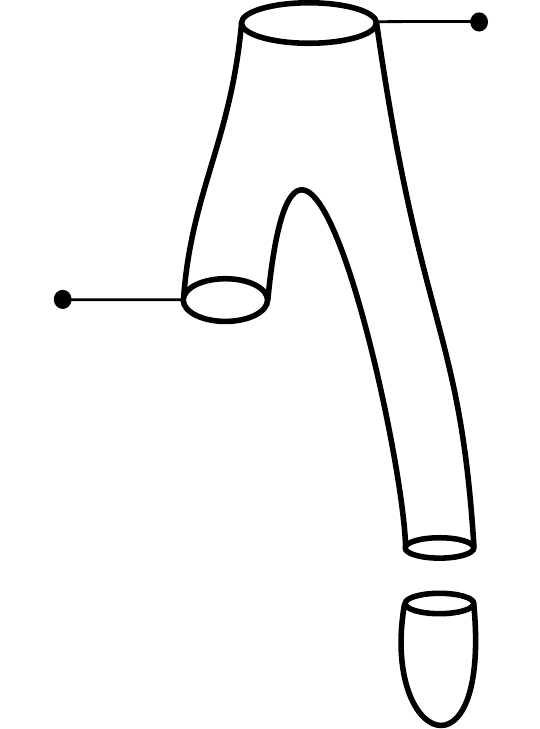}}%
    \put(-0.08341697,0.77649313){\color[rgb]{0,0,0}\makebox(0,0)[lb]{\smash{$\widehat p_{k_-}$}}}%
    \put(0.95206187,1.29539538){\color[rgb]{0,0,0}\makebox(0,0)[lb]{\smash{$\widecheck p_{k_+}$}}}%
    \put(0.54043248,1.11827007){\color[rgb]{0,0,0}\makebox(0,0)[lb]{\smash{$\tilde v$}}}%
    \put(0.6052953,0.09792872){\color[rgb]{0,0,0}\makebox(0,0)[lb]{\smash{$U$}}}%
    \put(0.95206187,0.77649313){\color[rgb]{0,0,0}\makebox(0,0)[lb]{\smash{$\R\times Y$}}}%
    \put(1.04206187,0.097649313){\color[rgb]{0,0,0}\makebox(0,0)[lb]{\smash{$W$}}}%
  \end{picture}%
\endgroup

   \end{center}
  \caption[Case 2]{Case 2 in Proposition \protect{\ref{P:split SH cases}}}
  \label{case_2_fig}
\end{figure}

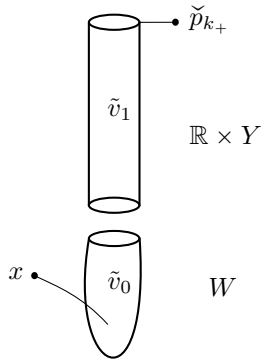
\begin{figure}
  \begin{center}
    \def\svgwidth{0.2\textwidth}

\begingroup
  \makeatletter
  \providecommand\color[2][]{%
    \errmessage{(Inkscape) Color is used for the text in Inkscape, but the package 'color.sty' is not loaded}
    \renewcommand\color[2][]{}%
  }
  \providecommand\transparent[1]{%
    \errmessage{(Inkscape) Transparency is used (non-zero) for the text in Inkscape, but the package 'transparent.sty' is not loaded}
    \renewcommand\transparent[1]{}%
  }
  \providecommand\rotatebox[2]{#2}
  \ifx\svgwidth\undefined
    \setlength{\unitlength}{146.34692076pt}
  \else
    \setlength{\unitlength}{\svgwidth}
  \fi
  \global\let\svgwidth\undefined
  \makeatother
  \begin{picture}(1,1.81416906)%
    \put(0,0){\includegraphics[width=\unitlength]{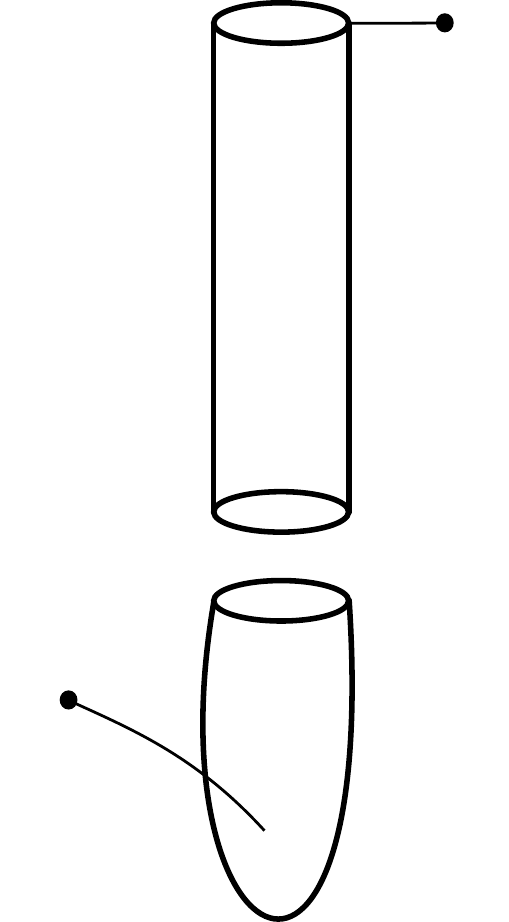}}%
    \put(-0.00144135,0.41693098){\color[rgb]{0,0,0}\makebox(0,0)[lb]{\smash{$x$}}}%
    \put(0.94907221,1.75491512){\color[rgb]{0,0,0}\makebox(0,0)[lb]{\smash{$\widecheck p_{k_+}$}}}%
    \put(0.5192222,1.27663133){\color[rgb]{0,0,0}\makebox(0,0)[lb]{\smash{$\tilde v_1$}}}%
    \put(0.5192222,0.3684972){\color[rgb]{0,0,0}\makebox(0,0)[lb]{\smash{$\tilde v_0$}}}%
    \put(0.94907221,1.12663133){\color[rgb]{0,0,0}\makebox(0,0)[lb]{\smash{$\R\times Y$}}}%
    \put(1.04907221,0.3284972){\color[rgb]{0,0,0}\makebox(0,0)[lb]{\smash{$W$}}}%
  \end{picture}%
\endgroup

   \end{center}
  \caption[Case 3]{Case 3 in Proposition \protect{\ref{P:split SH cases}}} 
  \label{case_3_fig}
\end{figure}

\begin{proposition} \label{P:split SH cases}
In what follows, $q,p\in \Crit(f_\Sigma)$ and $x\in \Crit(f_W)$. There are the following contributions to the 
split symplectic homology differential:
\begin{enumerate}
\item[(0)] An index $1$ gradient trajectory in $Y_k$, connecting $\tilde q_k$ to $\tilde p_k$ or 
an index $1$ gradient trajectory in $W$, connecting two critical points of $f_W$.
\item A half-infinite gradient trajectory $\gamma_- \colon (-\infty, 0] \to
    Y_{k_-}$, with $\gamma_-(-\infty) = \widehat q_{k_-}$, followed by a Floer
    cylinder $\tilde v \colon \R \times S^1 \to \R \times Y$ 
    followed by a
    half-infinite gradient trajectory $\gamma_+ \colon [0, +\infty) \to
    Y_{k_+}$, with $\gamma_+(+\infty) = \widecheck p_{k_+}$. 
    Writing $A \in H_2(\Sigma; \Z)$ for the (spherical) homology class
    represented by the projection of $\tilde v$ to $\Sigma$, these satisfy the
    conditions
    \begin{align*}
        &k_+ - k_- = K \omega(A) > 0 \\
        &\gamma_-(0) = \lim_{s \to -\infty} \tilde v(s, 0) \\
        &\lim_{s \to +\infty} \tilde v(s, 0) = \gamma_+(0).
    \end{align*}
        See Figure \ref{case_1_fig}. 
    \item A half-infinite gradient trajectory $\gamma_- \colon (-\infty, 0] \to
        Y_{k_-}$ with $\gamma_-(-\infty) = \widehat p_{k_-}$, followed by a
        Floer cylinder $\tilde v \colon \R \times S^1 \setminus \{ P \} \to \R
        \times Y$ and one augmentation puncture, with trivial projection to $\Sigma$, and a
$J_W$--holomorphic plane $U \colon \C \to W$, followed by a
half-infinite gradient trajectory $\gamma_+ \colon [0, +\infty) \to Y_{k_+}$
with $\gamma_+(+\infty) = \widecheck p_{k_+}$.  
Writing $B \in H_2(X; \Z)$ for the homology class represented by the plane $U$,
these satisfy
\begin{align*}
    &k_+-k_-= K \omega(B) > 0 \\
    &\gamma_-(0) = \lim_{s \to -\infty} \tilde v(s, 0) \\
    &\lim_{s \to +\infty} \tilde v(s, 0) = \gamma_+(0),
\end{align*}
and the Reeb orbit to which $\tilde v$ converges at $P$ is the same Reeb orbit
to which $U$ converges at $+\infty$.
Note that $\widecheck p, \widehat p \in \Crit(f_Y)$ both project to the same
critical point $p\in \Crit(f_\Sigma)$. 
See Figure \ref{case_2_fig}.
\item A half-infinite gradient trajectory $\gamma_- \colon [0, +\infty) \to W$, with
    $\gamma_-(+\infty) = x$, 
    followed by a pseudoholomorphic plane $\tilde v_0$ in $W$, followed by a
    Floer cylinder $\tilde v_1$ in $\R \times Y$, and a half-infinite gradient
    trajectory $\gamma_+ \colon [0, +\infty) \to Y_{k_+}$ with
    $\gamma_+(+\infty) = \widecheck p_{k_+}$. The Floer cylinder
    $\tilde v_1$ projects to a constant in $\Sigma$.  
    These satisfy
    asymptotic matching conditions as follows: $\gamma(0) = \tilde v_0(0)$, the
    plane $\tilde v_0$ converges at its positive puncture to the same Reeb orbit
    as $\tilde v_1$ converges to at its negative puncture, which has
    multiplicity $k_+$, and $\gamma_+(0) = \lim_{s \to \infty} \tilde v_1(s,
    0)$.
See Figure \ref{case_3_fig}.
\end{enumerate}
\end{proposition}

In Cases 1, 2 and 3, call the (possibly punctured) Floer cylinder in $\R\times Y$ the {\em upper level} of the split Floer cylinder. 
In Case 2, the augmentation plane $U$ is considered modulo its domain
automorphisms. In Case 3, notice that the gradient trajectory is descending from
$x$ to $\tilde v_0(0)$. 

We say that a Floer cylinder is simple if 
$\tilde v \colon \R\times S^1\setminus \Gamma \to \R\times Y$ projects to 
a somewhere injective curve in $\Sigma$, or if it projects to a constant.
Using the notation from \cite{DiogoLisiSplit}, 
we write 
$$
\M^*_{H,l,\R\times Y;k_-,k_+}(A;J_Y)
$$
for the space of simple Floer cylinders $\tilde v \colon \R\times S^1\setminus \Gamma \to \R\times Y$, 
where $\Gamma$ is a set of $l$ augmentation punctures, $\lim_{s\to \pm\infty} \tilde v(s,.)$ are Hamiltonian orbits of 
multiplicity $k_\pm$ and $\pi_\Sigma \circ \tilde v$ represents the homology
class $A\in H_2(\Sigma;\Z)$. 

The spaces of Floer cylinders with cascades in Case 1 are unions of fibre products 
\begin{equation}
 W^s_{Y}(\widehat q) \times_{\ev} \M^*_{H,0,\R\times Y;k_-,k_+}(A;J_Y) \times_{\ev} W_{Y}^u(\widecheck p).
\label{fib prod1}
\end{equation}
The fibre products are defined with respect to the inclusion maps
\begin{align*}
W^s_{Y}(\widehat q), W_{Y}^u(\widecheck p) &\to Y
\end{align*}
and the evaluation maps
\begin{align*}
\widetilde{\ev}_Y \colon \M^*_{H,0,\R\times Y;k_-,k_+}(A;J_Y) &\to Y\times Y. %
\end{align*}

To give a similar description of Cases 2 and 3, we denote by
$$
\M_H^*(B;J_W)
$$
the space of simple $J_W$-holomorphic cylinders in $W$ with a removable singularity at $-\infty$, realizing the class 
$B\in H_2(X;\Z)$. Denote the quotient of $\M_H^*(B;J_W)$ by the $\C^*$-action as $\M_X^*(B;J_W)$. We also define 
$\M^*_{H,k_+}(0;J_Y)$ as the space of Floer cylinders $\tilde v \colon \R\times S^1 \to \R \times Y$ such that 
$\lim_{s\to\infty}\tilde v(s,.)$ is a Hamiltonian orbit of multiplicity $k_+$, $\lim_{s\to-\infty}\tilde v(s,.)$ 
is the corresponding Reeb orbit and $\pi_\Sigma \circ \tilde v$ is constant in $\Sigma$. 
 
The spaces of Floer cylinders with cascades in Case 2 can now be written as unions of fibre products 
\begin{equation}
 W^s_{Y}(\wh p) \times_{\ev} \big(\M_X^*(B;J_W) \times_{\tilde \ev} \M^*_{H,1,\R\times Y;k_-,k_+}(0;J_Y) \big) \times_{\ev} W_{Y}^u(\wc p).
\label{fib prod2}
\end{equation}
and in Case 3 we have unions of 
\begin{equation}
 W^u_{W}(x) \times_{\ev^1_-} \left(\M^*_{H}(B;J_W)  \prescript{}{\ev^1_+\,}\times\prescript{}{\ev^2_-}{}  \M^*_{H,k_+}(0;J_Y) \right) \times_{\ev^2_+} W_{Y}^u(\wc p).
\label{fib prod3}
\end{equation}

The relevant transversality results for the spaces \eqref{fib prod1}, \eqref{fib prod2} and \eqref{fib prod3} are established in 
\cite{DiogoLisiSplit}*{Section 5}. These spaces are oriented using the fibre product orientation convention.
For details, see \cite{DiogoLisiSplit}*{Section 7} and Section \ref{compute signs} below.

\ 

Given our assumptions on $H$ and $J_Y$, there are geometric actions on the spaces of solutions to Floer's 
equation in $\R\times Y$. 
Let $\tilde v = (b,v) \colon \R\times S^1 \setminus \Gamma \to \R \times Y$ be a solution to Floer's equation (\ref{Floer eq Y}). 
There is an action of $S^1 \times S^1$ on the space of such solutions, with each circle factor acting by rotation on the domain and on the target, respectively:
$$
 (\theta_1,\theta_2) . \tilde v = \tilde v_{(\theta_1,\theta_2)}
$$
where
$$
\tilde v_{(\theta_1,\theta_2)} (s,t) = (b(s,t+\theta_1), \phi^{\theta_2}_R \circ v(s,t+\theta_1)).
$$
Here, $\phi^\theta_R$ denotes the Reeb flow for time $\theta$.
If we identify the images of the simple Reeb orbits underlying $x_\pm(t) = \lim_{s\to \pm\infty} v(s,t)$ with $S^1$, then the following map keeps track of the asymptotic effect of the action: 
\begin{align*}
 \rot_{\tilde v}\colon S^1 \times S^1 &\to S^1 \times S^1 \\
 (\theta_1,\theta_2) &\mapsto \left(\lim_{s\to -\infty} v(s,\theta_1) + \theta_2, \lim_{s\to +\infty} v(s,\theta_1) + \theta_2 \right).
\end{align*}
If $k_\pm$ are the multiplicities of the periodic orbits $x_\pm (t)$, then, $\rot_{\tilde v}$ is the linear map represented in matrix form as 
\begin{equation}\label{rot v}
\rot_{\tilde v} = \begin{pmatrix}
k_- & 1 \\
k_+ & 1
\end{pmatrix}.
\end{equation}

Suppose now that $\tilde v_0 \colon \R \times S^1 \to W$ is a non-constant
$J_W$-holomorphic cylinder with removable singularity at $-\infty$ (as
in Case 3 in Proposition \ref{P:split SH cases}). There is an $S^1$-action on such
curves, by rotation on the domain: $$ \theta . \tilde v_0 (s,t) = \tilde v_0^\theta (s,t) 
= \tilde v_0(s,t+\theta).  $$ Identifying the image of the simple Reeb
orbit underlying $\gamma(t) = \lim_{s\to \infty} \pi_Y(\tilde v_0(s,t))$ with
$S^1$, we have 
\begin{align} \label{rot v0}
 \rot_{\tilde v_0}\colon S^1 &\to S^1 \\ 
 \theta &\mapsto \lim_{s\to +\infty}
 \tilde v_0(s,\theta). \nonumber
\end{align} 
If the multiplicity of the periodic Reeb orbit $\gamma(t)$
is $k$, then $\rot_{\tilde v_0}$ is multiplication by $k$.

\section{Sketch of isomorphism between the two Floer complexes}\label{S:compactness}

\begin{proposition}
For a generic choice of almost complex structures of the cylindrical type $J_\kappa$, $J_W$ and $J_Y$ described in Section \ref{sec:almost complex structures}, and for $R>0$ large enough,
the presplit and split chain complexes are well-defined and are chain isomorphic.
Furthermore, these complexes compute the symplectic homology of $W$.
\end{proposition}

We will provide a sketch of the proof of this proposition. 
The main steps of such a proof are as follows:
\begin{itemize}
	\item Show that these complexes are well-defined. For this, we need 
		to obtain transversality for generic almost complex structures
		in the (very restrictive) class we consider, which are both $\R$ and
		$S^1$ (Reeb) invariant in the cylindrical end. This problem is
		similar in both the presplit and split complexes, and is addressed
                in detail in \cite{DiogoLisiSplit}*{Section 5}  
		in the more difficult setting of the split complex. 
		(The additional difficulty in the split complex comes from 
		Floer cylinders with cascades that have punctures capped
		with planes in $W$. The planes can be made transverse
		by a perturbation of $J$ supported in $\V \subset \overline W$. 
                Such a perturbation
		does not make the upper level(s) transverse. 
		In the presplit case,
		such a building corresponds to a single cylinder with ``toes''
		that dip in to $\V \subset \overline W$, and thus 
		a perturbation in $\V$ suffices to obtain 
		transversality. See Figure
		\ref{fig:transversality split presplit}.)

	\item For these complexes to be well-defined, we also need to establish
		that there are no curves counted either in $\partial$ or in the proof of $\partial^2=0$ that
		are asymptotic to the degenerate constant orbits at 
		$\partial \supp dH$. As we pointed out in Remark \ref{R:not MB}, these orbits fail to even be Morse--Bott,
		and thus represent a break-down of the standard analytical theory.
		We saw in Lemma \ref{non-MB} why these orbits don't appear in $\partial$. A similar argument (applying 
		Lemma \ref{no + limits in W}) implies that they also don't need to be considered when proving $\partial^2 = 0$.
	\item Compactness and gluing arguments allow us to identify the presplit and split moduli spaces
		for sufficiently large stretching parameter. Notice that by formula \eqref{J_kappa}, 
		each of the almost complex structures $J_\kappa$ used in 
		stretching is biholomorphic to $J_W$, though the support of $H$ is moved towards $+\infty$
		by the biholomorphism. Thus, a slight modification 
		of the standard SFT compactness theorem works in our setting.
		It is important to point out that after stretching the neck, the component of a Floer cylinder that is contained in $\R\times Y$ is connected, which is why the split Floer cascades we described above contain information about all the presplit Floer cascades. 
		This follows from an argument in \cite{BOExactSequence}*{Step 1 in proof of Proposition 5}.  
		By a gluing argument, for each subcomplex of
		the split complex given by a bound in action, there exists a sufficiently large stretching 
		parameter for which we get an identification with the corresponding
		subcomplex of the presplit complex.
	\item Show that the presplit complex is quasi-isomorphic to a symplectic homology complex
		obtained from a non-degenerate Hamiltonian. This involves
		constructing a continuation map $\Phi$ connecting the two chain complexes. 
		The delicate part of the
		proof that this is a chain map
		stems again from the failure of $H$ to be 
		Morse--Bott non-degenerate along the boundary $\partial \supp dH$.
		We address this difficulty by means of the Abouzaid--Seidel
		Lemma \ref{non-MB} below.
		A usual action filtration argument
		gives that this is a chain isomorphism. 
\end{itemize}

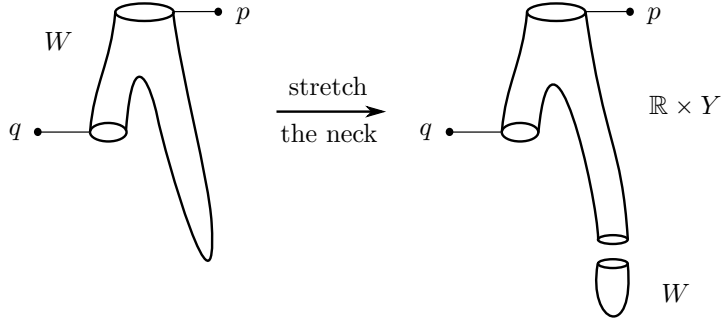
\begin{figure}
  \begin{center}
    \def\svgwidth{0.8\textwidth}

\begingroup
  \makeatletter
  \providecommand\color[2][]{%
    \errmessage{(Inkscape) Color is used for the text in Inkscape, but the package 'color.sty' is not loaded}
    \renewcommand\color[2][]{}%
  }
  \providecommand\transparent[1]{%
    \errmessage{(Inkscape) Transparency is used (non-zero) for the text in Inkscape, but the package 'transparent.sty' is not loaded}
    \renewcommand\transparent[1]{}%
  }
  \providecommand\rotatebox[2]{#2}
  \ifx\svgwidth\undefined
    \setlength{\unitlength}{505.90228431pt}
  \else
    \setlength{\unitlength}{\svgwidth}
  \fi
  \global\let\svgwidth\undefined
  \makeatother
  \begin{picture}(1,0.41457556)%
    \put(0,0){\includegraphics[width=\unitlength]{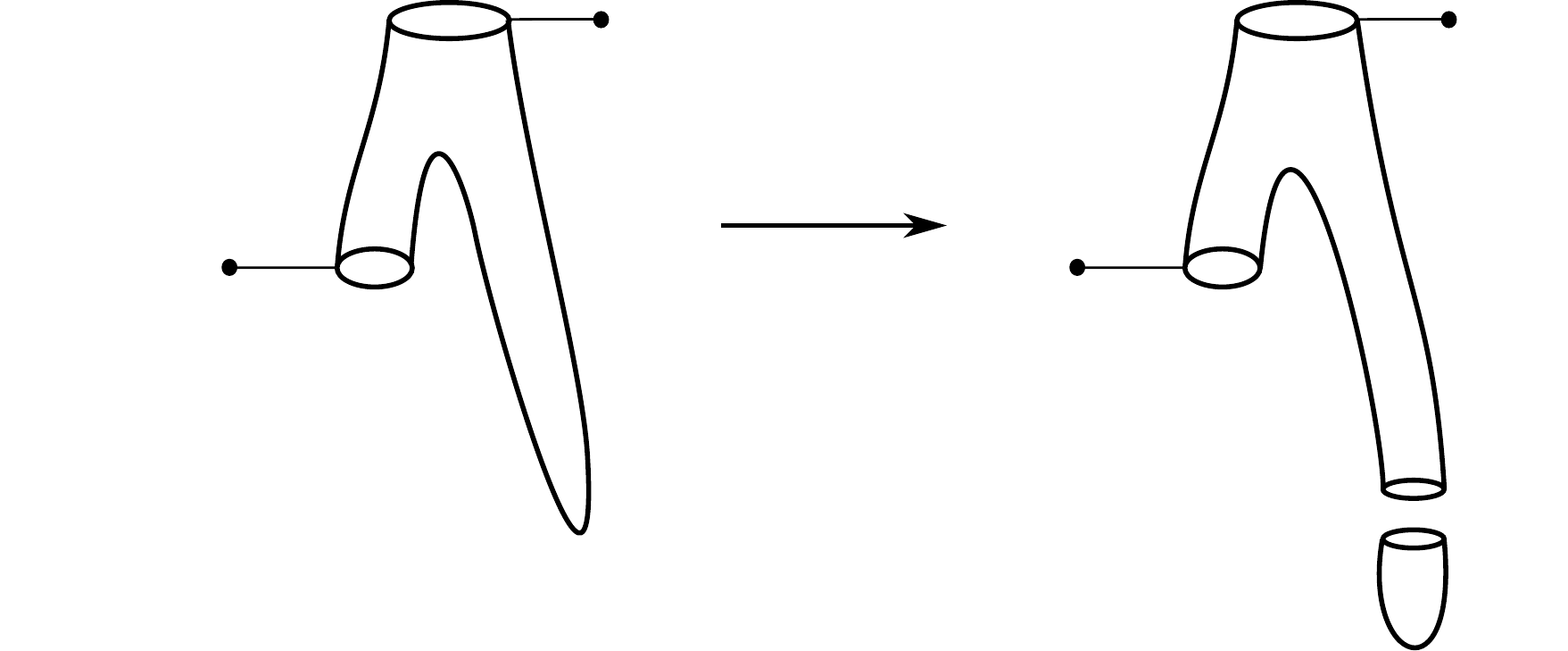}}%
    \put(0.10819533,0.24029146){\color[rgb]{0,0,0}\makebox(0,0)[lb]{\smash{$q$}}}%
    \put(0.15491812,0.35196881){\color[rgb]{0,0,0}\makebox(0,0)[lb]{\smash{$W$}}}%
    \put(0.94931704,0.26739761){\color[rgb]{0,0,0}\makebox(0,0)[lb]{\smash{$\R\times Y$}}}%
    \put(0.96558072,0.02018953){\color[rgb]{0,0,0}\makebox(0,0)[lb]{\smash{$W$}}}%
    \put(0.64779644,0.24029146){\color[rgb]{0,0,0}\makebox(0,0)[lb]{\smash{$q$}}}%
    \put(0.40744728,0.39208592){\color[rgb]{0,0,0}\makebox(0,0)[lb]{\smash{$p$}}}%
    \put(0.94704832,0.39208592){\color[rgb]{0,0,0}\makebox(0,0)[lb]{\smash{$p$}}}%
    \put(0.47599853,0.29229921){\color[rgb]{0,0,0}\makebox(0,0)[lb]{\smash{stretch}}}%
    \put(0.46526508,0.23096499){\color[rgb]{0,0,0}\makebox(0,0)[lb]{\smash{the neck}}}%
  \end{picture}%
\endgroup
   \end{center}
  \caption{ By perturbing $J$ in $\V$
(i.e.~away from the cylindrical end of $W$), the curve on the left can
be made transverse. Such a perturbation only makes the plane in $W$
transverse for the building on the right.} 
   \label{fig:transversality split presplit}
\end{figure}

\ 

We now address the difficulty in the proof of the fact that a continuation map between presplit symplectic homology and a Morse perturbed version is a chain map.

First, we will explain what the construction would entail if all
of the 1-periodic orbits of the admissible Hamiltonian $H$ (as in Definition \ref{def:J_shaped}) 
were Morse--Bott non-degenerate.

Recall that we defined a differential $\partial_{pre}$ on $SC_*(W,H)$, as in
\eqref{eqn:chain complex}, which we called the presplit differential. 
We then construct a non-degenerate Hamiltonian $\tilde H\colon S^1 \times W \to \R$ such that $\tilde H(t,w) = H(w)$ outside a small neighbourhood of the periodic orbits of $X_H$. 
We require that $\tilde H$ be $C^2$-small and Morse on $W \setminus \big((-\varepsilon,\infty)\times Y\big)$ 
for some $\varepsilon >0$. In addition, 
$\tilde H$ is a small time-dependent perturbation of $H$ near the non-constant periodic orbits of $X_H$, using auxiliary Morse functions on the manifolds of orbits in a manner similar to \cite{BOSymplecticHomology}*{page 73}.  
Picking a generic almost complex structure on $W$, we get a chain complex $SC_*(W,\tilde H)$ that computes the symplectic homology of $W$. Denote its differential by $\tilde \partial$.  

We construct a chain map 
$$
\Phi \colon SC_*(W,H) \to SC_*(W,\tilde H)
$$
as follows. Let $\overline H\colon \R\times S^1 \times W \to \R$ be such that
\begin{itemize}
 \item $\overline H(s,t,w) = H(w)$ if $s> 1$;
 \item $\overline H(s,t,w) = \tilde H(t,w)$ if $s<-1$;  
 \item $\partial_s \overline H \leq 0$.
\end{itemize}
We will impose two more technical conditions on $\overline H$, in Lemma \ref{lem:noConvergeToBadOrbits}. 

Picking a domain-dependent almost complex structure $\overline{J}$ interpolating
between the ones used in the two chain complexes, we define
$\Phi$ via counts of index 0 Floer cylinders with cascades for
the data of the pair ($\overline H$, $\overline J$) 
(these are cascades that satisfy \eqref{eq:0,1 continuation} below with $\beta = dt$).
We will refer to such cylinders with cascades as \textit{continuation map cascades}.

If all the periodic orbits of $H$ were Morse--Bott non-degenerate, 
the argument is standard: we consider index $1$ families of cascades solving
the continuation map equation. These families are compact, and their
boundaries are unions of terms that correspond to configurations either in
$\Phi \circ \partial_{pre}$ or in $\tilde \partial \circ \Phi$. Furthermore,
any such configuration in $\Phi \circ \partial_{pre}$ or in $\tilde \partial
\circ \Phi$ can be glued, showing that all such configurations arise as
boundary components of these 1-dimensional moduli spaces. 

The presence of degenerate orbits of $X_H$ along $\partial \supp dH$
introduces two main difficulties. The first is in the consideration of index
$0$ and index $1$ cascades solving the continuation map equation, specifically
in the case of $\Phi(x)$ where $x$ is a generator of the complex coming from a 
critical point of $f_W$ (or the index $1$ analogue). In this case, the
asymptotic boundary condition for the cascade requires that the positive end
converge to a (constant) orbit in $W$ in the \textit{ascending} manifold of
$x$. This ascending manifold is non-compact, and thus the family of solutions is not necessarily
compact. We will show that in actual fact, for energy reasons and for suitably
chosen $\tilde H$, no non-constant cascade of this form exists. 

The second potential problem comes from the compactification of the moduli
space of index $1$ continuation map cascades. A portion of such cascades can
converge in $C^\infty_{loc}$ to a Floer cylinder with an end that converges
(badly) to the degenerate orbits in $\partial \supp dH$. (We may arrange for
such convergence to happen at the $+\infty$ puncture by changing our
subsequence and shifts.) We will again show that this kind of behaviour does
not occur for energy reasons.

The rest of this section will now prove these two claims. 
In order to rule both of these problem configurations out, we will 
combine 
the convergence of sequences 
in Lemma \ref{L:convergence to orbit} 
and a careful application of the following result 
originally due to Abouzaid and Seidel (see \cite{AbouzaidSeidel}*{Lemma 7.2} 
and also \cite{Ritter}*{Lemma 19.3}). 
We will state the result in slightly greater generality than required for our
purposes here.
In the following lemma, for instance, the annulus $S = (-\epsilon, 0] \times
S^1$ should be thought of as a subset of the punctured cylinder, and $\beta$
will then be the restriction of the form $dt$ to this annulus. 

\begin{lemma} \label{L:ASestimate}
Let $Y$ be a co-oriented contact manifold with contact form $\alpha$ and denote the corresponding 
Reeb vector field by $R$. 
Let $S$ be the annulus $S = (-\epsilon, 0] \times S^1 \subset \R \times S^1$,
with coordinates $\sigma \in (-\epsilon, 0]$ and $\tau \in S^1 = \R / \Z$.
Let $i$ denote the complex structure on $S$ with $i \partial_\sigma =
\partial_\tau$.

Let $J$ be an $S$-dependent family of almost complex structures on $[r_0, r_0 + \delta] \times Y$ for small $\delta>0$
and let $H \colon S \times [r_0, r_0 + \delta] \times Y \to \R$ be an $S$-dependent family of Hamiltonians
with the following locally radial behaviour:
\begin{enumerate}
\item $J(z, r, y) \partial_r = R$, $J(z, r, y) \ker \alpha = \ker \alpha$,
\item $H(z, r, y) = h(z, \e^r)$, for some function $h(z,\rho)$.
\end{enumerate} 
Let $\beta$ be a closed $1$-form on $S$. 

Suppose that $\tilde v \colon S \to [r_0, r_0 + \delta) \times Y$ satisfies the following equation:
\begin{equation}
\label{eq:0,1 continuation}
0 = 2 (d\tilde v - X_H \otimes \beta)^{0,1} = d\tilde v + J(z, \tilde v) d\tilde v \circ i - X_H \otimes \beta - J(z, \tilde v) X_H \otimes \beta \circ i,
\end{equation}
and $\tilde v( \{ 0 \} \times S^1 ) \subset \{ r_0 \} \times Y$.

Then, taking $\lambda = \e^r \alpha$,
\begin{align*}
\int_{ \{ 0 \} \times S^1 } \tilde v^*\lambda - H(z, \tilde v(z)) \beta &\le \int_{ \{ 0 \} \times S^1 } ( \lambda(X_H) - H ) \beta = \\
&=\int_{ \{ 0 \} \times S^1 } \left ( \e^{r_0} \frac{\partial h}{\partial \rho}(z, \e^{r_0}) - h(z, \e^{r_0}) \right ) \beta
\end{align*}
\end{lemma}

\begin{proof}

Applying $\lambda$ to \eqref{eq:0,1 continuation} and using $\lambda \circ J = \e^r dr$ and $\lambda( JX_H) = 0$, we obtain
\[
\tilde v^*\lambda + \e^r dr( d\tilde v \circ i) = \lambda(X_H) \beta.
\]
By hypothesis, we have that $r( \tilde v(0,\tau) ) \equiv r_0$. 
Furthermore, since the image of the annulus $S$ is not below $r=r_0$, 
we have $dr( d\tilde v \circ i)(\frac{\partial}{\partial \tau}) \ge 0$ along $\{0\} \times S^1$, hence
\[
\tilde v^*\lambda(\frac{\partial}{\partial \tau}) \le \lambda(X_H) \beta(\frac{\partial}{\partial \tau}).
\]
The result now follows.
\end{proof}

Following \cites{AbouzaidSeidel,Ritter}, we introduce the following terminology:

\begin{definition}
Let $(S, j)$ be a Riemann surface and let $\beta$ be a 1-form on $S$ so $d\beta \le 0$. 
Let $(W, \omega)$ be a symplectic manifold, and let $J$ be an $S$-dependent family of almost complex structures on $W$, compatible with $\omega$.
Let $H \colon S \times W \to \R$ be as $S$-dependent Hamiltonian. We write $dH$ to indicate the (total) differential of $H$ and $d_W H$ to denote the $S$-dependent family of 1-forms on $W$ given by $d_W H|_{z \in S} = d (H(z, \cdot))$.

For a map $\tilde v \colon S \to W$, we define its
\begin{itemize}
\item {\em Topological energy}: $E_\topo(\tilde v) = \int_{S} \tilde v^*( \omega - d(H \beta) )$;
\item {\em Geometric energy}: $E_\geom(\tilde v) = \int_{S} \tilde v^*\omega - (\tilde v^*d_W H) \wedge \beta$.
\end{itemize}
\end{definition}

Note that for any solution to Floer's equation \eqref{eq:0,1 continuation}, $E_\geom(\tilde v) = \int_S \| \partial_s \tilde v\|^2 \geq 0$ is (pointwise) non-negative. 
We say that the triple $(\beta, H,J)$ is \defin{monotone} if the topological energy dominates
the geometric energy pointwise (in the sense that the integrand in the geometric energy is equal to the integrand in the topological energy times a function $f\leq 1$). Observe that they are equal if $H$ is $z$-independent and $\beta = dt$.

\begin{lemma} \label{lem:noConvergeToBadOrbits}
Let $(W, d\lambda)$ be an exact symplectic manifold with cylindrical end $((r_0, +\infty) \times Y, d( \e^r \alpha ))$, $r_0 < \log 2$.

Let $(\beta,H, J)$ be a monotone triple on the cylinder $(\R \times S^1, i)$ with 1-form $\beta = dt$ and 
with $H_{(+\infty, t)}$ an admissible Hamiltonian as in Definition \ref{def:J_shaped}.

Let $\gamma_+$ be a (constant) 1-periodic orbit of $H_{(+\infty, t)}$ contained in the level $\{ r_1 \} \times Y$ for some $r_0< r_1 \leq \log 2$
and let $\gamma_-$ be a 1-periodic orbit
of $H_{(-\infty, t)}$, contained entirely below the the level $ \{r_1 \} \times Y$.

Suppose furthermore that $H \colon \R \times S^1 \times W \to \R$ and $J$ are locally radial (in the sense of Lemma \ref{L:ASestimate}) 
on $(r_1-\delta, r_1 + \delta ) \times Y$ for small $\delta>0$, with the additional condition that 
$H(s,t, r, y) = h(s, \e^r)$ for some function $h(s,\rho)$, with 
\begin{align*}
\e^r \frac{\partial h}{\partial \rho}(s, \e^r) - h(s, \e^r) &\ge 0 \\
\partial_s \Big( \e^r \frac{\partial h}{\partial \rho}(s, \e^r) - h(s, \e^r) \Big) &\le 0.
\end{align*}

Then, there does not exist any $\tilde v \colon \R \times S^1 \to W$ solving Floer's equation \eqref{eq:0,1 continuation}
for which there exists a sequence $s_k \to +\infty$ along which $\tilde v(s_k, t) \to \gamma_+(t) \subset \{ r_1 \} \times Y$  
and for which $\tilde v(s,t) \to \gamma_-(t), s\to -\infty$.
\end{lemma}

\begin{proof}
Suppose there is a $\tilde v$ as in the statement. Consider a smooth function $\tilde r \colon W \to \R$ that is equal to the first coordinate $r$ on $(r_1-\delta/2, r_1 + \delta/2) \times Y$ and constant outside of $(r_1-\delta, r_1 + \delta) \times Y$.
We can assume that $\delta>0$ is small enough that $\tilde r \circ \gamma_-(t) < r_1-\delta/2$. 
Take a regular value $r_1-c$ of $\tilde r \circ \tilde v \colon \R \times S^1 \to \R$ in the interval $(r_1-\delta/2, r_1)$. 

By hypothesis, there exists a $k_0$ sufficiently large so that for each $k \ge k_0$, $\min_{t \in S^1} \tilde r( \tilde v(s_{k}, t) ) > r_1-c/2$. Fix such a $k$.

Define the subdomain $S_k$ of $\R \times S^1$ by 
\[
 S_k = \{ (s,t) \in (-\infty, s_{k}) \times S^1 \, | \, \tilde r( \tilde v(s,t)) > r_1-c \}.
\]
Note that while this subdomain may be disconnected, its closure in $\R \times S^1$ is compact, due the fact that $\tilde v(s, t) \to \gamma_-(t)$ as $s \to -\infty$ and that $\tilde r \circ \gamma_-(t) < r_1-\delta/2$.  Since $k \ge k_0$, this set is non-empty.

By the fact that $r_1-c$ is a regular value of 
$\tilde r \circ \tilde v$, we have that $\partial S_k$ is a collection of smooth, embedded closed curves.
Let $\Gamma$ be the collection of embedded closed curves in $\R \times S^1$ so that the boundary of $S_k$ is $\partial S_k = \Gamma \cup \{ s_{k} \} \times S^1$. Let $\Gamma$ be oriented with the orientation induced as the boundary of $S_k$. 
As an aside, note that the projection of $\Gamma \subset \R \times S^1$ to $S^1$ has total degree $-1$, since it separates the cylinder.

By construction, $\tilde v(\Gamma) \subset \{ r_1- c \} \times Y$, and by definition $\tilde r (\tilde v(S_k))\subset (r_1-c,r_1+\delta]$. 
Combining Stokes's Theorem with Lemma \ref{L:ASestimate} (and implicitly using biholomorphisms between neighborhoods of the connected components of $\Gamma$ in $S_k$ and annuli of the form $(-\epsilon,0]\times S^1$), we obtain
\begin{align*}
E_{\topo} ( \tilde v|_{S^k} ) &=  \int_{ \{ s_k \} \times S^1 } \tilde v^*\lambda - H(s,t,\tilde v(s,t) ) dt  + \int_{\Gamma} \tilde v^*\lambda -H(s,t,\tilde v(s,t) ) dt \\
	&\le \int_{ s=s_k } \tilde v^*\lambda - H(s_k,t,\tilde v(s_k,t) ) dt + \int_{\Gamma} (\e^{r_1-c} \frac{\partial h}{\partial \rho}(s, \e^{r_1-c}) - h(\e^{r_1-c}) ) dt \\
	&= \int_{ s=s_k }  \tilde v^*\lambda  - H(s_k,t,\tilde v(s_k,t) ) dt - \\
						& \hphantom{mmmmmmm}  -  \int_{ s=s_k} ( \e^{r_1-c} \frac{\partial h}{\partial \rho}(s, \e^{r_1-c}) - h(\e^{r_1-c}) ) dt + \\
						& \hphantom{mmmmmmm}  + \int_{S_k} \partial_s \big( \e^{r_1-c} \frac{\partial h}{\partial \rho}(s, \e^{r_1-c}) - h(\e^{r_1-c}) \big) \, ds \wedge dt \\
	&\le \int_{s=s_k} \tilde v^*\lambda  - H(s_k,t,\tilde v(s_k,t) ) dt.  
\end{align*}

As the parameter $k \to +\infty$, the right hand side converges to the action of $\gamma_+$, which is $0$. Furthermore, since the Floer data is assumed to be monotone, $E_{\geom}( \tilde v|_{S_k}) \le E_{\topo}( u|_{S_k})$ and since the $S_k$ are nested, 
$E_{\geom}(\tilde v|_{S_k})$ is monotone non-decreasing in $k$. It follows that $E_{\geom}( \tilde v|_{S_k}) = 0$ for all $k$. As $S_k$ has non-empty interior, we conclude that $\partial_s\tilde v = 0$, which contradicts the fact that $\tilde v(s,t) \to \gamma_-$ as $s \to -\infty$.
\end{proof}

We now explain how the previous result completes the argument that none of the problematic configurations described earlier occur, and thus $\Phi$ is a chain map. 

We take $\overline H$ satisfying the radial
conditions of Lemma \ref{lem:noConvergeToBadOrbits}. 
The lemma implies that the moduli
space of index $0$ continuation cascades that are counted in $\Phi$ are
compact. Furthermore, none has positive end with an asymptotic limit in the
degenerate orbits at $\partial \supp dH$.  

Now, consider index $1$ continuation cascades. First of all, by this lemma,
none of the curves has a positive puncture at a degenerate orbit. A sequence of such
cascades, with a fixed number of cascades and with a uniform bound on the
action of the positive puncture, will then either converge to a broken
cylinder with cascades (corresponding to configurations in 
$\tilde \partial \circ \Phi$ or in $\Phi \circ \partial_{pre}$) or will admit
a sequence of domain reparametrizations along which it converges in
$C^\infty_{loc}$ to a cylinder of the kind ruled out by Lemma
\ref{lem:noConvergeToBadOrbits}.

Finally, we observe that Lemma \ref{L:ASestimate} gives us control over the configurations that 
can occur as possible limits after splitting. Specifically, this lemma, understood as a maximum 
principle, is behind the requirement that a split Floer
cylinder with negative end asymptotic to a non-constant Hamiltonian orbit in
$\R \times Y$ must have its positive end asymptotic to a non-constant
Hamiltonian orbit. (This is by an argument similar to 
\cite{BOExactSequence}*{Proof of Proposition 5, Step 1}.)

\part{Ansatz}

\label{Ansatz}

\section{A correspondence between Floer and pseudoholomorphic cylinders}

The content of this section is technical, so we begin
with an overview and explanation of its relevance to the study of the split symplectic
homology differential. The main results are Propositions \ref{Ansatz1}
and \ref{Ansatz2}, which we summarize in the following result. For precise
statements, notably of the uniqueness, see the full statements of the
propositions in the later sections.

\begin{proposition} \label{P:summary ansatz}
 There is a bijection 
 $$
 \{\tilde v\colon \R \times S^1\setminus \Gamma \to \R\times Y\} \longleftrightarrow \{(\tilde u,{\bf e})\} / \R
 $$
 where 
 \begin{itemize}
     \item $\tilde v$ is an augmented Floer cylinder of finite hybrid energy (as
         in Equation \eqref{E:hybridEnergy}),
  \item $\tilde u\colon \R \times S^1\setminus \Gamma \to \R\times Y$ is a finite energy pseudoholomorpic curve with negative ends at $\Gamma \cup \{-\infty\}$ and positive end at $+\infty$,
  \item ${\bf e}\colon \R\times S^1\setminus \Gamma \to \R\times S^1$ is an essentially unique solution 
      to a partial differential equation that depends on $\tilde u$ (Equation \eqref{E:cylinder1}),
  \item if we write $\tilde u = (a,u)$ and ${\bf e} = (e_1,e_2)$, the quotient
      is by the $\R$-action given, for $c \in \R$, by
  $$
  c \cdot ( (a,u),(e_1,e_2)) = ((a+c,u),(e_1-c,e_2) ), \qquad c\in \R,
  $$
  \item if $\tilde v$ converges to a Hamiltonian orbit of the form $(b,\gamma)$ in $\R\times Y$ at one of its ends (where $b$ is a 
  constant and $\gamma$ is a Reeb orbit), then $\tilde u$ converges to the Reeb orbit $\gamma$ at the corresponding end; if $\tilde v$ 
  converges to a Reeb orbit in $Y$ at one of its ends, then $\tilde u$ converges to the same orbit at the corresponding end.
 \end{itemize}
\end{proposition}

Recall that by Proposition \ref{P:split SH cases}, there are four types of contributions to the differential, referred to as Cases 0, 1, 2 and 3.  

Case 0 configurations are gradient flow lines in $Y$, so they can be described by the Morse differential in $Y$. 

\

Case 1 configurations are elements of 1-dimensional fibre products \eqref{fib prod1}
$$
W^s_{Y}(\wh q) \times_{\ev} \MM^*_{H,0,\R\times Y;k_-,k_+}(A;J_Y) \times_{\ev} W_{Y}^u(\wc p)
$$
where $q\neq p\in \Crit(f_\Sigma)$ and $A\in H_2(\Sigma;\Z)$. Recall that $\MM^*_{H,0,\R\times Y;k_-,k_+}(A;J_Y)$ is the 
space of parametrized unpunctured Floer cylinders $\tilde v$, going from orbits of multiplicity $k_-$ (as $s\to -\infty$) to orbits 
of multiplicity $k_+$ (as $s\to +\infty$), where $k_+>k_-$.
The fibre product above has a free action of $\R$ by translations of $\tilde v$ in the $s$-direction, so it is a finite disjoint union of copies of $\R$.

By Proposition \ref{P:summary ansatz}, we get a bijection between $\MM^*_{H,0,\R\times Y;k_-,k_+}(A;J_Y)$ 
and the space of $J_Y$-holomorphic cylinders in $\R\times Y$, going from Reeb orbits of multiplicity $k_-$ to Reeb orbits of multiplicity $k_+$. 
In Section \ref{S:meromorphic}, we will relate such pseudoholomorphic cylinders to 
pseudoholomorphic spheres in $\Sigma$ and meromorphic sections of line bundles over $\CP^1$.  

\ 

Case 2 configurations are elements of 1-dimensional fibre products \eqref{fib prod2} 
$$
W^s_{Y}(\wh p) \times_{\ev} \left(\M^*_X(B;J_W) \times_{\tilde\ev} \MM^*_{H,1,\R\times Y;k_-,k_+}(0;J_Y) \right) \times_{\ev} W_{Y}^u(\wc p)
$$
where $p\in \Crit(f_\Sigma)$, $B\in H_2(X;\Z)$. Recall that $\M^*_X(B;J_W)$ is a space of unparameterized $J_W$-holomorphic planes $U$ in $W$ 
(which can be identified with a space of $J_X$-holomorphic spheres in $X$, by Lemma \ref{planes = spheres}) and 
$\MM^*_{H,1,\R\times Y;k_-,k_+}(0;J_Y)$ is a space of Floer cylinders $\tilde v\colon \R\times S^1\setminus \{P\} \to \R \times Y$ 
projecting to a point in $\Sigma$, asymptotic at $-\infty$ (resp. $+\infty$) to an orbit of multiplicity $k_-$ (resp. $k_+$) and with 
a negative puncture at $P\in \R\times S^1$ that is asymptotic to a Reeb orbit of multiplicity $k_+-k_-$. The stable and unstable manifolds 
in the fibre product force $\tilde v$ to project to $p\in \Sigma$. Hence, $U$ is asymptotic to a Reeb orbit of multiplicity 
$k_+-k_- = B\bullet \Sigma$ over $p$. Proposition \ref{P:summary ansatz} now implies that $\MM^*_{H,1,\R\times Y;k_-,k_+}(0;J_Y)$ 
can be identified with the space of punctured $J_Y$-holomorphic cylinders $\R\times S^1 \setminus \{P\} \to \R\times Y$ that are branched covers of a trivial cylinder, with 
negative punctures at $-\infty$ and $P$ asymptotic to orbits of multiplicity $k_-$ and $k_+-k_-$, respectively, and with a positive 
puncture at $+\infty$ asymptotic to an orbit of multiplicity $k_+$. 

\ 

Finally, Case 3 configurations are elements of fibre products \eqref{fib prod3}
$$
 W^u_{W}(x) \times_{\ev^1_-} \left(\M^*_{H}(B;J_W)  \prescript{}{\ev^1_+\,}\times\prescript{}{\ev^2_-}{}  \M^*_{H,k_+}(0;J_Y) \right) \times_{\ev^2_+} W_{Y}^u(\wc p)
$$
where $p\in \Crit(f_\Sigma)$, $x\in \Crit(f_W)$ and $B\in H_2(X;\Z)$. Here, $\M^*_{H}(B;J_W)$ is a space of 
$J_W$-holomorphic planes $\tilde v_0$ in $W$, asymptotic to a Reeb orbit of multiplicity $k_+=B\bullet\Sigma$ 
(recall that $H=0$ in $W$ after neck-stretching, so we can think of $J_W$-holomorphic planes as Floer cylinders that asymptote to 
constants at $-\infty$). An element in $\M^*_{H,k_+}(0;J_Y)$ is a Floer cylinder $\tilde v_1$ in $\R\times Y$ that projects to a point in $\Sigma$, is asymptotic at $+\infty$ 
to a Hamiltonian orbit of multiplicity $k_+$, and asymptotic at $-\infty$ to a Reeb orbit of multiplicity $k_+$. The unstable manifold 
at the right of the fibre product forces $\tilde v_1$ to project to $p$ in $\Sigma$. The fibre product forces the positive asymptotic orbit of $\tilde v_0$ to be the same as the negative asymptote of $\tilde v_1$ (which projects to $p$ in $\Sigma$). 
Proposition \ref{P:summary ansatz} implies that $\M^*_{H,k_+}(0;J_Y)$ can be identified with the 
space of trivial $J_Y$-holomorphic spheres in $\R\times Y$, over Reeb orbits of multiplicity $k_+$.

\subsection{Sobolev spaces and Morse--Bott Riemann--Roch}

Before we state and prove the main results in this section, we need to briefly set-up 
the appropriate Fredholm theory. We refer to 
\cite{DiogoLisiSplit}*{Section 5.2} for more details. 

We will need to consider exponentially weighted Sobolov spaces of sections of the trivial complex 
line bundle over a Riemann surface $\dot S= \R\times S^1\setminus \Gamma$, where $\Gamma$ is a finite set.  
We think of this as a punctured sphere, with punctures at $\Gamma \cup \{ \pm
    \infty \}$. 
These are partitioned into a single positive puncture $\Gamma_+ = \{ + \infty
\}$ and negative punctures $\Gamma_- = \Gamma \cup \{ -\infty \}$. Near each
puncture, we fix cylindrical coordinates with the sign corresponding to the
sign of the puncture. 
For $\delta > 0$, we 
denote by $W^{1,p,\delta}(\dot S, \C)$
the space of sections that decay exponentially like $\e^{-\delta |s|}$ near the
punctures (in the cylindrical coordinates already chosen near the punctures).
In certain circumstances, we will find it useful to consider sections with 
exponential growth, and denote these by taking $\delta < 0$.

Let 
\begin{align*}
 D &\colon W^{1,p,\delta}(\dot S,\C) \to L^{p,\delta}(\dot S,\Lambda^{0,1}T^*\dot S)
\end{align*}
be a Cauchy--Riemann operator. Denote by $\mathbf{A}_z\colon W^{1,p}(S^1,\C) \to L^p(S^1,\C)$ 
the asymptotic operator at the puncture $z\in \Gamma \cup \{ \pm \infty \}$. We do not assume 
that it is non-degenerate, since we want to study Morse--Bott moduli spaces. 
Let $\mathbf V$ be a collection of vector spaces, associating to each
puncture $z$ a vector subspace $V_z$ of $\ker (\mathbf{A}_z)$.
For $\delta >0$, we then consider the space $W^{1,p,\delta}_{\mathbf V}(\dot S, \C)$ 
of sections that converge exponentially at each puncture $z$ to a vector
in the corresponding vector space $V_z$. 

Denote the spectrum of an asymptotic operator $\mathbf{A}_z$ as
above by $\sigma(\mathbf{A}_z)$. The operator $\mathbf{A}_z$ is of
the form $-i\frac{d}{dt} - A_z(t)$, for some path $A_z(t)$ of
symmetric matrices. If $\delta\in \R$ is such that $\mathbf A_z +
\delta$ is a non-degenerate operator, then it has an associated
path of symplectic matrices, given by the fundamental matrix of the
ODE $\frac{d}{dt}x = i (A_z(t)-\delta) x$.  This path has a
Conley--Zehnder index, that we denote by $CZ(\mathbf{A}_z + \delta)$
(see also \cite{DiogoLisiSplit}*{Section 5.2.1}).

\begin{theorem} \label{T:RiemannRochMB}
Let $\delta > 0$ be sufficiently small that 
for each puncture $z \in \Gamma \cup \{ \pm \infty \}$,
$[-\delta, \delta]\cap\sigma(\mathbf{A}_z) \subset \{ 0 \}$.
Let $\mathbf{V}$ be a collection of vector spaces, associating to each 
$z \in \Gamma \cup \{ \pm \infty \}$ the vector subspace $V_z \subset \ker \mathbf{A}_z$. 

Then, 
\[
D \colon W^{1,p, {\delta}}_{\mathbf V}(\dot S, \C) \to L^{p, \mathbf{\delta}}(\dot S, \Lambda^{0,1} T^*\dot S).
\]
is Fredholm, and its Fredholm index is given by
\[
\Ind(D) = \chi_{\dot S} + \sum_{z \in \Gamma_+} \big( \CZ( \mathbf{A}_z + \delta) + \dim(V_z)\big)
- \sum_{z \in \Gamma_-} \big(\CZ( \mathbf{A}_z + \delta ) + \codim(V_z)\big).
\]
\end{theorem}

The following will be useful in determining the relevant Conley--Zehnder indices in our setting. 

\begin{lemma}\label{L:spectrum}
Given a constant $C\geq 0$, the spectrum $\sigma(A_C)$ of the operator
$$
A_C:= - i \frac{d}{dt} - %
\begin{pmatrix} C & 0 \\
0 & 0  
\end{pmatrix} %
\colon W^{1,p}(S^1,\C) \to L^p(S^1,\C)
$$
is the set
 $$\left \{\frac{1}{2}\left(- C - \sqrt{ C^2 + 16 \pi^2 k^2}\right) \, | \, k\in \Z \right\} \cup \left \{\frac{1}{2}\left(- C + \sqrt{ C^2 + 16 \pi^2 k^2}\right) \, | \, k\in \Z \right\}.$$
If $\lambda$ is an eigenvalue associated to $k\in \Z$, then the winding number of the corresponding eigenvector is $|k|$ if $\lambda\geq 0$ and $-|k|$ if $\lambda\leq 0$. If $C=0$, then all eigenvalues have multiplicity 2. 
If $C>0$, then the same is true except for the eigenvalues $-C$ and $0$, corresponding to $k=0$ above, both of which have multiplicity 1.  
 
In particular, the $\sigma(A_0) = 2\pi\Z$ and the winding number of $2\pi k$ is $k$. 
\end{lemma}

\begin{corollary} \label{C:CZ computation}
Take $C\geq 0$ and $\delta>0$ such that $[-\delta,\delta] \cap \sigma(A_C) = \{0\}$. Then
$$
\CZ(A_C + \delta) = \begin{cases}
                     0 & \text{ if } C > 0 \\
                     -1 & \text{ if } C = 0 
                    \end{cases}
\qquad \text{and} \qquad
\CZ(A_C - \delta) = 1. 
$$

\end{corollary}

\subsection{From Floer cylinders to pseudoholomorphic curves} \label{S:ansatz_Floer_to_holo}

\label{Floer_to_holo}

We show that every Floer cylinder as in \eqref{Floer eq Y} 
can be obtained from a $J_Y$-holomorphic curve and a solution
to an auxiliary equation.

It is useful to first recall that for every integer $T \geq 1$,
there is a $Y$-family of Reeb orbits of period $T$, and a bijective
correspondence to the $Y$-family of 1-periodic Hamiltonian orbits
contained in $\{b\}\times Y \subset \R \times Y$, where $h'(\e^b)
= T$.  
A Reeb orbit $\gamma$ \defin{corresponds} to the Hamiltonian
orbit $t \mapsto (b,\gamma(Tt))$.  We say that two Reeb orbits of
the same period are the {\em same as unparametrized orbits} if they
differ by a constant translation on the domain.

\begin{proposition}
\label{Ansatz1}
Let $\Gamma\subset \R\times S^1$ be either the empty set or the singleton $\{P\}$. Suppose $\tilde v = (b, v)\colon \R\times S^1 \setminus \Gamma \rightarrow \R \times Y$
is the upper level of a split Floer cylinder as in Cases (1) and (2) of Proposition \ref{P:split SH cases}.
Then, there exists a pair $(\tilde u, {\bf e})$ consisting of a map $\tilde u = (a, u)\colon \R\times S^1 \setminus \Gamma \rightarrow
\R \times Y$ and of a map ${\bf e} = (e_1, e_2)\colon \R\times S^1\setminus \Gamma \rightarrow \R \times S^1$, with the
following properties:
\begin{enumerate}
\item The map $\tilde u$ is a finite Hofer energy $J_Y$-holomorphic curve.
\item If $\tilde v$ converges to a Hamiltonian orbit at $\pm \infty$, then $u$ converges to the 
	corresponding parametrized Reeb orbit $\gamma_\pm$ at $\pm\infty$, i.e.~$v(\pm \infty, t) = u(\pm \infty, t)$.
\item If $\Gamma = \{P\}$, then $u$ is asymptotic to the same unparameterized Reeb orbit that $v$ converges to at $P$, i.e.~the asymptotic limit is the same orbit up to a phase shift.
\item The map ${\bf e} \colon \R\times S^1\setminus \Gamma \to \R \times S^1$ satisfies the equation   
	\begin{equation} \label{E:cylinder0}
	de_1 -de_2 \circ i +  h'(\e^{b}) ds = 0 .
	\end{equation}
\item The original Floer solution is given by 
\[
(b, v) = (a + e_1, \phi_R^{e_2} \circ u )
\] 
	where $\phi_R^t \colon Y\to Y$ denotes Reeb flow for time $t\in S^1$. 
\end{enumerate}

\noindent 
The pair $(\tilde u,{\bf e})$ is unique, up to replacing $\tilde u = (a,u)$ with $(a+c_1,u)$ and replacing ${\bf e}= (e_1,e_2)$ with $(e_1 - c_1, e_2)$, for some constant $c_1\in \R$.

Suppose now that $\Gamma = \emptyset$ and that $\tilde v$ is the upper level of
a split Floer cylinder as in Case (3) of Proposition \ref{P:split SH cases}.
Then, there exists a pair $(\tilde u, \bf e)$ as above, with the following difference:
\begin{enumerate}
 \item[(2')] If $\tilde v$ converges to a Hamiltonian orbit at $+\infty$ and to a Reeb orbit at $-\infty$, then $u$ converges to the corresponding Reeb orbit $\gamma_+$ at $+\infty$ and to the same Reeb orbit $\gamma_-$ at $-\infty$. 
\end{enumerate}

\noindent
The pair $(\tilde u,{\bf e})$ is again unique up to an $\R$-action.

\end{proposition}

In the following, we will construct the function $\mathbf{e}$ in
Proposition \ref{Ansatz1} by studying the differential equation
\eqref{E:cylinder0}.

To this end, we need to construct spaces of maps $\mathbf{e} \colon \R \times S^1 \setminus \Gamma \to \R \times S^1$.
These will be subspaces of  $W^{1,p}_{\text{loc}}(\R\times S^1\setminus \Gamma, \R\times S^1)$ satisfying different types of asymptotic conditions.
The maps $\ee$ we consider will induce trivial maps of fundamental groups.
It will be more convenient therefore to consider the lift to the universal cover
$\tilde \ee \colon \R \times S^1 \setminus \Gamma \to \C$, and study Equation \eqref{E:cylinder0} on this space. This lift is well-defined up to an additive constant.

As we will see below, all the punctures have the same asymptotic operator, namely
\begin{equation} \label{asymptotic operator A}
 A := - i \frac{d}{dt} \colon W^{1,2}(S^1,\C) \to L^2(S^1,\C).
\end{equation}
Note that $A$ is self-adjoint as a partially defined operator $L^2(S^1,\C) \to L^2(S^1,\C)$.
Lemma \ref{L:spectrum} describes the spectrum of $A$. For the remainder of this section, we will fix $0 < \delta < 2\pi$, which is smaller than the absolute value of every non-zero eigenvalue of $A$. 
We will use this $\delta$ below when defining weighted Sobolev spaces.

We introduce some notation first.  For a given $z \in \R \times S^1$, or $z =
\pm \infty$, we fix $\mu_z$ to be a bump function supported near $z$ and
identically 1 near $z$.  
Recall that the space of functions $W^{1,p,\delta}_{(0,\C)}( \R \times S^1, \C)$ consists of complex-valued functions
exponentially decaying to $0$ at $-\infty$ and exponentially decaying
to an unspecified constant at $+\infty$.  In the case of a punctured
cylinder, $W^{1,p,\delta}_{(0, \C; \C)}(\R \times S^1 \setminus \{ P \}, \C)$ 
denotes the space of functions exponentially decaying to $0$
at $-\infty$, decaying to a free constant at $P$ and decaying to a free
constant at $+\infty$.

In the case with $\Gamma= \{P\}$, we fix a parametrization 
\begin{align*}
\varphi_P \colon (-\infty, -1]\times S^1 &\to \R\times S^1 \setminus \Gamma \\
(\rho,\eta) &\mapsto P + \e^{2\pi(\rho + i \eta)}
\end{align*}
of a neighborhood of $P$, under the identification of $\R\times S^1$ with $\C/(z\sim z + i)$. 
Recall that the measure on $\R \times S^1 \setminus \Gamma$ and the norm on derivatives 
used for defining $W^{1,p}$ and $L^p$
comes from these cylindrical coordinates near $P$.
In particular, we make the following observation that will be used multiple times,
\begin{equation} \label{E:ds_is_exponential}
\varphi_P^* ds = 2\pi \e^{2\pi \rho} \cos(2\pi \theta) d\rho - 2\pi \e^{2\pi \rho} \sin(2\pi\theta) d\theta
\end{equation}
In particular then, $| \varphi_P^*ds | = C \e^{2\pi \rho}$.

We will consider solutions to equation \eqref{E:cylinder0} in three different spaces, corresponding to 
Cases 1, 2 and 3 appearing in Proposition \ref{P:split SH cases} (recall Figures \ref{case_1_fig}, \ref{case_2_fig} and \ref{case_3_fig}). 

    If $\tilde v = (b, v)$ is the upper level of a split Floer cylinder as in Cases
    1,2 or 3 of Proposition \ref{P:split SH cases}, it converges at
    $+\infty$ to a Hamiltonian orbit of the form $(b_+, \gamma(T_+ t))$, where
    $\gamma$ is a parametrization of a simple (embedded) Reeb orbit in $Y$
    (and thus satisfying $h'(\e^{b_+}) = T_+$). Similarly, at $-\infty$,
    $\tilde v$ converges either to a Hamiltonian trajectory of the form $(b_-,
    \gamma(T_- t))$ or $v$ converges to $\gamma(T_- t)$ at $-\infty$. In the
    following, these periods $T_\pm$ will be used to define the relevant
    function spaces.

\begin{itemize}
\item[Case 1:] $\Gamma = \emptyset$ and both asymptotic limits of $\tilde v$ are Hamiltonian orbits. 

We say $\tilde \ee \in \tilde \X^1$ if
\[
\tilde \ee +  (T_+ s + i 0) \, \mu_{+\infty} +  (T_- s + i 0) \, \mu_{-\infty} \in W^{1,p, \delta}_{(\C; \C)} (\R \times S^1, \C).
\]

\item[Case 2:] $\Gamma = \{ P\}$ and both asymptotic limits of $\tilde v$ are Hamiltonian orbits. 
We say $\tilde \ee \in \tilde \X^2$ if
\[
\tilde \ee +  (T_+ s + i 0) \, \mu_{+\infty} +  (T_- s + i 0) \, \mu_{-\infty} \in W^{1,p,\delta}_{(\C, \C; \C)}( \R \times S^1 \setminus \{ P \}, \C )
\]

\item[Case 3:] $\Gamma = \emptyset$ and the $+\infty$ asymptotic limit of $\tilde v$ is a Hamiltonian orbit,
whereas at $-\infty$, it converges to an orbit cylinder over a Reeb orbit in $\{ -\infty \} \times Y$ (this is (2') in Proposition \ref{Ansatz1}). 
We say $\tilde \ee \in \tilde \X^3$ if 
\[
\tilde \ee +  (T_+ s + i 0) \, \mu_{+\infty} \in W^{1,p\delta}_{(\C; \C)}( \R \times S^1, \C).
\]
\end{itemize}
We then define $\X^1, \X^2, \X^3$ to be the set of functions obtained as compositions of functions in $\tilde \X^1, \tilde \X^2, \tilde \X^3$ with the projection map $\C \to \R \times S^1$.

    Notice that $\tilde \X^1, \tilde \X^2, \tilde \X^3$ are described as 
affine spaces modelled on $W^{1,p,\delta} \oplus \C^k$, where $k=2$ in Cases 1 and 3, and $k=3$ in Case 2, so $\X^i$ are Banach manifolds.

\begin{remark}
The summands $(T_\pm s + i0) \, \mu_{\pm\infty}$ keep track of the fact that the Floer trajectory $\tilde v=(b,v)$ and the pseudoholomorphic curve $\tilde u=(a,u)$ in Proposition \ref{Ansatz1} have different asymptotic behavior at $\pm\infty$, and that $e_1 = b-a$. 

Observe that in all three cases, we impose $\delta$-decay to unspecified constants at $\pm\infty$ (and possibly $P$). 

\end{remark}

Denote the spaces of solutions to \eqref{E:cylinder0} in $\tilde \X^j$
by $\tilde {\mathcal{C}}^{j}(\tilde v)$, where $j\in \{1,2,3\}$, and let
$\mathcal C^j(\tilde v)$ be the corresponding space of maps to $\R \times S^1$. We can
write $\tilde {\mathcal C}^j(\tilde v)$ as the preimage of zero under
the operator
\begin{align*}
 \tilde \X^j &\to L^{p,\delta}(T^*( \R\times S^1 \setminus \Gamma)) \\
\tilde \ee = (e_1,e_2) &\mapsto d e_1 - de_2 \circ i + h'(\e^b) ds.
\end{align*}
Observe that this map descends to a map $\X^j \to L^{p,\delta}(T^*(\R \times S^1
\setminus \Gamma))$. We then have that $\mathcal C^j(\tilde v)$ is the preimage
of zero under this quotient operator.

It will be useful to consider the corresponding linearized operators. 
Let $\ee \in \mathcal C^j(\tilde v)$. The corresponding linearized operator is 
\begin{equation} \label{E:cylLinear0}
\begin{aligned} 
D^j \colon T_{\bf e} \X^j &\to L^{p,\delta}(T^*( \R \times S^1 \setminus \Gamma) ) \\
E = E_1 + i E_2 &\mapsto dE_1 - d E_2 \circ i
\end{aligned}
 \end{equation}
where 
in Cases 1 and 3, we have $T_{\bf e} \X^j = W^{1,p,\delta}_{(\C; \C)}(\R\times S^1, \C)$ and 
and in Case 2, we have $T_{\bf e} \X^2 = W^{1,p,\delta}_{(\C, \C; \C)}(\R\times S^1 \setminus \Gamma, \C)$.

Evaluating \eqref{E:cylLinear0} at $\partial_s$ and $\partial_t$ and rearranging terms, we get operators with values in $L^{p,\delta}(\R\times S^1 \setminus \Gamma, \C)$ given by 
\begin{equation} \label{E:cylLinear0 s}
\frac{\partial E}{\partial s} + i \frac{\partial E}{\partial t}.
\end{equation} 
By an abuse of notation, denote these operators also by $D^j$. Observe that the
asymptotic operators to $D^j$ at all punctures are given by 
$-i\frac{d}{dt}$.

\begin{lemma} \label{L:cylTransversality1}
Fix $\e \in \mathcal C^{j}(\tilde v)$, for $j\in \{1,2,3 \}$. Then, the linearized operator $D^{j}$ has index 2, is surjective and its kernel includes the span of the generator of the $S^1$-action on the target $\R\times S^1$.
As a consequence, the space $\mathcal C^j(\tilde v)$ has virtual dimension 2. 
\end{lemma}

\begin{proof}
Since the operator $D^j$ in \eqref{E:cylLinear0} is just the standard Cauchy--Riemann operator, a more elementary proof of this result is possible, but we illustrate here the methods also used in the more technical setting of the next section. 

We start with the Cases $j=1$ or $3$. Fix ${\bf e}$ as in the statement. Recall that the asymptotic operator is $A = - i\frac{d}{dt}$ at all punctures. 
We need to compute the Conley--Zehnder index of the perturbed asymptotic operator $A+\delta$, which is $-1$ by Corollary \ref{C:CZ computation}. 
Therefore using the fact that all vector spaces in $\mathbf{V}$ are the kernels of the corresponding asymptotic operators, which we identified with $\C$, Theorem \ref{T:RiemannRochMB} implies that 
$$
\Ind D^{j} = (-1+2)-(-1) = 2.
$$
Following the notation of Wendl \cite{WendlAutomatic}*{Equation (2.5)}, we have 
$$
c_1(E,l,\textbf{A}_\Gamma) = \frac{1}{2}(2-2) = 0 < 2 = \Ind \widehat D^{j}
$$
and surjectivity and the statement about the kernel follow from \cite{WendlAutomatic}*{Proposition 2.2}
together with \cite{DiogoLisiSplit}*{Lemma 5.20}.

Case 2 is slightly different, since there is an additional puncture at $P$. Its asymptotic operator is again $A$. The Euler characteristic of the domain is now $-1$. 
Therefore, 
$$
\Ind D^{2} = -1 + \big((-1+2)-(-1)-(-1)\big) = 2,
$$
$$
c_1(E,l,\textbf{A}_\Gamma) = \frac{1}{2}(2-2) = 0 < 2 = \Ind \widehat D^{2}
$$
and we can once more apply \cite{WendlAutomatic}*{Proposition 2.2} together with
\cite{DiogoLisiSplit}*{Lemma 5.20}. 
\end{proof}

\begin{lemma} \label{smooth_solutions_Ansatz0}
    Let $\ee \in \CC^j$, $j=1, 2, 3$. Then, $\ee$ admits a smooth extension 
    $\R \times S^1 \to \R \times S^1$ that satisfies 
    Equation \eqref{E:cylinder0}.
\end{lemma}
\begin{proof}
    Let $\ee \in \CC^j$ be a solution. Then, in Cases 1 and 3, the result
    follows immediately from elliptic regularity and the fact that $h'(\e^b)$
    is a smooth function on the cylinder.

    For Case 2, observe that $h'(\e^r)$ vanishes for $r << 0$. 
    Therefore, $h'(\e^b)$ vanishes in a neighbourhood of the puncture $P$,
    and thus extends as a smooth function on the cylinder. Furthermore, 
    $\ee$ is holomorphic in this same punctured neighbourhood, and the
    requirement it be in $W^{1,p, \delta}_\C( \R^- \times S^1)$ forces it to
    have a continuous and thus smooth extension across the puncture.

    It follows therefore that $\ee$ is smooth in Case 2 as well.
\end{proof}

\begin{lemma} \label{L:aprioriMarker1}
Let $\tilde \ee = (e_1, \tilde e_2) \colon \R\times S^1 \to \C$ 
be a lift of a solution $\e =(e_1, e_2) \in
\CC^j$. 
Then, for every $t\in S^1$
\[
\lim_{s\to -\infty } \tilde e_2(s,t) = \lim_{s\to \infty } \tilde e_2(s,t).
\]
\end{lemma}
\begin{proof}
    By Lemma \ref{smooth_solutions_Ansatz0}, it follows that $\tilde \ee$ is
    smooth and satisfies the equation
    \[
        de_1 - d \tilde{e_2} \circ i + h'(\e^b) ds =0.
    \]

For each $c >> 1$, let $S_c := [-c,c] \times S^1 \subset \R \times S^1$. 
Then, 
\begin{align*}
\int_{\partial S_c} \tilde e_2 \, dt 
&= \int_{S_c} d\tilde e_2 \wedge dt = \int_{S_c} de_2 \wedge dt \\
&= \int_{S_c} de_2 \circ i \wedge dt\circ i \\
&= \int_{S_c} ( de_1 + h'(\e^{b})ds ) \wedge ds \\
&= \int_{S_c} de_1 \wedge ds \\
&= \int_{\partial S_c} e_1 \, ds =0
\end{align*}
since $ds|_{\partial S_c} = 0$. Hence, we obtain 
\[
\int_{\{c\}\times S^1} \tilde e_2(c, t) dt 
	 = \int_{\{-c\}\times S^1} \tilde e_2(-c, t) dt.
\]
The result now follows from taking a limit as $c \to +\infty$.
\end{proof}

\begin{lemma} \label{L:inhomoCRexistence}
Equation \eqref{E:cylinder0} has a solution in $\X^j$ for $j\in \{1,2,3\}$, which is unique up to adding a constant in the target $\R\times S^1$.
\end{lemma}

\begin{proof}
We will instead prove the result for the lifts $\tilde \ee \in \tilde \X^j$, solving the same differential equation.
Notice that any solution $\ee$ admits a $\Z$ family of lifts. It suffices then to show that the solutions
$\tilde \ee = (e_1, \tilde e_2)$ are unique up to adding a constant in $\C$.

Write \eqref{E:cylinder0} as an inhomogeneous Cauchy--Riemann equation on $\R\times S^1 \setminus \Gamma$:
\begin{equation}\label{E:inhomoCR}
d{e}_1 - d\tilde{e}_2 \circ i = -h'(\e^b)ds.
\end{equation}

We will consider Cases 1, 2 and 3. Let $\nu_j: \R \to \R$ be smooth functions such that 
\begin{equation*}
\begin{cases}
 \nu_j(s) = T_- s \text{ for } s<<0 \text{ and } \nu_j(s) = T_+ s \text{ for } s>>0 & j = 1, 2 \\
 \nu_j(s) = 0 \text{ for } s<<0 \text{ and } \nu_j(s) = T_+ s \text{ for } s>>0 & j = 3 \\
\end{cases}
\end{equation*}
  
Instead of solving equation \eqref{E:inhomoCR}, we will consider the equivalent problem of finding ${\bf g} = (g_1, g_2 ) = (e_1 + \nu_j, \tilde{e}_2)$ such that%
\[
{\bf g} \in 
\begin{cases}
 W^{1,p,\delta}_{(\C; \C)}(\R\times S^1, \C), \quad &\text{in cases }  j = 1, 3 \\
 W^{1,p,\delta}_{(\C, \C; \C)}(\R\times S^1 \setminus \Gamma, \C),  &\text{in case } j = 2
\end{cases}
\]
solves the equation
\begin{equation} \label{E:inhomoCRg}
dg_1 - dg_2 \circ i =-h'(\e^b) ds + d\nu_j. 
\end{equation}
Notice that $\mathbf{g}$ is taken in the function space $T_{\ee}\X^j$.

Let $K = -h'(\e^b) ds + d\nu_j $. 
First, recall that $b$ converges exponentially fast to $b_\pm$. 
In our setting of Morse--Bott degeneracy, exponential convergence follows from 
\cite{BourgeoisThesis} and \cite{SFTcompactness}*{Appendix A}. Then, by
\cite{SiefringAsymptotics}, the exponential rate of convergence is given by
the spectral gap, which is larger than $\delta$ by construction.

In Cases 1 and 3, the fact that $K \in L^{p,\delta}(T^*(\R\times S^1))$ follows immediately from the exponential rate of convergence of $b$ at $\pm \infty$ to $b_\pm$, and $T_\pm = h'(\e^{b_\pm})$. 

In Case 2, the fact $K \in L^{p, \delta}(T^*(\R \times S^1 \setminus \{ P \})$ follows from this and also from Equation \eqref{E:ds_is_exponential} so that $\nu'_2(s)ds$ has exponential decay near $P$.

According to Lemma \ref{L:cylTransversality1}, the linearized operators \eqref{E:cylLinear0 s} have index 2 and are surjective.
Their kernels have dimension 2 and consist of constant functions. 
Let now ${\bf g} \in T_{\bf e} \X^j$, with $T_{\bf e} \X^j$ defined as in \eqref{E:cylLinear0}, be such that 
$$
dg_1 - dg_2 \circ i = K.
$$

Any two solutions differ then by an element of the kernel of this linear operator, 
and thus differ by a constant.
From this, we obtain the same statement for $\tilde \ee$ and hence for $\ee$.
 \end{proof}

\begin{proof}[Proof of Proposition \ref{Ansatz1}]
 Using the fact that $X_H = h'(\e^r) R$, the Floer equation \eqref{Floer eq Y} satisfied by $\tilde v = (b,v)$ is equivalent to:
\begin{equation}
\begin{cases}
db - v \pb \alpha \circ i + h'(\e^b) ds = 0 \\
d\pi_Y \circ dv + J \circ d\pi_Y \circ dv \circ i = 0
\end{cases}
\end{equation}
By Lemma \ref{L:inhomoCRexistence}, we have a solution 
${\bf e}$ to \eqref{E:cylinder0}. Now, $a := b - e_1$ and $u := \phi_R^{-e_2} \circ v$ are such that
\begin{align*}
da - u \pb \alpha \circ i 
	&= db - de_1 - v \pb \alpha \circ i - de_2 \circ i\\
	&=db - v \pb \alpha \circ i + h'(\e^b) ds\\
	& = 0,\\
d\pi_Y \circ du + J \circ d\pi_Y \circ du \circ i 
	&= d\pi_Y \circ dv + J \circ d\pi_Y \circ dv \circ  i \\
	& = 0.
\end{align*}
So, $\tilde u:=(a,u)$ is $J_Y$-holomorphic. The asymptotic constraints on the spaces $\X^i$ are such that $\tilde u$ converges to Reeb orbits at $\pm\infty$ (and at $P$, if $\Gamma = \{P \}$).
The uniqueness statement in Lemma \ref{L:inhomoCRexistence}, together with Lemma \ref{L:aprioriMarker1}, imply that there is an $\R$-family of solutions ${\bf e} = (e_1,e_2)\in\X^j$, such that 
$$
\lim_{s\to \pm\infty} e_2(s,t) = 0\in S^1,
$$
where distinct solutions differ by adding a constant to $e_1$.
This implies the existence statements in Proposition \ref{Ansatz1}, and uniqueness up to an $\R$-action.
\end{proof}

\subsection{From pseudoholomorphic curves to Floer cylinders} \label{S:ansatz_holo_to_Floer}

\label{holo_to_Floer}

Proposition \ref{Ansatz1} implies that every split Floer cylinder $\tilde v \colon \R\times S^1\setminus \Gamma \to \R\times Y$ has an underlying $J_Y$-holomorphic curve $\tilde u \colon \R\times S^1\setminus \Gamma \to \R\times Y$. 
We will show that every $J_Y$-holomorphic curve $\tilde u \colon \R\times S^1\setminus \Gamma \to \R\times Y$ underlies a split Floer cylinder.

\begin{proposition}\label{Ansatz2}
 Let $\Gamma\subset \R\times S^1$ be either the empty set or the singleton $\{P\}$. Suppose $\tilde u = (a, u)\colon \R\times S^1 \setminus \Gamma \rightarrow \R \times Y$
is a non-constant, finite Hofer energy $J_Y$-holomorphic curve, with a negative
puncture at $\Gamma$ if $\Gamma \ne \emptyset$. Let $\gamma_\pm$ denote the
Reeb orbits to which $u$ is asymptotic at $\pm \infty$.
Then, there exists a pair $(\tilde v, {\bf e})$ consisting of a map $\tilde v = (b, v)\colon \R\times S^1 \setminus \Gamma \rightarrow
\R \times Y$ and of a map ${\bf e} = (e_1, e_2)\colon \R\times S^1 \rightarrow \R \times S^1$, with the
following properties:
\begin{enumerate}
\item The map $\tilde v$ solves the Floer equation \eqref{Floer eq Y}.

\item At each of $\pm \infty$, the map $\tilde v$ converges to the Hamiltonian orbit to which 
    $\gamma_\pm$ corresponds.
\item If $\Gamma = \{P\}$, then $v$ is asymptotic to the same unparameterized Reeb orbit that $u$ converges to at $P$.
\item The map ${\bf e} \colon \R\times S^1\setminus \{P\} \to \R \times S^1$ satisfies the equation
	\begin{equation} \label{E:cylinder1}
	de_1 -de_2 \circ i +  h'(\e^{a+e_1}) ds = 0.
	\end{equation}
\item The original $J_Y$-holomorphic curve is given by 
\[
(a, u) = (b - e_1, \phi_R^{-e_2} \circ v).
\] 
\end{enumerate}

\noindent 
The pair $(\tilde v,{\bf e})$ with these properties is unique.

In the case that $\Gamma = \emptyset$ and $\tilde u(s,t)=(Ts + c,\gamma(Tt))$ is a trivial cylinder, for some Reeb orbit $\gamma$ of period $T$, there also exists a pair $(\tilde v, \bf e)$ as above, except replacing property (2) by
\begin{enumerate}
 \item[(2')] $\tilde v$ converges to a Hamiltonian orbit whose underlying Reeb orbit is $\gamma$ at $+\infty$, and $v$ converges to $\gamma$ at $-\infty$. 
\end{enumerate}
The pair $(\tilde v,{\bf e})$ is now unique only up to replacing $b(s,t)$ with
$\overline b(s,t)=b(s+s_0,t)$, for some constant $s_0\in \R$, and replacing $e_1$
with $\overline{e}_1 = \overline{b}-a$.

\end{proposition}

Equation \eqref{E:cylinder1} has a solution ${\bf e}= (e_1,e_2)$ iff there is a solution ${\bf f} = (f_1, f_2) = (e_1 + a,e_2)$ to the equation 
\begin{equation} \label{E:cylinder}
\begin{aligned}
	{\bf f}=&(f_1,f_2) \colon \R \times S^1 \setminus \Gamma  \rightarrow \R \times S^1 \\
	&df_1 - df_2 \circ i + h'(\e^{f_1}) ds - da = 0,
\end{aligned}
\end{equation}
so we will study this equation instead. As we will see below equation \eqref{E:cylLinear2 s}, the linearized operator associated to \eqref{E:cylinder} has asymptotic operators of the form 
$$
A_C = -i\frac{d}{dt} - \begin{pmatrix}
                       C & 0 \\
                       0 & 0
                      \end{pmatrix} \colon W^{1,p}(S^1,\C) \to L^{p}(S^1,\C)
$$
for some $C\geq 0$. Lemma \ref{L:spectrum} implies that the kernel of $A_C$ has real dimension 2 if $C = 0$ (given by constant functions $c_1 + i c_2 \in \C$) and 1 if $C\neq 0$ (given by constant functions $i c_2 \in i\R$). 

Recall that we write 
$W^{1,p,\delta}_{(0, \C; i\R)}(\R \times S^1 \setminus \{ P \}, \C)$ to denote
maps converging exponentially fast to $0$ at $-\infty$, to an arbitrary complex constant at $P$ and to an imaginary constant at $+\infty$, and write $W^{1,p, \delta}_{(0, i\R)}(\R \times S^1, \C)$ to denote maps converging exponentially fast to $0$ at $-\infty$ and to an arbitrary imaginary constant at $+\infty$.
Recall that we have chosen cylindrical coordinates near $P$, $\varphi \colon
(-\infty, -1] \times S^1 \to \R \times S^1$ by $\varphi(\rho, \theta) = P +
\e^{2\pi(\rho + i \theta)}$, and also a bump function $\mu_P$ that is 
identically $1$ near $P$ and supported in the image of $\varphi$.

Similarly to the previous section, we will study solutions to (\ref{E:cylinder}) in $W^{1,p}_{\text{loc}}(\R\times S^1\setminus \Gamma, \R\times S^1)$ satisfying three types of asymptotic conditions.%

Notice that if $\tilde u$ is a $J_Y$-holomorphic [punctured] cylinder as in
Proposition \ref{Ansatz2}, it converges at $\pm \infty$ 
to closed Reeb orbits in $\R \times Y$, of period $T_\pm$. 
Let $b_\pm \in \R$ be the unique solutions to $T_\pm =
h'(\e^{b_\pm})$. (Recall that such solutions exist and are unique by the
admissibility conditions imposed on the Hamiltonian in Definition \ref{def:J_shaped}. Furthermore, we then have $b_\pm > \log 2$.)

\begin{itemize}
\item[Case 1:] $\Gamma = \emptyset$. Define the space $\Y^1$ by the condition $\ff \in \Y^1$ if it admits a lift $\tilde \ff \in W^{1,p}_\loc(\R \times S^1, \C)$ with 
\[
\tilde{\ff} - (b_+ + i 0) \, \mu_{+\infty} -  (b_- + i 0) \, \mu_{-\infty} \in W^{1,p,\delta}_{(0, i\R)}(\R \times S^1, \C)
\]

\item[Case 2:]
 $\Gamma = \{ P\}$. Define the space $\Y^2$ by the condition that $\ff \in \Y^2$ if it admits a lift $\tilde \ff \in W^{1,p}_\loc(\R \times S^1, \C)$ with 
 \[
 \tilde \ff -  (b_+ + i 0) \, \mu_{+\infty} -  (b_- + i 0) \, \mu_{-\infty} - \mu_P \, a \in W^{1, p, \delta}_{(0, \C; i\R)}( \R \times S^1 \setminus \{ P \}, \C).
 \]

\item[Case 3:] $\Gamma = \emptyset$. Define the space $\Y^3$ by the condition that $\ff$ admits a lift $\tilde \ff$ with
\[
\tilde\ff -  (b_+ + i 0) \, \mu_{+\infty} \in W^{1,p, \delta}_{(\C; i\R)}( \R \times S^1, \C)
\]
(this corresponds to (2') in Proposition \ref{Ansatz2}).

\end{itemize}
As before, we notice that the spaces $\Y^i$ are Banach manifolds modelled
    on $W^{1,p,\delta} \oplus \R^k$, where $k$ depends on which of the 3 cases
we are considering.

\begin{remark}
 As in the previous section, the first part of Proposition \ref{Ansatz2}
 is related to Cases 1 and 2 (see Figures \ref{case_1_fig} and
 \ref{case_2_fig}, respectively).  The last part of the Proposition,
 where $\tilde u$ is assumed to be a trivial orbit cylinder, is
 related to Case 3 (see Figure \ref{case_3_fig}).

 It is important to observe that the asymptotic conditions imposed in this section are different from the ones that were imposed in the previous section. This will have an effect on the indices of the relevant Fredholm problems. 
Just to give an example, in Case 1 we are considering functions that converge at
$+\infty$ to an unspecified constant in $\{b_+\}\times S^1 \subset \R\times S^1$, 
and that converge at $-\infty$ to $(b_-,0)\in \R\times S^1$. 
 \end{remark}

We will consider auxiliary parametric versions of Equation
\eqref{E:cylinder} in order to deform the solutions to \eqref{E:cylinder}
to solutions of a PDE for which we can explicitly describe the
solution space.  The deformation will give an identification of the
solution spaces.

For Case 1, we consider the map
\begin{equation} \label{E:parametric operator1}
\begin{aligned}
\mathcal F \colon \mathcal Y^1 \times [0,1] &\to L^{p,\delta}(\R\times S^1,(\R^2)^*) \\
({\bf f},\tau) = ((f_1,f_2),\tau) &\mapsto df_1 - df_2 \circ i + h'(\e^{f_1}) ds - \tau da - (1-\tau) d\nu
\end{aligned}
\end{equation}
where $\nu \colon \R\to \R$ is a smooth function such that $\nu(s) = T_- s$ for $s <<0$ and $\nu(s) = T_+ s$ for $s >>0$ (just like $\nu_1$ in the proof of Lemma  \ref{L:inhomoCRexistence}). We will study solutions to equation
\begin{equation} \label{E:parametric cylinder1}
\begin{aligned}
\mathcal F({\bf f},\tau) = 0,
\end{aligned}
\end{equation}
which interpolates between (\ref{E:cylinder}) and the equation
\begin{equation} \label{E:deformed cylinder1}
df_1 - df_2 \circ i + h'(\e^{f_1}) ds - d\nu= 0.
\end{equation}
The latter has the advantage that $d\nu$ is independent of $t$. Denote the space of solutions to (\ref{E:parametric cylinder1}) in $\mathcal Y^1$ for fixed $\tau$ by $\mathcal C^{1,\tau}(\tilde u)$.

In Case 2, it will be more convenient to study functions $\gg=(g_1,g_2)=(f_1 - \mu_P \, a, f_2)$.
We think of ${\bf g}$ as an element of the subspace $\hat{ \Y^2}
\subset W^{1,p}_\text{loc}(\R\times S^1\setminus \Gamma,\R\times S^1)$, with
$\gg \in \hat \Y^2$ if it has a lift $\tilde \gg$ with
\[
\tilde \gg -  (b_+ + i 0) \, \mu_{+\infty} -  (b_- + i 0) \, \mu_{-\infty} \in W^{1,p,\delta}_{(0, \C; i\R)}( \R \times S^1 \setminus \{ P \}, \C).
\]
We will consider the map 
\begin{equation}  \label{E:parametric operator2}
\begin{aligned}
\mathcal G \colon \hat{\Y^2} \times [0,1] &\to L^{p,\delta}(\R\times S^1 \setminus \Gamma,(\R^2)^*) \\
({\bf g},\tau) = ((g_1,g_2),\tau) &\mapsto dg_1 - dg_2 \circ i + h'(\e^{g_1+\tau
\mu_P a}) ds - \tau d((1-\mu_P)a) - (1-\tau) d\nu
\end{aligned}
\end{equation}
where $\nu \colon \R\to \R$ is as above. 

We will study solutions to 
\begin{equation} \label{E:parametric cylinder2}
\begin{aligned}
\mathcal G({\bf g},\tau) = 0.
\end{aligned}
\end{equation}
This equation interpolates between the equivalent of equation \eqref{E:cylinder} for $\bf g$:
$$
dg_1 - dg_2 \circ i + h'(\e^{g_1+ \mu_P a}) ds - d((1-\mu_P)a) = 0
$$
and equation \eqref{E:deformed cylinder1} applied to $\bf g$:
$$
dg_1 - dg_2 \circ i + h'(\e^{g_1}) ds - d\nu= 0,
$$
The latter has again the advantage that $d\nu$ is independent of $t$. Denote the space of solutions to (\ref{E:parametric cylinder2}) in $\hat{\mathcal Y^2}$ for fixed $\tau$ by $\mathcal C^{2,\tau}(\tilde u)$. 

Also, denote the space of solutions to (\ref{E:cylinder}) in $\mathcal Y^3$ by $\mathcal C^{3}(\tilde u)$.  In Case 3, we will study solutions directly and do not require any deformation argument.

We need the following smoothness in order to apply the implicit function
theorem, which we will need in the proof of Proposition \ref{Ansatz2}.
\begin{lemma} \label{L:smooth parameter}
 The nonlinear operators $\mathcal F$ and $\mathcal G$ are $C^1$ in 
 $\mathcal Y^1 \times [0,1] \to L^{p,\delta}(\R\times S^1,(\R^2)^*)$ 
 and 
 $\hat{\mathcal Y^2} \times [0,1] 
 \to  L^{p,\delta}(\R\times S^1\setminus \{ P \},(\R^2)^*)
 $, respectively. 
\end{lemma}
\begin{proof}
    Notice that $\mathcal F( \mathbf{f}, \tau)$ is a Floer-type differential
    operator with an inhomogeneous term depending on $\tau$. 
    The claimed result holds for $\mathcal F$ since $-da + \nu'(s) \, ds$ is in
    $L^{p, \delta}$, which follows from the fact that $a(s,t)$ converges 
    exponentially fast, together with its derivatives, to $T_\pm s + c_\pm$ at $\pm \infty$. 

    In order to show the result for $\mathcal G( \mathbf{g}, \tau)$, it suffices
    to show that the following map is $C^1$, since the remaining terms in $\mathcal G$ are $C^1$ by standard arguments.
    \begin{equation} \label{E:smoothness}
    \begin{aligned}
       \mathcal H \colon W^{1,p, \delta}_{(0, \C; i\R)}( \R \times S^1 \setminus \Gamma, \C) 
                    &\to L^{p, \delta}( T^*(\R \times S^1 \setminus \Gamma))  \\
        (g_1, g_2, \tau) &\mapsto h'( \e^{g_1 + \tau \mu_P a}) \, ds.
    \end{aligned}
    \end{equation}
    First, we restrict our attention to a $W^{1,p,\delta}_{(0, \C;
    i\R)}$-ball of radius $R$. We will show the map is $C^1$ smooth
    on each such ball by showing its second derivative is bounded
    by a constant that depends on $R$. Exhausting the total space
    by such balls will prove the result.

    Notice that the $W^{1,p,\delta}$ bound implies a uniform $C^0$
    bound on $g_1$. Also notice that $\mu_P a$ is bounded above.
    Thus, $g_1 + \tau \mu_Pa$ is bounded above (where the bound
    depends on $R$). 
    Hence, there exists a $C$ (depending on $R$)
    so that 
    \begin{align*}
        &0 \le h'(\e^{g_1+ \tau \mu_P a}) \le C,\\
        &0 \le h''(\e^{g_1+ \tau \mu_P a})\e^{g_1+ \tau \mu_P a} \le C\\
        &| h'''(\e^{g_1+ \tau \mu_P a})\e^{2(g_1+ \tau \mu_P a)}|\le C.  
    \end{align*}

    The weak derivative of $\mathcal H$ is given by 
    \[
    d{\mathcal H}(\mathbf{g}, \tau)(\mathbf G, \tau') 
    	= G_1 h''(\e^{g_1+ \tau \mu_P a})\e^{g_1+ \tau \mu_P a} \, ds 
			+ \tau' h''(\e^{g_1+ \tau \mu_P a})\e^{g_1+ \tau \mu_P a} \mu_P a \, ds
    \]
    where $\mathbf{G} = (G_1, G_2)$ is a variation of $\mathbf{g} = (g_1,
    g_2)$ and $\tau'$ is a variation of $\tau$.
Notice that this expression is in $W^{1,p,\delta}$ away from
    $P$. Near $P$, in the cylindrical coordinates $(\rho, \theta)$,
    we note that this is bounded by 
    \[
    C (|G_1| + |\tau'| (T_P|\rho| + D)) \e^{2\pi \rho} (|d\rho| + |d\theta|) \]
    for suitable constants $C, D$ and where $T_P$ is the period of the Reeb orbit to 
    which $(a, u)$ converges at $P$. 
    From this, we obtain that $d{\mathcal H}$ is uniformly bounded on each ball of radius $R$, 
    showing $\mathcal H$ is Lipschitz on each ball of radius $R$, and hence $\mathcal H$ is continuous.
    
    By a similar argument, we estimate the second derivative to show $C^1$ smoothness of $\mathcal H$. 
    The $C^1$-smoothness of $\mathcal G$ then follows.    
\end{proof}

We now discuss the linearized operators associated to $\mathcal F$ and $\mathcal G$, for fixed values of $\tau$. In Case 1, writing a solution to \eqref{E:parametric cylinder1} as ${\bf f} = f_1 + i f_2$, the linearization of the operator $\mathcal F^\tau$ is
\begin{equation} \label{E:cylLinear1}
\begin{aligned}
D^{1,\tau} \colon T_{\bf f}\mathcal Y^1 &\to L^{p,\delta}(\R\times S^1, (\R^2)^*) \\
F = F_1 + iF_2 &\mapsto dF_1 - dF_2 \circ i + h''(\e^{f_1}) \e^{f_1} F_1 ds.
\end{aligned}
\end{equation}
where $T_{\bf f} \mathcal Y^1 = W^{1,p,\delta}_{\mathbf{V}}(\R \times S^1, \C)$ with $ V_- = 0$ and $V_+ = i\R$. 

In Case 2, writing ${\bf g} = g_1 + i g_2$ and again fixing $\tau$, the linearized operator is 
\begin{equation} \label{E:cylLinear2} 
\begin{aligned}
D^{2,\tau} \colon T_{\bf g} \hat{\mathcal Y^2} &\to L^{p,\delta}(\R\times S^1 \setminus \Gamma, (\R^2)^*) \\
G = G_1 + iG_2 &\mapsto dG_1 - dG_2 \circ i + h''(\e^{g_1+ \tau \mu a}) \e^{g_1+ \tau \mu a} G_1 ds.
\end{aligned}
\end{equation}
where $T_{\bf g}\hat{\mathcal Y^2} = W^{1,p,\delta}_{\mathbf{V}}(\R \times S^1 \setminus \Gamma, \C)$ with $V_- = 0$, $V_+ = i\R$ and $V_P = \C$.

In Case 3, the linearized operator
$$
D^3\colon T_{\bf f} \mathcal Y^3 \to L^{p,\delta}(\R\times S^1, (\R^2)^*), 
$$
where $T_{\bf f} \mathcal Y^3 = W^{1,p,\delta}_{\mathbf{V}}(\R \times S^1, \C)$ with $V_- = \C$ and $V_+ = i\R$, is also given by \eqref{E:cylLinear1} 

Evaluating \eqref{E:cylLinear1} and \eqref{E:cylLinear2} at $\partial_s$ and $\partial_t$ and rearranging, we get operators with values in $L^{p,\delta}(\R\times S^1, \C)$, given by 
\begin{equation} \label{E:cylLinear1 s}
\frac{\partial F}{\partial s} + i \frac{\partial F}{\partial t} + h''(\e^{f_1}) \e^{f_1} F_1
\end{equation} 
and 
\begin{equation} \label{E:cylLinear2 s}
\frac{\partial G}{\partial s} + i \frac{\partial G}{\partial t} + h''(\e^{g_1 + \tau \mu a}) \e^{g_1 + \tau \mu a} G_1,
\end{equation} 
respectively.

The asymptotic operators at $\pm\infty$ in Cases 1 and 2, and at $+\infty$ in Case 3, are 
\[
A_\pm := -i \frac{d}{dt} 
- 
\begin{pmatrix} h''(\e^{b_\pm}) \e^{b_\pm} 
& 0 \\
0 &
0  
\end{pmatrix} : W^{1,p}(S^1,\C)\to L^{p}(S^1,\C)
\]

In Case 2, there is also an asymptotic operator at the augmentation puncture:
\[
A_P := -i \frac{d}{dt} : W^{1,p}(S^1,\C)\to L^{p}(S^1,\C)
\]

In Case 3, we have the asymptotic operator at $-\infty$ given by $A_-=-i\frac{d}{dt}$.

All these asymptotic operators are self-adjoint, as partially defined operators $L^{2}(S^1,\C)\to L^{2}(S^1,\C)$. 
Their eigenvalues can be computed using Lemma \ref{L:spectrum}. Note that they all have eigenvalue 0. For the 
remainder of this section, we will {fix $0 < \delta < 2\pi$} such that $\delta < C_\pm = h''(\e^{b_\pm}) \e^{b_\pm} > 0$ 
and $\delta < \frac{1}{2}\left(- C_\pm + \sqrt{ C_\pm^2 + 16 \pi^2}\right)$.
So, $\delta$ is smaller than the absolute value of every non-zero eigenvalue of $A_\pm$ and $A_P$. This $\delta$ will 
be used in the definitions of the relevant weighted Sobolev spaces.

\begin{lemma} \label{L:cylTransversality}
Fix $\tau\in [0,1]$ and ${\bf f} \in \mathcal C^{1,\tau}(\tilde u)$. Then, 
the linearized operator $D^{1,\tau}$ is Fredholm of index 0 and is an isomorphism. 

Fix $\tau\in [0,1]$ and ${\bf g} \in \mathcal C^{2,\tau}(\tilde
u)$. Then, the linearized operator $D^{2,\tau}$ is Fredholm of index 0
and is an isomorphism.

For fixed ${\bf f} \in \mathcal C^{3}(\tilde u)$, $D^3$ is Fredholm of index
2, is surjective and its kernel includes the span of the generator
of the $S^1$-action on the target $\R\times S^1$.  
\end{lemma}

\begin{proof}
The proof has the same structure as that of Lemma \ref{L:cylTransversality1}. 
We start with Case 1. Fix $\tau$ and ${\bf f}$ as in the statement. 
Corollary \ref{C:CZ computation} implies that the Conley--Zehnder indices of the perturbed asymptotic operators $A_\pm +\delta$ are $0$. 
Recall that the kernels of the asymptotic operators $A$ are identified with $i\R$ and that the vector spaces in $\mathbf V$ are $V_- =0$ and $V_+=i\R$. 
Theorem \ref{T:RiemannRochMB} now implies that
$$
\Ind \widehat D^{1,\tau} = (0 + 1) - (0 + 1) = 0.
$$
Following the notation of Wendl \cite{WendlAutomatic}*{Equation (2.5)}, we have 
$$
c_1(E,l,\textbf{A}_\Gamma) = \frac{1}{2}(0-2) = -1 < 0 = \Ind \widehat D^{1,\tau}
$$
so $\widehat D^{1,\tau}$ is an isomorphism by \cite{WendlAutomatic}*{Proposition 2.2}. 

In Case 3, recall that the asymptotic operator at $-\infty$ is now $A_-=-i\frac{d}{dt}$. By Corollary \ref{C:CZ computation}, the Conley--Zehnder index of the perturbed asymptotic operator $A_- +\delta$ is $-1$. 
The kernel of the asymptotic operator $A_-$ can be identified with $\C$, and the vector spaces in $\mathbf V$ are now $V_- =\C$ and $V_+=i\R$. 
By Theorem \ref{T:RiemannRochMB},
$$
\Ind \widehat D^3 = (0+1)-(-1+0) = 2,
$$
$$
c_1(E,l,\textbf{A}_\Gamma) = \frac{1}{2}(2-2) = 0 < 2 = \Ind \widehat D^3
$$
and $\widehat D^{3}$ is surjective by \cite{WendlAutomatic}*{Proposition 2.2}. 

In Case 2, we have the additional puncture $P$, with asymptotic operator $A_P=-i\frac{d}{dt}$, and the Euler charateristic of the domain is $-1$. 
We now have $V_-=0$, $V_+=i\R$ and $V_P =\C$, so Theorem \ref{T:RiemannRochMB} implies
$$
\Ind \widehat D^{2,\tau} = -1 + (0+1) - (0+1)- (-1 + 0) = 0,
$$
$$
c_1(E,l,\textbf{A}_\Gamma) = \frac{1}{2}(0-2) = -1 < 0 = \Ind \widehat D^{2,\tau}
$$
and we can apply once more \cite{WendlAutomatic}*{Proposition 2.2}, to obtain 
that $D^{2, \tau}$ is an isomorphism. 
\end{proof}

We now prove a result analogous to Lemma \ref{L:aprioriMarker}.
\begin{lemma} \label{L:aprioriMarker}

Let $\tau \in [0,1]$.

Suppose that 
$\ff \in \Y^1$, $\gg \in \hat \Y^2$ or $\ff \in \Y^3$ satisfy
\[ \F(\ff, \tau) = 0, \quad \mathcal {G}(\gg, \tau) = 0, \text{ or } df_1 - df_2 \circ i +h'(\e^{f_1}) -da = 0, \]
respectively.

Then, the asymptotic constants satisfy
$f_2(+\infty, \cdot) = f_2(-\infty, \cdot)$, $g_2(+\infty, \cdot) = g_2(-\infty, \cdot)$ 
and $f_2(+\infty, \cdot) = f_2(-\infty, \cdot)$.

In particular, in Cases 1 and 2, these asymptotic constants are $0$.

\end{lemma}

\begin{proof}
	The proof is very similar to the proof of Lemma \ref{L:aprioriMarker1}. 

	Notice that if we have such a solution $\ff$, $\gg$, the resulting function 
	is smooth on the punctured cylinder, by standard elliptic regularity arguments.

	We will prove the result for functions that solve the equation
	\[
		df_1 - df_2 \circ i + h'(\e^{f_1+x})\, ds - \tau \, dy - (1-\tau)\, d\nu = 0,
	\]
	on $\R \times S^1 \setminus \Gamma$, 
	where $x, y \colon \R \times S^1 \setminus \Gamma \to \R$ are 
        functions satisfying the following:
	\begin{itemize}
		\item[Case 1:] $x=0, y=a$, $\Gamma = \emptyset$;
		\item[Case 2:] $x= \tau \mu_P a$, $y = (1-\mu_P) a$, $\Gamma = \{ P \}$;
		\item[Case 3:] $\tau =1$, $x=0$, $y = a$, $\Gamma = \emptyset$.
	\end{itemize}
	We assume furthermore that $f_2$ converges exponentially fast to a 
	constant at $\pm \infty$, and that $f_1$ converges exponentially fast to 
	a constant at $P$ in Case 2.

	Notice in particular that if the cylinder has a puncture, 
	$y$ vanishes in a neighbourhood of the puncture.

	For each $c >> 1$, let $S_c := [-c,c] \times S^1 \setminus D_P(1/c) 
	\subset \R \times S^1 \setminus \Gamma$, 
	where $D_P(\epsilon)$ denotes a disk of radius $\epsilon$ centered at $P$. 

	We then obtain (noticing that $d\nu = \nu'(s) \, ds$)
 \begin{align*}
 \int_{\partial S_c} \tilde f_2 dt 
 &= \int_{S_c} d\tilde f_2 \wedge dt = \int_{S_c} df_2 \wedge dt \\
 &= \int_{S_c} df_2 \circ i \wedge dt\circ i \\
 &= \int_{S_c} ( df_1 + h'(\e^{f_1+x}) ds -\tau dy - (1-\tau) d\nu) \wedge ds \\
 &= \int_{\partial S_c} ( f_1 -\tau y ) \wedge ds
 \intertext{
 This vanishes in Cases 1 and 3, and we are done. 
 In Case 2, $\partial S_c$ also contains a 
 small loop around $P$. In that case, for $c$ sufficiently large,  we obtain:}
 &= \int_{\partial D_P(1/c)} (f_1- \tau y) \, ds \\
 &= \int_{\partial D_P(1/c)} f_1 \, ds. 
 \end{align*}
 Notice that $\int_{\partial D_P(1/c)} f_1 \, ds$ decays like $\frac{1}{c^{2\pi\delta+1}}$.

 The result now follows by taking the limit as $c \to +\infty$.
\end{proof}

We are now ready for the key techincal result of this section, which guarantees existence and uniqueness of solutions to our auxiliary equations. Existence of solutions will be obtained from reducing the problem to an ODE. 
Uniqueness will follow from positivity of intersections of pseudoholomorphic curves in symplectic 4-manifolds.

\begin{lemma} \label{L:countingSolutions1}
There is a unique solution to \eqref{E:deformed cylinder1} in $\mathcal Y^1$ and in $\hat{\mathcal Y^2}$.

If $a = Ts + c$, then there is an 2-parameter family of solutions to \eqref{E:cylinder} in $\mathcal Y^3$, parametrized by $\R\times S^1$.

\end{lemma}
\begin{proof}
We begin with Case 1. Equation \eqref{E:deformed cylinder1} can be rewritten in coordinates as
\begin{equation}\label{E:system}
\begin{cases}
\partial_s f_1 - \partial_t f_2 + h'(\e^{f_1}) - \nu' = 0 \\
\partial_t f_1 + \partial_s f_2 = 0
\end{cases}
\end{equation}

Let us start by considering solutions where $f_2\equiv k$ is constant, and $f_1(s,t)=f_1(s)$ is independent of $t$. Then, we get the $s$-dependent ODE
$$
\frac{d f_1}{d s} + h'(\e^{f_1})- \nu' = 0. 
$$
Write it as
$$
\frac{d f_1}{d s} = F(f_1,s)
$$
where $F(x,s) = - h'(\e^x) + \nu'(s)$. 
\begin{claim}
For every $s_0\in \R$ and every $c_1\in \R$, there is a unique solution $f_1:\R\to \R$ such that $f_1(s_0) = c_1$. 
\end{claim}
To prove the claim, recall the assumptions that $h'(\rho)\geq 0$ and $h'(\rho) = 0$ for $\rho\leq 2$. These and the asymptotic behavior of $\nu$ imply that there are constants $C$ and $\tilde C$ such that:
\begin{itemize}
 \item $F(x,s)\leq C$ for all $(x,s)$;
 \item $F(x,s) \geq \tilde C$ for $x\leq \ln 2$, uniformly in $s$. 
\end{itemize}
The claim now follows from standard existence and uniqueness theory for solutions to ODE's.

The facts that for $s<<0$ we have that $b_-$ is the only zero of the function $x \mapsto F(x,s)$ and that $\partial_x F(x,s)\leq 0$ (a consequence of $h''(x) \geq 0$) imply that there is a unique solution $\tilde f_1$ of the ODE such that $\lim_{s\to -\infty} \tilde f_1(s) = b_-$, see Figure \ref{F:tilde f_1}.

\begin{figure} 
\begin{tikzpicture}[domain=0:6]
    \draw[->] (-4,0) -- (4.1,0) node[right] {$s$};
    \draw[->] (0,-1) -- (0,4) node[above] {$\tilde f_1(s)$};
    \draw [dashed, color=black] (0,3) -- (4,3);
    \draw [dashed, color=black] (-4,1) -- (0,1);
    \draw [dashed, color=black] (-2,-1) -- (-2,4.3);
    \draw [dashed, color=black] (2,-1) -- (2,4.3);
    \draw[line width=1.5pt, color=black] (-4,1) -- (-2,1);
    \draw[line width=1.5pt, color=black] (-2,1) ..  controls(-1,1) and (1,2.2) .. (2,2.594);
    \draw[line width=1.5pt,domain=2:4,smooth,variable=\x,black] plot ({\x},{3-3*exp(-\x)});
    \node at (-.3,3) {$b_+$};
    \node at (.3,1) {$b_-$};
    \node at (-3.2,0.5) {$F\geq 0$};
    \node at (-3.2,2) {$F\leq 0$};
    \node at (-3.2,-0.5) {$\nu = T_- s$};
    \node at (3.2,1.5) {$F\geq 0$};
    \node at (3.2,3.5) {$F\leq 0$};
    \node at (3.2,-0.5) {$\nu = T_+ s$};
\end{tikzpicture}
\caption{The function $\tilde f_1$}
\label{F:tilde f_1}
\end{figure}
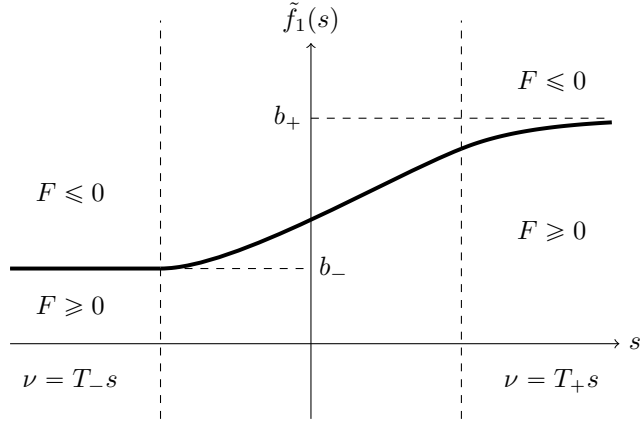

Now, ${\bf f}^0(s,t)=(\tilde f_1(s),0)$ is a solution to \eqref{E:system} in $\mathcal Y^1$. The fact that it has the correct exponential convergence properties at $\pm \infty$ follow from $\tilde f_1$ solving the ODE above. We want to show that there are no other solutions to \eqref{E:system} in $\mathcal Y^1$. The proof will be an application of positivity of intersections of pseudoholomorphic curves in dimension 4.

Suppose that there is a solution ${\bf f}^1=(f_1^1,f_2^1)$ of \eqref{E:system} in $\mathcal Y^1$ such that $f_2^1$ is not constant. By the definition of $\mathcal Y^1$ and by Lemma \ref{L:aprioriMarker}, $\lim_{s\to \pm\infty} f_2^1(s,t) = 0$. Since $f_2^1$ is not constant, there is $(s_0,t_0)\in \R\times S^1$ such that $f_2^1(s_0,t_0)\neq 0$. Let $f_1^2:\R\to\R$ be the unique solution to the ODE above, for which $f_1^2(s_0) = f_1^1(s_0,t_0)$. Then, ${\bf f}^2(s,t)=(f_1^2(s),f_2^1(s_0,t_0))$ is another solution to \eqref{E:system} (although $(f^2_1,0)$ does not have the correct asymptotic behavior to belong to $\mathcal Y^1$). 

Now, observe that \eqref{E:system} is the Floer equation on the K\"ahler manifold $(\R\times S^1,dx\wedge dy,i)$, for the time-dependent Hamiltonian 
\begin{align*}
H\colon \R \times (\R\times S^1) &\to \R \\
(s,x,y) &\mapsto \int_0^x h'(\e^z) dz - \nu'(s)x
\end{align*}
By Gromov's trick \cite{Gromov85}, solutions to the Floer equation can be thought of as pseudoholomorphic curves $\R \times S^1 \to M:= (\R\times S^1)\times (\R\times S^1)$, for some twisted almost complex structure on $M$. The projection of these curves to the first cylinder factor is the identity. 

Since $\lim_{s\to -\infty} f_2^1 = \lim_{s\to -\infty} f_2^0$, Lemma \ref{L:aprioriMarker} implies that they have lifts $\tilde f_2^1$ and $\tilde f_2^0$ to $\R$, with the same limits as $s\to \pm \infty$. 
By assumption, ${\bf f}^1$ and ${\bf f}^0$ are both in $\mathcal Y^1$, which implies that $\lim_{s\to \pm \infty} f_1^1(s,t)=\lim_{s\to \pm \infty} f_1^0(s,t)=b_\pm$. Therefore, there is a homotopy from ${\bf f}^1$ to ${\bf f}^0$ that is $C^0$-small on neighborhoods of $\pm\infty\times S^1$ in $\R \times S^1$. This gives a homotopy from the graph of ${\bf f}^1$ to the graph of ${\bf f}^0$ in $M$, and the intersections with the graph of ${\bf f}^2$ during the homotopy will remain in a compact region of $M$. So, we have an equality of signed intersection numbers
$$
\#(Graph({\bf f}^1)\cap Graph({\bf f}^2))=\#(Graph({\bf f}^0)\cap Graph({\bf f}^2))=0,
$$
since the second components of ${\bf f}^0$ and ${\bf f}^2$ are different constants. But then positivity of intersections of pseudoholomorphic curves in dimension 4 \cite{McDuffSalamon}*{Exercise 2.6.1} implies that the graphs of ${\bf f}^1$ and ${\bf f}^2$ do not intersect, which, by the construction of ${\bf f}^2$, is a contradiction. 
Therefore, there is no such ${\bf f}^1$ to start with, which proves the first statement in the Lemma in Case 1.

Case 2 can be dealt with by the same argument, applied to ${\bf g} \in \hat{\mathcal Y^2}$. 
We again apply Lemma \ref{L:aprioriMarker}.
The argument for Case 1 above produces a solution ${\bf g} = (g_1(s), 0) \in
\Y^1$. We need to argue that such a solution is indeed in $\hat \Y^2$. 
This means that $({\bf g} \circ \varphi_P) \in W^{1,p,\delta}_{\C_-}((-\infty,-1]\times S^1,\R\times S^1)$. This will be a consequence of the fact that $g_1$ is $C^1$, and of its asymptotics. Write $P = (P_s,P_t)$. Then,
$$
({\bf g} \circ \varphi_P)(\rho,\theta) = (g_1(\e^{2\pi \rho} \cos(2\pi\theta) + P_s),0)
$$
and for $\rho << -1$ the Mean Value Theorem gives
$$
|g_1(\e^{2\pi \rho} \cos(2\pi\theta) + P_s) - g_1(P_s)| \leq ||g_1'||_{L^\infty}\,\e^{2\pi \rho} 
$$
 where the sup norm is finite by the exponential convergence of $g_1$ at $\pm \infty$. Since $\delta < 2\pi$, we get the desired exponential convergence at $P$.

Let us now consider the case when $\tilde u = (a,u)$ is such that $a = Ts + c$, and look in $\Y^3$ for solutions to \eqref{E:cylinder}. 
This is equivalent to the system of equations   

\begin{equation}%
\begin{cases}
\partial_s f_1 - \partial_t f_2 + h'(\e^{f_1}) - T = 0 \\
\partial_t f_1 + \partial_s f_2 = 0
\end{cases}
\end{equation}
We again have solutions of the form ${\bf f}(s,t)=(f_1(s),f_2)$ with $f_2$ constant. There is an $\R\times S^1$ family of such solutions, parametrized by $\lim_{s\to -\infty} (f_1(s)-Ts, f_2)\in \R\times S^1$. The union of the graphs of these solutions foliates the region $(\R\times S^1)\times ((-\infty,b_+)\times S^1) \subset (\R\times S^1)\times (\R\times S^1)$. 
We can apply Lemma \ref{L:aprioriMarker} to equation \eqref{E:cylinder}, since it is a special case of \eqref{E:parametric cylinder1} when $\tau = 1$. The argument above using positivity of intersections can then be adapted to prove that there are no other solutions with the same asymptotics. 
To get different solutions, we can replace $(f_1(s),f_2)$ with $(\overline f_1(s),\overline f_2)=(f_1(s+c_1),f_2+c_2)$, for a constant $(c_1,c_2)\in \R\times S^1$. 
\end{proof}

\begin{proof}[Proof of Proposition \ref{Ansatz2}]
 Most of the claims made in the statement follow from the existence of a solution $\bf e$ to \eqref{E:cylinder1} with the appropriate asymptotics. Indeed, given such ${\bf e} = (e_1,e_2)$, one can construct $\tilde v = (b,v)$ by taking 
 $$
 (b,v) = (a + e_1, \phi_R^{e_2} \circ u).
 $$
 As we pointed out, the existence of such $\bf e$ is equivalent to the existence of a solution $\bf f$ to \eqref{E:cylinder}. Let us consider different cases separately. 
 
 In Case 1, note that Lemmas \ref{L:cylTransversality} and \ref{L:smooth parameter} and the Implicit Function Theorem imply that solutions to \eqref{E:parametric cylinder1} vary smoothly in $\tau$. Lemma \ref{L:countingSolutions1} guarantees the existence and uniqueness of solutions to \eqref{E:deformed cylinder1} in $\Y^1$ (which corresponds to $\tau = 0$). 
 Therefore, we conclude that \eqref{E:cylinder} (which corresponds to $\tau = 1$) also has a unique solution in $\Y^1$, which finishes the proof. 

 Case 2 works the same way, with solution $\bf g$ to \eqref{E:parametric cylinder2} instead of solutions $\bf f$ to \eqref{E:parametric cylinder1}.
 
 Case 3 corresponds to the final part of Proposition \ref{Ansatz2}, when $\tilde u$ is assumed to be a trivial cylinder. Lemma \ref{L:countingSolutions1} implies that there is an $\R\times S^1$-family of solutions to \eqref{E:cylinder} in $\Y^3$. By Lemma \ref{L:aprioriMarker}, there is an $\R$-family of such solutions with $\lim_{s\to \pm \infty} f_2(s,t) = 0 \in S^1$. 
 This finishes the proof.
 \end{proof}

 \begin{remark}
     Proposition \ref{Ansatz2} relates Floer cylinders in $\R \times Y$ 
     with pseudoholomorphic cylinders and cylinders satisfying 
     equation \eqref{E:cylinder1}. 
     We will now discuss the relationship between the Fredholm indices of these problems.
     Indeed, suppose that $\tilde u$ is a holomorphic curve, and $\mathbf{e}$
     is a solution to Equation \eqref{E:cylinder1}. This then describes a
     Floer solution $\tilde v$.
     We will now see that the Fredholm index describing deformations of
     $\tilde v$ is one less than the sum of the Fredholm index for $\tilde u$
     and of the index of the problem given by \eqref{E:cylinder1} (the latter 
     of which is give in Lemma \ref{L:cylTransversality}). This
     difference is explained by the need to take a quotient by an $\R$-action,
     though this action is to be understood differently in the 3 cases we
     consider.
     
     First, in Case 1 and in Case 2, we notice that 
     the index of the Floer
     solution is one less than the index of the corresponding holomorphic
     [punctured] cylinder in $\R \times Y$. This index difference comes from
     the fact that the Floer cylinder corresponds to an equivalence class of
     pairs $(\tilde u, \mathbf{e})$, where the equivalence comes from
     translation of $\tilde u$ in the symplectization and a corresponding
     translation of $\mathbf{e}$.

     In Case 3, the Floer solution is contained in a fibre of $\R \times Y \to
     \Sigma$. In this case, for any point $p \in \Sigma$, there is a $2$
     parameter family of Floer solutions in the fibre over $p$, with positive
     end at a closed Hamiltonian orbit and negative end at a Reeb orbit in
     $\{-\infty \} \times Y$. 
     These 2-parameters are given by domain automorphisms. Notice that the
     domain automorphisms are generated by rotations and translations. The
     domain rotation and the Reeb flow are conjugate in this case. 
     Similarly, the trivial
     orbit cylinders in $\R \times Y$ come in a 2-parameter family, which can
     also be described by domain automorphism (or by the $\C^*$-action on $\R
     \times Y$, which is the same thing for these curves). 
     In particular then, we again obtain the claimed equality. 
\end{remark}

\section{Pseudoholomorphic spheres in \texorpdfstring{$\Sigma$}{Sigma} and meromorphic sections}

\label{S:meromorphic}

The previous sections explain that split Floer cylinders are in
some sense equivalent to $J_Y$-holomorphic curves. 
We will now describe $J_Y$-holomorphic curves in a manner that is more suitable
for computing the Floer differential.

Let $\tilde u = (a,u): \R\times S^1\setminus \Gamma \to \R \times Y$ be a $J_Y$-holomorphic curve, where $\Gamma$ is either empty or the singleton $\{P\}$. Since $J_Y$ is $\R\times S^1$-equivariant, it is a lift of an almost complex structure $J_\Sigma$ in $\Sigma$, and one can project $\tilde u$ to obtain a $J_\Sigma$-holomorphic map $w: \R\times S^1\setminus \Gamma \to\Sigma$. 
Since punctures of finite energy pseudoholomorphic curves in $\R\times Y$ are asymptotic to Reeb orbits in $Y$, and these are 
multiple covers of the fibres of $Y\to \Sigma$, $w$ extends to a $J_\Sigma$-holomorphic map $\CP^1 \to \Sigma$ 
\cite{McDuffSalamon}*{Theorem 4.1.2}. We use the standard identification of $\R\times S^1$ with $\C^* \subset \CP^1$. 

The symplectization $\R \times Y$ can be identified with the complement of the zero section in a complex line bundle $E \to \Sigma$ (the dual to the normal bundle to $\Sigma$ in $X$). 
We have the following commutative diagram

\tikzset{node distance=2cm, auto}

\begin{center}
\begin{tikzpicture}
\node (P) {$w \pb (E\setminus 0)$};
  \node (B) [right of=P] {$\R\times Y$}; 
  \node (B') [right of=B] {\hspace{-2cm} $= E \setminus 0$};
  \node (A) [below of=P] {$\CP^1$};
  \node (C) [below of=B] {$\Sigma$};
  \node (P1) [node distance=2cm, left of=A] {$\R\times S^1\setminus \Gamma$};
  \draw[->] (P) to node {} (B);
  \draw[->] (P) to node {} (A);
  \draw[->] (A) to node {$w$} (C);
  \draw[->] (B) to node {} (C);
  \draw[right hook->] (P1) -- (A) [anchor=west]{};
  \draw[->, bend left=60] (P1) to node {$(a,u)$} (B);
  \draw[->, dashed, bend left] (A) to node {$s$} (P);
\end{tikzpicture}
\end{center}
where $s$ is a section of the pullback bundle $w^*E \to \CP^1$ with a zero at $0$ (and at $P$, if $\Gamma = \{P\}$) and a pole at $\infty$. 
The line bundle $w^*E$ has an induced complex linear Cauchy--Riemann operator, in the sense of \cite{McDuffSalamon}*{Section C.1}. Complex linearity is a consequence of the Reeb-invariance of $J_Y$. Such a Cauchy--Riemann operator corresponds to a unique holomorphic structure on $w^*E$ \cite{Kobayashi}*{Proposition 1.3.7}. The section $s$ is in the kernel of this operator and is therefore meromorphic. This proves the following result.

\begin{lemma} \label{L:cylinders to spheres}
Every $J_Y$-holomorphic curve $\tilde u = (a,u): \R\times S^1\setminus \Gamma \to \R \times Y$ defines a $J_\Sigma$-holomorphic map $w: \CP^1 \to \Sigma$ and a meromorphic section of $w^*E \to \CP^1$, with zero at  $0$ (and at $P$, if $\Gamma = \{P\}$) and pole at $\infty$.  
\end{lemma}

The converse to the Lemma (the statement that a meromorphic section induces a $J_Y$-holomorphic curve) is immediate. Therefore, 
we can reduce the question of counting punctured $J_Y$-holomorphic maps $\tilde u$ to that of counting $J_\Sigma$-holomorphic curves 
$w: \CP^1 \to \Sigma$, together with meromorphic sections $s$ of $w^*E$, with prescribed zeros and poles. The count of maps $w$ is 
related with the computation of genus zero Gromov--Witten numbers of $\Sigma$ \cite{McDuffSalamon}. More about this will be said below. 
We can also give a complete description of the relevant meromorphic sections $s$. Given a divisor of points $D$ in $\CP^1$ and a holomorphic line bundle $L$ over $\CP^1$, where both $D$ and $L$ have degree $d$, there is a $\C^*$-family of meromorphic sections of $L$, such that the divisor associated with each section is $D$. 
One can justify this fact by reducing it to the simplest case of trivial $L$: use a trivialization of $L$ over $\C \subset \CP^1$ to identify meromorphic sections of $L$ with meromorphic functions on $\CP^1$ 
\cite{Miranda}*{pp 342--345}.

Fix a $J_\Sigma$-holomorphic map $w: \CP^1 \to \Sigma$ and let $\tilde u = (a,u) \colon \R \times S^1\setminus \Gamma \to \R\times Y$ be a $J_Y$-holomorphic lift. Given $(\theta_1,\theta_2) \in S^1\times S^1$, we can produce a new $J_Y$-holomorphic curve 
$$
(\theta_1,\theta_2) . \tilde u = \tilde u_{(\theta_1, \theta_2)}
$$
by rotating in the domain and in the target: 
$$
 \tilde u_{(\theta_1, \theta_2)}(s,t) = (a(s,t+\theta_1), \phi_R^{\theta_2}\circ u(s,t+\theta_1)).
$$

\begin{lemma} \label{number of lifts}
The number of such lifts that have prescribed asymptotic markers at $\pm\infty$ is 
$k_+ - k_- = - \langle c_1(E \to \Sigma), w_*[\CP^1] \rangle + \sum_{P\in \Gamma} k_P > 0$. 
\end{lemma}
\begin{proof}
    This result is similar to \cite{DiogoLisiSplit}*{Lemma 5.43}. 
    We identify the simple Reeb orbits to which $u$ converges at $\pm \infty$
    with $S^1 \times S^1$. Then, the action of the torus of rotation in the
    domain and in the fibre can be
    seen as a linear map $\T^2 \to \T^2$ 
    represented by the matrix $\begin{pmatrix} k_- &
    1\\ k_+ & 1 \end{pmatrix}. $ See Equation \eqref{rot v} for the analogous discussion 
    in the case of lifts to solutions of Floer's equation in $\R \times Y$, instead of pseudoholomorphic curves.

The number of lifts we want is the absolute 
value of the degree of this map $T^2 \to T^2$, 
which is the absolute value of the determinant 
of the matrix, i.e.~$k_+ - k_-$.

Notice now that $-k_+ + k_- + \sum_{P \in \Gamma} k_P
= \langle c_1(E \to \Sigma), w_*[\CP^1] \rangle$ is the degree of the
bundle $E$, which is negative by construction.
\end{proof}

\begin{remark}
We have related a Floer cylinder $\tilde v$ in $\R\times Y$ with a pseudoholomorphic sphere $w$ in $\Sigma$, by relating $\tilde v$ with a pseudoholomorphic curve $\tilde u$ in $\R\times Y$ and then projecting $\tilde u$ to $w$. The difference between $\tilde v$ and $\tilde u$ is given by ${\bf e} \colon \R \times S^1 \setminus \Gamma \to \R\times S^1$ solving \eqref{E:cylinder0}, and the information lost when we project $\tilde u$ to $w$ is a meromorphic section of a line bundle $L$ over $\CP^1$. Alternatively, we could have projected $\tilde v$ to $w$ directly. The information lost in this process would now be a section of $L$ solving an equation analogous to \eqref{E:cylinder0}. We could have written analogues of Sections \ref{S:ansatz_Floer_to_holo} and \ref{S:ansatz_holo_to_Floer} to analyse such Hamiltonian-perturbed sections of $L$, but decided to separate the analytical and topological aspects of the problem by factoring through the pseudoholomorphic map $\tilde u$.   
\end{remark}

\section{Orientations}

\label{compute signs}

We want to determine the signs of the contributions to the split symplectic homology differential. The orientation conventions 
for Cases 0 through 3 are explained in \cite{DiogoLisiSplit}*{Section 7}. As we saw before, Cases 1 through 3 can be related to counts 
of pseudoholomorphic spheres in $\Sigma$ and $X$. 
We wish to compare the signs that appear in the differential to the signs involved in counts of pseudoholomorphic spheres (as in Gromov--Witten theory).

Recall that Floer cascades are oriented using the following fibre sum convention
(see also \cite{DiogoLisiSplit}*{Definition 7.4}). 

\begin{definition} \label{fibre sum orient} 
Given linear maps between oriented vector spaces $f_i\colon V_i \to W$, $i = 1,2$, such that $f_1 - f_2 \colon V_1\oplus V_2 \to W$ is surjective, the {\em fibre sum orientation} on $V_1 \oplus_{f_i} V_2 = \ker (f_1 - f_2)$ is such that 
\begin{enumerate}
 \item $f_1 - f_2$ induces an isomorphism $(V_1 \oplus V_2) / \ker (f_1 - f_2) \to W$ which changes orientations by $(-1)^{\dim V_2 . \dim W}$,
 \item where a quotient $U/V$ of oriented vector spaces is oriented in such a way that the isomorphism $V \oplus (U/V) \to U$ (associated to a section of the quotient short exact sequence) preserves orientations. 
\end{enumerate}
\end{definition}

It is useful to begin with some general properties of the fibre sum orientation. %
The first result is Proposition 7.5.(a) in \cite{JoyceCorners} (which uses the same fibre sum orientation convention as Definition 
\ref{fibre sum orient}). We include a proof here because the structure of its argument will serve as a model for other proofs in this section. 

\begin{lemma} \label{fibre sum commute} 
 Let $f_i\colon V_1 \to W$ be as in Definition \ref{fibre sum orient}. Consider the map $r\colon V_1 \oplus V_2 \to V_2 \oplus V_1$, such that $r(v_1,v_2) = (v_2,v_1)$. Then, $r$ induces an identification 
 $$
 \ker\big((f_1 - f_ 2)\colon V_1 \oplus V_2 \to W \big) \to  \ker\big((f_2 - f_ 1)\colon V_2 \oplus V_1 \to W \big),
 $$
 which changes the fibre sum orientations by $(-1)^{(\dim V_1 +\dim W)(\dim V_2 + \dim W)}$. 
\end{lemma}
\begin{proof}
 Call $f\coloneqq(f_1-f_2)\colon V_1 \oplus V_2 \to W$ and $g\coloneqq(f_2-f_1)\colon V_2 \oplus V_1 \to W$. Denoting also by $f$ and $g$ the appropriate isomorphisms in Part (1) of Definition \ref{fibre sum orient}, we have an isomorphism 
 $$
 g \circ r \circ f^{-1} \colon W \to W,
 $$
which is $-\Id$, hence it changes orientation by $(-1)^{\dim W}$. 

Denote also by $r$ the induced map $(V_1 \oplus V_2)/\ker f \to (V_2 \oplus V_1)/\ker g$. A simple calculation yields the following.

\begin{claim}
 Let $s_{12} \colon (V_1 \oplus V_2)/\ker f \to V_1 \oplus V_2$ be a right-inverse for the projection $\pi_{12} \colon V_1 \oplus V_2 \to (V_1 \oplus V_2)/\ker f$. Then, 
 $s_{21}\coloneqq r \circ s_{12} \circ r^{-1} \colon (V_2 \oplus V_1)/\ker g \to V_2 \oplus V_1$ is a right-inverse for $\pi_{21} \colon V_2 \oplus V_1 \to (V_2 \oplus V_1)/\ker g$. 
\end{claim}
We can now justify the sign in the statement of the Lemma. By the Claim, the following diagram commutes: 
\begin{diagram}
\ker f \bigoplus \, \big((V_1\oplus V_2)/\ker f\big) & \rTo^{i_{12}+s_{12}\,} & V_1\oplus V_2 \\
\dTo_{r\oplus r}                                & & \dTo_{r}       \\
\ker g \bigoplus \, \big((V_2\oplus V_1)/\ker g\big) & \rTo^{i_{21}+s_{21}\,} & V_2\oplus V_1
\end{diagram}
where $i_{12}$ and $i_{21}$ are inclusion maps. 

By Part (1) in Definition \ref{fibre sum orient} and the argument above the Claim, $r\colon (V_1\oplus V_2)/\ker f \to (V_2\oplus V_1)/\ker g$ changes orientations by $(-1)^{(\dim V_1 + \dim V_2 + 1) \dim W}$. By Part (2) in Definition \ref{fibre sum orient} and the fact that $r\colon V_1\oplus V_2 \to V_2\oplus V_1$ changes orientations by $(-1)^{\dim V_1 \dim V_2}$, we get the remaining contribution to the sign in the Lemma. 
\end{proof}

\begin{definition} \label{preimage orient} 
Let $f_1,f_2$ be as in Definition \ref{fibre sum orient}. Denote by $\Delta \subset W\oplus W$ the diagonal, and orient it as the image of the map $w\mapsto (w,w)$. The {\em preimage orientation} on the fibre sum $V_1 \prescript{}{f_1\,}\oplus\prescript{}{f_2}{} V_2  = (f_1,f_2)^{-1}(\Delta) \subset V_1 \oplus V_2$ is so that, if $H\subset V_1\oplus V_2$ is a complementary subspace to $(f_1,f_2)^{-1}(\Delta)$, the following isomorphisms preserve orientations
\begin{enumerate}
\item $H \oplus  (f_1,f_2)^{-1}(\Delta) \to (-1)^{\dim V_2(\dim V_1 + \dim W)}V_1\oplus V_2$;
\item $(f_1,f_2)(H) \oplus \Delta \to W\oplus W$.
\end{enumerate}
\end{definition}

\begin{remark} \label{GuilleminPollack} 
If $g_1\colon X \to Y$ and $g_2\colon Z \to Y$ are transverse maps of smooth oriented manifolds, then the previous definition can be applied to maps of tangent spaces to orient the fibre product $X \prescript{}{g_1\,}\times\prescript{}{g_2}{} Z$. In the special case when $g_2$ is the inclusion of a submanifold, this agrees with the orientation on the preimage, as defined in \cite{GuilleminPollack}.
This is the reason for the name used for this orientation convention, and for the sign in Part (1) above.
\end{remark}

\begin{lemma} \label{fibre sum vs preimage}
The identification of $\ker f$ with the fibre sum orientation (as in Definition \ref{fibre sum orient}) and $(f_1,f_2)^{-1}(\Delta)$ with preimage orientation (as in Definition \ref{preimage orient}) changes orientations by $(-1)^{(\dim V_1 + \dim W)(\dim V_2 +\dim W)}$. 
\end{lemma}
\begin{proof}
For convenience, denote Parts (1) and (2) in Definition \ref{preimage orient} by (P1) and (P2), respectively, and similarly denote by (F1) and (F2) the Parts in Definition \ref{fibre sum orient}. Note that to write (F2) in this case, we need a section $s\colon \frac{V_1\oplus V_2}{\ker f} \to V_1 \oplus V_2$. We will take $H = s(\frac{V_1\oplus V_2}{\ker f})$ in (P1). 

The result follows from a sequence of oriented isomorphisms:
\begin{align*}
W \oplus W &\oplus\ker f \cong  (-1)^{(\dim V_1 + \dim V_2 - \dim W)\dim W} W \oplus \ker f \oplus W \cong \\
&\stackrel{(F1)}{\cong} (-1)^{\dim V_1 \dim W + \dim W} W \oplus \ker f \oplus \frac{V_1 \oplus V_2}{\ker f } \cong\\ 
& \stackrel{(F2)}{\cong} (-1)^{\dim V_1 \dim W + \dim W} W\oplus V_1 \oplus V_2 \cong \\
&\stackrel{(P1)}{\cong} (-1)^{\dim V_1 \dim W + \dim V_1 \dim V_2 + \dim V_2 \dim W + \dim W}. \\
&\hspace{5cm} .W \oplus H \oplus (f_1,f_2)^{-1}(\Delta) \cong \\
&\cong (-1)^{\dim V_1 \dim W + \dim V_1 \dim V_2 + \dim V_2 \dim W + \dim W}. \\
&\hspace{5cm} .\Delta \oplus (f_1,f_2)(H) \oplus (f_1,f_2)^{-1}(\Delta) \cong \\
&\cong (-1)^{\dim V_1 \dim W + \dim V_1 \dim V_2 + \dim V_2 \dim W}. \\
&\hspace{5cm} .(f_1,f_2)(H) \oplus \Delta \oplus  (f_1,f_2)^{-1}(\Delta) \cong \\
&\stackrel{(P2)}{\cong} (-1)^{\dim V_1 \dim W + \dim V_1 \dim V_2 + \dim V_2 \dim W} W \oplus W \oplus (f_1,f_2)^{-1}(\Delta) \cong\\
&\stackrel{\varphi}{\cong} (-1)^{\dim V_1 \dim W + \dim V_1 \dim V_2 + \dim V_2 \dim W + \dim W} W \oplus W \oplus (f_1,f_2)^{-1}(\Delta) 
\end{align*}
where $\varphi = \begin{pmatrix} 1 & 1 \\
                                                1 & -1 
                          \end{pmatrix} \oplus \id \colon W \oplus W \oplus (f_1,f_2)^{-1}(\Delta) \to W \oplus W \oplus (f_1,f_2)^{-1}(\Delta)$. 
Note that $\varphi$ changes orientations by $(-1)^{\dim W}$. The composition of the chain of isomorphisms above is of the form 
$$
\begin{pmatrix} 2 & * \\
                        0 & 1 
                          \end{pmatrix} \oplus \id \colon W \oplus W \oplus \ker f \to W \oplus W \oplus (f_1,f_2)^{-1}(\Delta)
$$
so it preserves orientations iff the identity $\ker f \to (f_1,f_2)^{-1}(\Delta)$ does. The result now follows. 
\end{proof}

\begin{remark}
Since the signs appearing in Lemmas \ref{fibre sum commute} and \ref{fibre sum vs preimage} are the same, the order-reversing 
map $r$ in Lemma \ref{fibre sum commute} induces an orientation-preserving isomorphism between $\ker(f_1-f_2)$ with the fibre 
sum orientation and $(f_2,f_1)^{-1}(\Delta)$ with the preimage orientation (note the orders of the $f_i$).
\end{remark}

Let us now study the signs that appear in the split symplectic homology differential. We start with Case 1 configurations. 
Let $p,q \in \Crit(f_\Sigma)$ and consider a Floer cylinder with one cascade giving a contribution of $\wh{q}_{k_-}$ to the 
differential of $\wc{p}_{k_+}$.
As we saw in \eqref{fib prod1}, these cascades form spaces 
\begin{equation}
 \M(\wh q_{k_-},\wc p_{k_+};A)\coloneqq W^s_{Y}(\wh q) \times_{\ev} \M^*_{H,0,\R\times Y;k_-,k_+}(A;J_Y) \times_{\ev} W_{Y}^u(\wc p)
 \label{fib prod 1}
\end{equation}
for suitable $A\in H_2(\Sigma;\Z)$. In \cite{DiogoLisiSplit}*{Section 7}, these spaces are oriented using the fibre sum rule. 
the spaces of cylinders are given a coherent orientation and the critical manifolds are oriented as follows. 

For any oriented manifold $S$ with a Morse--Smale pair $(f_S,Z_S)$, fixing an orientation on a critical manifold of a critical point $p$ 
also fixes an orientation on the other critical manifold, using the convention
\begin{equation}
 T_p W^s_S(p) \oplus T_p W^u_S(p) \cong T_p S.
 \label{orient stable and unstable}
\end{equation}

Pick orientations on all unstable manifolds of $\Sigma$ and $W$. Given $p\in \Crit(f_\Sigma)$, assume that 
the restrictions of $\pi_\Sigma\colon Y \to \Sigma$ to
\begin{equation}
 W^u_Y(\widecheck p) \to W^u_\Sigma(p) \qquad \text{and} \qquad W^s_Y(\widehat p) \to W^s_\Sigma(p)
 \label{orient point lift}
\end{equation}
are orientation-preserving diffeomorphisms. 

It will also be useful to assume that, with respect to the connection induced by the contact form of $Y$ on the $S^1$-bundles 
$\pi_\Sigma^{-1}\left( W^{u/s}_\Sigma(p) \right) \to W^{u/s}_\Sigma(p)$, the splittings 
\begin{equation}
 T_{\tilde x} \, \pi_\Sigma^{-1}\left( W^u_\Sigma(p) \right) \cong T_x W^u_\Sigma(p) \oplus \R \qquad \text{and} \qquad 
 T_{\tilde x} \, \pi_\Sigma^{-1}\left( W^s_\Sigma(p) \right) \cong T_x W^s_\Sigma(p) \oplus \R
 \label{orient full lift}
\end{equation}
are orientation-preserving (where $\pi_\Sigma(\tilde x) = x$). In particular, when $p$ is a maximum or minimum, we get the correct relation between the orientations of $Y$ 
and $\Sigma$. 

\begin{remark}
Note that we have codimension zero inclusions of oriented manifolds 
$W^u_Y(\widehat p) \subset \pi_\Sigma^{-1}\left( W^u_\Sigma(p) \right)$ and 
$W^s_Y(\widecheck p) \subset \pi_\Sigma^{-1}\left( W^s_\Sigma(p) \right)$. Under our conventions, these maps 
are always orientation-preserving for $W^u_Y(\widehat p)$, and preserve orientations for $W^s_Y(\widecheck p)$ iff $M(p)$ is even.
\end{remark}

Each cascade in \eqref{fib prod 1} projects to a chain of pearls in $\Sigma$ from $q$ to $p$, with exactly one sphere, contained in
\begin{equation}
 \M(q,p;A)\coloneqq W^s_{\Sigma}(q) \times_{\ev} \M^*_{0,\Sigma}(A;J_\Sigma) \times_{\ev} W_{\Sigma}^u(p),
 \label{fib prod 2}
\end{equation}
where $\M^*_{0,\Sigma}(A;J_\Sigma)$ is the space of simple $J_\Sigma$-holomorphic spheres in class $A$, with no marked points. 
Note that $\M(q,p;A)$ can be identified with $\ev_\Sigma^{-1}\left( W^s_{\Sigma}(q) \times W_{\Sigma}^u(p)\right)$, for
\begin{align*}
\ev_\Sigma \colon \M^*_{0,\Sigma}(A;J_\Sigma) &\to \Sigma \times \Sigma \\
 w &\mapsto (w(0),w(\infty))
\end{align*}

\begin{definition} \label{GW orient case 1} 
 The {\em Gromov--Witten orientation} on $\M(q,p;A)$ is the one induced by the identification with 
 $\ev_\Sigma^{-1}\left( W^s_{\Sigma}(q) \times W_{\Sigma}^u(p)\right)$, where the latter is equipped 
 with the preimage orientation. The {\em Gromov--Witten sign} of an element in $\M(q,p;A)$ is given 
 by the corresponding orientation on the zero-dimensional $\M(q,p;A)/\C^*$. 
\end{definition}

\begin{remark} \label{GW signs} 
To understand the reason for this terminology, suppose that these stable and unstable manifolds defined cycles in $H_*(\Sigma)$. An element in $\M(q,p;A)$ would then contribute to a two-point Gromov--Witten 
invariant of $\Sigma$, with insertions given by the stable and unstable manifolds. The sign of this contribution would be the the one prescribed by the Gromov--Witten orientation (see \cite{McDuffSalamon}*{Exercise 7.1.2}).
\end{remark}

Our goal is to compare 
\begin{itemize}
 \item the sign of the contribution of a Floer cylinder with cascades in \eqref{fib prod 1} to the symplectic homology differential
 \item to the sign induced on the projected chain of pearls in \eqref{fib prod 2} by the Gromov--Witten orientation.
\end{itemize}
This will be the content of Proposition \ref{case 1 and GW}. To compare the two signs, it will be useful to also consider the fibre product
\begin{equation}
 \overline{\M(\wc q_{k_-},\wh p_{k_+};A)}\coloneqq \pi_\Sigma^{-1}\left(W^s_{\Sigma}(q) \right) \times_{\ev} \M^*_{H,0,\R\times Y;k_-,k_+}(A;J_Y) \times_{\ev} \pi_\Sigma^{-1}\left(W^u_{\Sigma}(p) \right).
\label{fib prod 3}
\end{equation}
The notation is justified by the fact that this space contains 
\begin{equation*}
 \M(\wc q_{k_-},\wh p_{k_+};A)\coloneqq W^s_{Y}(\wc q) \times_{\ev} \M^*_{H,0,\R\times Y;k_-,k_+}(A;J_Y) \times_{\ev} W_{Y}^u(\wh p)
\end{equation*}
as a dense subspace. 

Recall that $\pi_\Sigma^{-1}\left(W^{s/u}_{\Sigma}(q) \right)$
is an $S^1$-bundle over $W^{s/u}_{\Sigma}(q)$, and that these were given orientations in \eqref{orient full lift}.
The submanifolds $W^s_{Y}(\wh q)$ and $W^u_{Y}(\wc q)$ specify sections of these bundles, which we use to trivialize 
$\pi_\Sigma^{-1}\left(W^{s/u}_{\Sigma}(q) \right)\cong W^{s/u}_{\Sigma}(q) \times S^1$ (identifying $W^s_{Y}(\wh q)$ 
with $W^s_{\Sigma}(q) \times \{1\}$ and $W^u_{Y}(\wc q)$ with $W^u_{\Sigma}(q) \times \{1\}$). 
We can now write 
\begin{equation}
W^s_{Y}(\wh q)= \pi_\Sigma^{-1}\left(W^s_{\Sigma}(q) \right) \prescript{}{\wc s\,}\times\prescript{}{i} \{1\} \text{ and }
W^u_{Y}(\wc p)= \pi_\Sigma^{-1}\left(W^u_{\Sigma}(p) \right) \prescript{}{\wh u\,}\times\prescript{}{i} \{1\},
\label{fib prod crit}
\end{equation}
where the maps in the fibre products are $\wc s\colon \pi_\Sigma^{-1}\left(W^s_{\Sigma}(q)\right) \to S^1$, the analogous 
$\wh u \colon \pi_\Sigma^{-1}\left(W^u_{\Sigma}(p) \right) \to S^1$ and $i\colon \{1\} \hookrightarrow S^1$. 
Recall that the left sides in \eqref{fib prod crit} were given orientations in \eqref{orient point lift}, which agree with 
the fibre sum orientations on the right sides.

Note that $\M(\wh q_{k_-},\wc p_{k_+};A)$ is a codimension 2 submanifold of $\overline{\M(\wc q_{k_-},\wh p_{k_+};A)}$. 
Given an element $(x,u,y) \in \M(\wh q_{k_-},\wc p_{k_+};A)$, there is a splitting
\begin{equation}
T_{(x,u,y)}\overline{\M(\wc q_{k_-},\wh p_{k_+};A)} \cong \left( T_{(x,u,y)}\M(\wh q_{k_-},\wc p_{k_+};A) \right) \oplus \R_{\domain} \oplus \R_{\target}
\label{orient splitting}
\end{equation}
where $\R_{\domain}$ is the direction spanned by an infinitesimal rotation of the Floer cylinder $u$ on the domain (in the $t$-direction) and $\R_{\target}$ is spanned by the infinitesimal rotation of $u$ on the target (in the Reeb direction).

\begin{lemma} \label{orient lifts free vs fixed} 
If we take the fibre sum orientation on \eqref{fib prod 1} and \eqref{fib prod 3}, then the splitting \eqref{orient splitting} preserves orientations iff $M(q)$ is even. 
\end{lemma}

\begin{proof}
Using the associativity of the fibre sum orientation and Lemma \ref{fibre sum commute} on the third identity, we have the oriented diffeomorphisms (dropping 0, $\R\times Y$ and $J_Y$ from the notation)
\begin{align*}
 \M(&\wh q_{k_-},\wc p_{k_+};A) = W^s_{Y}(\wh q) \times_{\ev}  \M^*_{H;k_-,k_+}(A) \times_{\ev} W_{Y}^u(\wc p) = \\
 &= \left(\pi_\Sigma^{-1}\left(W^s_{\Sigma}(q) \right) \prescript{}{\wc s\,}\times\prescript{}{i} \{1\}\right) \times_{\ev} \M^*_{H;k_-,k_+}(A) \times_{\ev} \left( \pi_\Sigma^{-1}\left(W^u_{\Sigma}(p) \right) \prescript{}{\wh u\,}\times\prescript{}{i} \{1\}\right) = \\
 &= (-1)^{\dim \pi_\Sigma^{-1}\left(W^s_{\Sigma}(q)\right) +1}. \\
 &\qquad .\{1\} \prescript{}{i\,}\times\prescript{}{\wc s}{} \big( \pi_\Sigma^{-1}\left(W^s_{\Sigma}(q) \right) \times_{\ev}  \M^*_{H;k_-,k_+}(A) \times_{\ev} \pi_\Sigma^{-1}\left(W^u_{\Sigma}(p) \right) \big) \prescript{}{\wh u\,}\times\prescript{}{i} \{1\} = \\
 &= (-1)^{M(q)} \{1\} \prescript{}{i\,}\times\prescript{}{\wc s}{} \, \overline{\M(\wc q_{k_-},\wh p_{k_+})} \, \prescript{}{\wh u\,}\times\prescript{}{i} \{1\} 
\end{align*}

We get an oriented isomorphism of vector spaces
$$
 T_{(x,u,y)} \M(\wh q_{k_-},\wc p_{k_+};A) \cong (-1)^{M(q)} 0 \prescript{}{di\,}\oplus\prescript{}{d\wc s}{} \big(T_{(x,u,y)} \overline{\M(\wc q_{k_-},\wh p_{k_+};A)} \big) \prescript{}{d\wh u\,}\oplus\prescript{}{di} \,0
$$
(where we wrote $0$ instead of $T_{\{1\}}{\{1\}}$).

On the other hand, abbreviating $\R^2\coloneqq \R_{\domain} \oplus \R_{\target}$ and thinking of the direct sum as a fibre sum (of maps to the zero vector space), we get
\begin{align*}
 0 \prescript{}{di\,}\oplus\prescript{}{d\wc s}{} &\big( \left( T_{(x,u,y)}\M(\wh q_{k_-},\wc p_{k_+};A) \right) \oplus \R^2 \big) \prescript{}{d\wh u\,}\oplus\prescript{}{di} \,0  = \\
 &= 0 \prescript{}{di\,}\oplus\prescript{}{d\wc s}{} \big( T_{(x,u,y)}\M(\wh q_{k_-},\wc p_{k_+};A) \oplus (\R^2 \prescript{}{d\wh u\,}\oplus\prescript{}{di} \,0 )\big) = \\
 &= (-1)^{\dim \M(\wh q_{k_-},\wc p_{k_+};A)} 0 \prescript{}{di\,}\oplus\prescript{}{d\wc s}{} \big( (\R^2 \prescript{}{d\wh u\,}\oplus\prescript{}{di} \,0 ) \oplus T_{(x,u,y)}\M(\wh q_{k_-},\wc p_{k_+};A)\big) = \\
 &= - \big( 0 \prescript{}{di\,}\oplus\prescript{}{d\wc s}{} \, \R^2 \prescript{}{d\wh u\,}\oplus\prescript{}{di} \,0 \big) \oplus T_{(x,u,y)}\M(\wh q_{k_-},\wc p_{k_+};A) = \\
 &= T_{(x,u,y)}\M(\wh q_{k_-},\wc p_{k_+};A) .
\end{align*}
The first identity above uses the associativity of the fibre sum orientation, the second uses Lemma \ref{fibre sum commute}, 
the third uses the fact that $\dim \M(\wh q_{k_-},\wc p_{k_+};A) = 1$ (since we are studying contributions to the 
differential) and again the associativity property. The last identity follows from the fact that the fibre sum 
orientation on the zero dimensional vector space 
$0 \prescript{}{di\,}\oplus\prescript{}{d\wc s}{} \, \R^2 \prescript{}{d\wh u\,}\oplus\prescript{}{di} \,0$ 
is negative, which is a simple calculation (the reason ultimately being the fact that the matrix in \eqref{rot v} has 
negative determinant). 

The Lemma follows from combining the previous two computations. 
\end{proof}

In what follows, it is useful to work from a slightly more abstract point of view. 
 Let $\pi_A\colon \tilde A \to A$ and $\pi_B\colon \tilde B \to B$ be $S^1$-bundles. Suppose that there are $S^1$-equivariant maps $\tilde f_1\colon \tilde A \to Y$ and $\tilde f_2\colon \tilde B \to Y$, and
 denote their projections by $f_1\colon A \to \Sigma$ and $f_2\colon B \to \Sigma$. Then, the fibre product $\tilde A \prescript{}{\tilde f_1\,}\times\prescript{}{\tilde f_2}{} \tilde B$ is an $S^1$-bundle over $A \prescript{}{f_1\,}\times\prescript{}{f_2}B$ (with `diagonal' $S^1$-action on $\tilde A \prescript{}{\tilde f_1\,}\times\prescript{}{\tilde f_2}{} \tilde B$). 
 
Recall that the contact form on $Y$ induces an Ehresmann connection on $\pi_\Sigma\colon Y \to \Sigma$. The $S^1$-equivariance of $\tilde f_1$ implies that there is a unique Ehresmann connection on $\tilde A$ inducing the top arrow in the commutative diagram
\begin{equation}
\begin{diagram} 
T_a A\oplus \R & \rTo^{\,} & T_{\tilde a}\tilde A \\
\dTo^{d f_1 \oplus \id}   & & \dTo_{d\tilde f_1}       \\
T_{f_1(a)}\Sigma \oplus \R & \rTo^{\,} & T_{\tilde f_1(\tilde a)}Y
\end{diagram} 
\label{connection A}
\end{equation}
where $\tilde a\in \tilde A$ and $a = \pi_A(\tilde a)$. Do the same for $\tilde B$.

 These connections induce a connection on the $S^1$-bundle $\tilde A \prescript{}{\tilde f_1\,}\times\prescript{}{\tilde f_2}{} \tilde B$ as follows. 
 Given $(\tilde a, \tilde b) \in \tilde A \prescript{}{\tilde f_1\,}\times\prescript{}{\tilde f_2}{} \tilde B$, we can take the horizontal subspace $H_{\tilde a} \prescript{}{d(\pi_\Sigma\circ \tilde f_1)\,}\oplus\prescript{}{d(\pi_\Sigma\circ \tilde f_2)}{} H_{\tilde b}$. 
This can be identified with $T_{(a,b)} (A \prescript{}{f_1\,}\times\prescript{}{f_2}B)$, via the restriction of $(\pi_A,\pi_B)$ to the fibre sum. 
This gives an identification 
\begin{equation} 
T_{(\tilde a, \tilde b)} (\tilde A \prescript{}{\tilde f_1\,}\times\prescript{}{\tilde f_2}{} \tilde B) \cong \left( T_{(a,b)} (A \prescript{}{f_1\,}\times\prescript{}{f_2}B) \right) \oplus \R,
\label{split lift}
\end{equation}
where $\R$ stands for the vector space generated by the infinitesimal generator of the $S^1$-action on $\tilde A \prescript{}{\tilde f_1\,}\times\prescript{}{\tilde f_2}{} \tilde B$. 
We rewrite \eqref{split lift} as an isomorphism
$$
\chi\colon \ker (d f_1 - d f_2) _{(a,b)}\oplus \R \to \ker (d\tilde f_1 - d\tilde f_2)_{(\tilde a,\tilde b)}.
$$

Pick an orientation on $A$ and orient $\tilde A$ so that 
the splitting $T_{\tilde a}\tilde A \cong T_a A \oplus \R$ induced by the connection on $\tilde A$ preserves orientations. 
Pick an orientation on $B$ and orient $\tilde B$ in an analogous manner. 

 \begin{lemma} \label{fib prod lift}
The identification \eqref{split lift} preserves orientations, if we use the fibre sum orientation on both fibre products, and 
orient $\R$ in the direction of the infinitesimal $S^1$-action on $\tilde A \prescript{}{\tilde f_1\,}\times\prescript{}{\tilde f_2}{} \tilde B$.
 \end{lemma}
\begin{proof}
The structure of the argument is parallel to that of Lemma \ref{fibre sum commute}.  

Fix $(\tilde a,\tilde b)\in \tilde A \prescript{}{\tilde f_1\,}\times\prescript{}{\tilde f_2}{} \tilde B$, so that 
$\pi_A(\tilde a)= a$ and $\pi_B(\tilde b)=b$. Write $d \tilde f\coloneqq(d \tilde f_1- d\tilde f_2)\colon T\tilde A \oplus T\tilde B \to T Y$ 
and $d f\coloneqq(d f_1- d f_2)\colon T A \oplus T B \to T \Sigma$, where we omited subscripts $a,b,\tilde a,\tilde b$ to make 
the notation lighter. We will do this throughout the proof. Denote by $\xi\colon T\Sigma \oplus \R \to TY$ the isomorphism induced 
by the contact structure on $Y$. The commutative diagram
\begin{diagram}
\frac{T A\oplus T B}{\ker (d f)} \oplus \R & \rTo^{[df]\oplus \id \,} & T \Sigma \oplus \R \\
\dTo_{\varphi}                                & & \dTo_{\xi}       \\
\frac{T\tilde A\oplus T\tilde B}{\ker (d \tilde f)} & \rTo^{[d\tilde f] \,} & T Y
\end{diagram} 
defines the isomorphism $\varphi$. 

Denote by $\psi$ the composition $T A \oplus T B \oplus \R \oplus \R \to T A \oplus \R \oplus T B \oplus \R \to T \tilde A \oplus T\tilde B$, where the first isomorphism permutes the second and third factors and the second isomorphism is induced by the connections on $\tilde A$ and $\tilde B$. %

\begin{claim}
 Let $s \colon \frac{T A\oplus T B}{\ker (d f)} \to TA \oplus TB$ be a right-inverse for the projection $\pi \colon TA \oplus TB \to \frac{T A\oplus T B}{\ker (d f)}$. Then, 
 $\tilde s\coloneqq \psi\circ (s\oplus (1,0)) \circ \varphi^{-1} \colon \frac{T\tilde A\oplus T\tilde B}{\ker (d \tilde f)} \to T\tilde A \oplus T\tilde B$ is a right-inverse for $\tilde \pi \colon T\tilde A \oplus T\tilde B \to \frac{T\tilde A\oplus T\tilde B}{\ker (d \tilde f)}$. 
\end{claim}

The claim follows from the fact that the following diagram commutes, which is a simple calculation (using \eqref{connection A} and its analogue for $\tilde B$).
\begin{diagram}
T A \oplus T B \oplus \R \oplus \R & \rTo^{\pi\oplus \begin{pmatrix} 1 & -1 \end{pmatrix} \,} & \frac{T A\oplus T B}{\ker (d f)} \oplus \R \\
\dTo_{\psi}                                & & \dTo_{\varphi}       \\
T\tilde A\oplus T\tilde B & \rTo^{\tilde \pi \,} & \frac{T\tilde A\oplus T\tilde B}{\ker (d \tilde f)}
\end{diagram}

We can now justify the sign in the statement of the Lemma. The Claim can be used to show that the following diagram commutes: 
\begin{diagram}
\big(\ker (d f)\oplus \R\big) \bigoplus \, \left(\frac{T A\oplus T B}{\ker (d f)} \oplus \R\right) & \rTo^{i+\alpha+s+\beta\,} & T  A \oplus T B \oplus \R \oplus \R \\
\dTo_{\chi \oplus \varphi}                                & & \dTo_{\psi}       \\
\ker (d \tilde f) \bigoplus \, \frac{T\tilde A\oplus T\tilde B}{\ker (d \tilde f)} & \rTo^{\tilde i + \tilde s\,} & T \tilde A \oplus T\tilde B 
\end{diagram}
where $\alpha \oplus \beta \colon \R\oplus \R \to \R^2$ is the linear map represented by 
$\begin{pmatrix}
1 & 1 \\
1 & 0
\end{pmatrix}$.

By Part (1) in Definition \ref{fibre sum orient}, $\varphi$ changes orientations by $(-1)^{\dim B \dim \Sigma + \dim \tilde B \dim Y} = (-1)^{\dim B + 1}$. The map $\psi$ changes orientations by $(-1)^{\dim B}$ (the sign comes from the permutation, since the orientations on $\tilde A$ and $\tilde B$ behave well with the connections). 
Note that to get the map $\alpha \oplus \beta$ (which has negative determinant), we need to permute $\R$ with $\frac{T A\oplus T B}{\ker (d f)}$ at the top of the diagram. The permutation picks up no sign, since $\dim \frac{T A\oplus T B}{\ker (d f)} = \dim \Sigma$ is even. 
By Part (2) in Definition \ref{fibre sum orient} and the commutativity of the diagram, we conclude that the isomorphism $\chi$ changes orientations by $(-1)^{\dim B + 1 + \dim B + 1} = 1$, as wanted. 
\end{proof}

The space \eqref{fib prod 3} is an $S^1$-bundle over \eqref{fib prod 2}. With respect to the connections discussed above (see \eqref{connection A}), we have a decomposition
\begin{equation}
T_{(x,u,y)}\overline{\M(\wc q_{k_-},\wh p_{k_+};A)} \cong \left( T_{(\pi_\Sigma(x),\pi_\Sigma\circ u,\pi_\Sigma(y))}\M(q,p;A) \right) \oplus \R
\label{orient splitting 2}
\end{equation}

\begin{lemma} \label{orient spheres and lifts} 
The isomorphism \eqref{orient splitting 2} preserves orientations, if the fibre products on the left and right are given their fibre sum orientations.
\end{lemma}

\begin{proof}
 This follows from applying Lemma \ref{fib prod lift} twice, to the $S^1$-bundles 
 $\pi_\Sigma^{-1}\left(W^s_{\Sigma}(q) \right) \to W^s_{\Sigma}(q)$, 
 ${\M}^*_{H,0,\R\times Y;k_-,k_+}(A;J_Y) \to \M^*_{0,\Sigma}(A;J_\Sigma)$ and 
 $\pi_\Sigma^{-1}\left(W^u_{\Sigma}(p) \right) \to W^u_{\Sigma}(p)$. 

 Note that to apply Lemma \ref{fib prod lift}, we need the identifications $T_{\tilde a}\tilde A\cong T_a A\oplus \R$ and 
 $T_{\tilde b}\tilde B\cong T_b B\oplus \R$ (introduced above the statement) to be orientation-preserving.
 This holds when $A$ or $B$ are critical manifolds of $(f_\Sigma,Z_\Sigma)$ and when $\tilde A$ and $\tilde B$ are the 
corresponding preimages under $\pi_\Sigma$, by \eqref{orient full lift}. 
On the other hand, from \cite{DiogoLisiSplit}*{Lemma 5.22 and discussion after Lemma 7.3} we get this for the coherent orientation on 
$\tilde B = {\M}^*_{H,0,\R\times Y;k_-,k_+}(A;J_Y)$ and the usual complex orientation on $B=\M^*_{0,\Sigma}(A;J_\Sigma)$.
\end{proof}

\begin{lemma} \label{fib prod vs preim in Sigma}
The fibre sum orientation on \eqref{fib prod 2} differs from the preimage orientation on \eqref{fib prod 2} by $(-1)^{\dim W^s_{\Sigma}(q) \dim W_{\Sigma}^u(p)}$. 
\end{lemma}
\begin{proof}
 This follows from applying Lemma \ref{fibre sum vs preimage} twice. Writing \eqref{fib prod 2} as 
 $$
(W^s_{\Sigma}(q) \times_{\ev} \M^*_{0,\Sigma}(A;J_\Sigma)) \times_{\ev} W_{\Sigma}^u(p)
 $$
we apply Lemma \ref{fibre sum vs preimage} to $W^s_{\Sigma}(q) \hookrightarrow \Sigma$ and $\ev_\Sigma^1\colon \M^*_{0,\Sigma}(A;J_\Sigma)\to \Sigma$ to get a sign 
$$(-1)^{(\dim W^s_{\Sigma}(q) + \dim \Sigma)(\dim \M^*_{0,\Sigma}(A;J_\Sigma) + \dim \Sigma)} = 1$$
since $\M^*_{0,\Sigma}(A;J_\Sigma)$ and $\Sigma$ are even dimensional. 
We then apply Lemma \ref{fibre sum vs preimage} to $\ev^2\colon W^s_{\Sigma}(q) \times_{\ev} \M^*_{0,\Sigma}(A;J_\Sigma) \to \Sigma$ and $W_{\Sigma}^u(p) \hookrightarrow \Sigma$ to get the sign 
\begin{align*}
(-1&)^{\dim (W^s_{\Sigma}(q) \times_{\ev} \M^*_{0,\Sigma}(A;J_\Sigma) + \dim \Sigma) (\dim W_{\Sigma}^u(p) + \dim \Sigma)} \\
&= (-1)^{\dim W^s_{\Sigma}(q) \dim W_{\Sigma}^u(p)}.
\end{align*}
\end{proof}

Combining these results, we get the following.
\begin{proposition} \label{case 1 and GW}
The sign with which an element of \eqref{fib prod 1} contributes to the symplectic homology differential agrees with 
the Gromov--Witten sign of the projection of this element to \eqref{fib prod 2}  
iff $M(q)$ is even.
\end{proposition}
\begin{proof}
Combining Lemmas \ref{orient lifts free vs fixed} and \ref{orient spheres and lifts}, we can relate the fibre sum orientations on \eqref{fib prod 1} and \eqref{fib prod 2}:
\begin{align*}
\left( T_{(x,u,y)}\M(\wh q_{k_-},\wc p_{k_+};A) \right) \oplus &\R_{\domain} \oplus \R_{\target} 0\cong (-1)^{M(q)} T_{(x,u,y)}\overline{\M(\wc q_{k_-},\wh p_{k_+};A)} \cong \\
&\cong (-1)^{M(q)} \left( T_{(\pi_\Sigma(x),\pi_\Sigma\circ u,\pi_\Sigma(y))}\M(q,p;A) \right) \oplus \R
\end{align*}
The factors $\R_{target}$ on the left and $\R$ on the right can be identified, yielding an orientation-preserving isomorphism 
\begin{align*}
\left( T_{(x,u,y)}\M(\wh q_{k_-},\wc p_{k_+};A) \right) \oplus \R_{\domain} \cong (-1)^{M(q)} T_{(\pi_\Sigma(x),\pi_\Sigma\circ u,\pi_\Sigma(y))}\M(q,p;A).
\end{align*}

Observe that $\M(q,p;A)$ in \eqref{fib prod 2} can be identified with 
 \begin{equation}
 \M^*_{0,\Sigma}(A;J_\Sigma) \times_{\ev_\Sigma} \big(W^s_{\Sigma}(q) \times W_{\Sigma}^u(p)\big).
 \label{fib prod 22}
 \end{equation}
 
The sign with which an element $(x,u,y)\in\M(\wh q_{k_-},\wc p_{k_+};A)$ contributes to the symplectic 
homology differential is the sign of the zero-dimensional vector space obtained by taking the quotient 
of $T_{(x,u,y)}\M(\wh q_{k_-},\wc p_{k_+};A)$ by translation in the $s$-variable in the domain of the 
Floer cylinder $u$. This oriented zero-dimensional vector space coincides with the quotient of 
$\left( T_{(x,u,y)}\M(\wh q_{k_-},\wc p_{k_+};A) \right) \oplus \R_{\domain}$ by the infinitesimal 
$\C^*$-action on the domain of $u$. 
On the other hand, by Definition \ref{GW orient case 1}, 
the Gromov--witten sign of $(\pi_\Sigma(x),\pi_\Sigma\circ u,\pi_\Sigma(y))\in\M(q,p;A)$ 
is the sign of the zero-dimensional vector space obtained as the quotient of 
$T_{(\pi_\Sigma(x),\pi_\Sigma\circ u,\pi_\Sigma(y))}\M(q,p;A)$ by the infinitesimal $\C^*$-action,
where $\M(q,p;A)$ is identified with \eqref{fib prod 22} with the preimage orientation. 
The result now follows from 
\begin{claim}
 If we equip \eqref{fib prod 2} with the fibre sum orientation and \eqref{fib prod 22} with the preimage orientation, then the natural identification of the two spaces is orientation-preserving.
\end{claim}

By Lemma \ref{fib prod vs preim in Sigma}, the fibre sum and preimage orientations on \eqref{fib prod 2} differ by $(-1)^{\dim W^s_{\Sigma}(q) \dim W_{\Sigma}^u(p)}$. 
But this is also the difference between the preimage orientations on \eqref{fib prod 2} and on \eqref{fib prod 22}. The Claim now follows.  
\end{proof}

\begin{remark}
 The Claim in the previous proof is also a consequence of the following argument. If we equip \eqref{fib prod 2} and 
 \eqref{fib prod 22} with their fibre sum orientations, then their natural identification preserves orientations, by 
 \cite{JoyceCorners}*{Proposition 7.5.(c)}. 
 On the other hand, the fibre sum and preimage orientations on \eqref{fib prod 22} are the same, by Lemma \ref{fibre sum vs preimage}.  
\end{remark}

\

There is an analogous argument for Case 2 contributions to the symplectic homology differential. As we saw in \eqref{fib prod2},
they come from elements of fibre products
\begin{equation}
 W^s_{Y}(\wh p) \times_{\ev} \big( \M_X^*(B;J_W) \times_{\tilde \ev} \M^*_{H,1,\R\times Y;k_-,k_+}(0;J_Y) \big) \times_{\ev} W_{Y}^u(\wc p),
\label{fib prod case 2}
\end{equation}
which are given the fibre sum orientation. An element in this fibre product has an underlying unparameterized plane in $\M_X^*(B;J_W)$ 
(that can also be thought of as a space of spheres in $X$), which is asymptotic to a Reeb orbit over $p\in \Sigma$. Denote by 
$\ev\colon \M_X^*(B;J_W) \to \Sigma$ the evaluation map.

\begin{definition} \label{GW orient case 2}
 Given a generic point $pt \in \Sigma$, the preimage orientation on $\ev^{-1}(pt)$ is called its 
 {\em Gromov--Witten orientation}. The corresponding sign associated to an element in the zero-dimensional
 $\ev^{-1}(pt)$ is its {\em Gromov--Witten sign}.
\end{definition}

Observe that this is oriented diffeomeorphic to $\M_X^*(B;J_W) \times_{\ev} \{pt\}$,
by Lemma \ref{fibre sum vs preimage}. Analogously to what was pointed out in Remark \ref{GW signs}, 
this orientation determines the sign with which an element in $\ev^{-1}(pt)$ contributes to a 
relative Gromov--Witten invariant of the pair $(X,\Sigma,\omega)$, with a single insertion of the point 
$pt$ (chosen to be generic) in $\Sigma$, with order of contact $B\bullet \Sigma$.

\begin{proposition} \label{case 2 and gw}
The sign with which an element of \eqref{fib prod case 2} contributes to the symplectic homology 
differential agrees with the Gromov--Witten sign of the underlying $J_W$-holomorphic plane in $W$ 
(thought of as a sphere in $X$ with one point constraint in $\Sigma$ of order of contact 
$k_+-k_-=B\bullet \Sigma$ at a generic $p\in \Sigma$).
\end{proposition}

\begin{proof}
The proof is analogous to that of Proposition \ref{case 1 and GW}. We have oriented isomorphisms
\begin{align*}
 \big[&W^s_{Y}(\wh p) \times_{\ev} \big( \M_X^*(B;J_W) \times_{\tilde \ev} \M^*_{H,1,\R\times Y;k_-,k_+}(0;J_Y) \big) \big]\times_{\ev} W_{Y}^u(\wc p) \\
 & \cong \big[\big(\M_X^*(B;J_W) \times_{\tilde \ev} \M^*_{H,1,\R\times Y;k_-,k_+}(0;J_Y) \big)\times_{\ev} W^s_{Y}(\wh p) \big] \times_{\ev} W_{Y}^u(\wc p) \\ 
 & \cong \big[\M_X^*(B;J_W) \times_{\ev} \big(\M^*_{H,1,\R\times Y;k_-,k_+}(0;J_Y) \times_{\ev} W^s_{Y}(\wh p) \big)\big] \times_{\ev} W_{Y}^u(\wc p)  \\
 & \cong \M_X^*(B;J_W) \times_{\ev} \big[\big(\M^*_{H,1,\R\times Y;k_-,k_+}(0;J_Y) \times_{\ev} W^s_{Y}(\wh p) \big) \times_{\ev} W_{Y}^u(\wc p)\big]  \\
 & \cong (-1)^{1 + M(p)} \M_X^*(B;J_W) \times_{\ev} \big[\M^*_{H,1,\R\times Y;k_-,k_+}(0;J_Y) \times_{\ev} \big(W^s_{Y}(\wh p)  \times W_{Y}^u(\wc p)\big)\big]  \\
 & \cong (-1)^{1 + M(p)}  \M_X^*(B;J_W) \times_{\ev} \big[\text{sign}(S^1_d\times S^1_t \to S^1_- \times S^1_+) \\
 & \hspace{4.5cm}  (\M^*_{1,\Sigma}(0;J_\Sigma)/\C^*) \times_{\ev} \big(W^s_{\Sigma}(p) \times W_{\Sigma}^u(p)\big) \R_p \big]  \\
 & \cong (-1)^{M(p)}\big(\Delta_\Sigma \cap (W^s_{\Sigma}(p)\times W^u_{\Sigma}(p))\big) \M_X^*(B;J_W) \times_{\ev} \R_p \\
 & \cong \big(\M_X^*(B;J_W) \times_{\ev} \{p\}\big) \R_p %
\end{align*}
where we identified the space of Floer cylinders from $\wh p$ to $\wc p$ with the real line $\R_p$ in the last three lines, with the orientation induced by $s$-translation on the domain (any two such cylinders differ by translation by a constant $s_0\in \R$). 
The first isomorphism preserves orientations, by Lemma \ref{fibre sum commute} ($\dim \big( \M_X^*(B;J_W) \times_{\tilde \ev} \M^*_{H,1,\R\times Y;k_-,k_+}(0;J_Y) \big) + \dim Y$ is even). 
The second and third isomorphisms are orientation-preserving, by the associativity of the fibre sum orientation. The sign in the fourth isomorphism comes from \cite{JoyceCorners}*{Proposition 7.5.(c)}. 
Note that by Lemma \ref{fibre sum vs preimage}, the fibre sum and preimage orientations on $\big[\M^*_{H,1,\R\times Y;k_-,k_+}(0;J_Y) \times_{ev} \big(W^s_{Y}(\wh p)  \times W_{Y}^u(\wc p)\big)\big]$ agree. 
In the fifth isomorphism, $\text{sign}(S^1_d\times S^1_t \to S^1_- \times S^1_+)=-1$ means that the evaluation map
taking domain and target rotations to the $S^1$-fibres at the orbits at $-\infty$ and $+\infty$ reverses orientations 
(again by the fact that the matrix in \eqref{rot v} has negative determinant). 
In the penultimate line, $\big(\Delta_\Sigma \cap (W^s_{\Sigma}(p)\times W^u_{\Sigma}(p))\big)$ stands for 
the sign of the intersection, which is the same as the sign of $(-1)^{M(p)} (W^s_{\Sigma}(p)\cap W^u_{\Sigma}(p))$, 
by Remark \ref{GuilleminPollack}. The sign of $W^s_{\Sigma}(p)\cap W^u_{\Sigma}(p)$ is 1, by \eqref{orient stable and unstable}.
\end{proof}

Finally, Case 3 contributions to the symplectic homology differential come from
\begin{equation}
 W^u_{W}(x) \times_{\ev^1_-} \left(\M^*_{H}(B;J_W)  \prescript{}{\ev^1_+\,}\times\prescript{}{\ev^2_-}{}  \M^*_{H,k_+}(0;J_Y) \right) \times_{\ev^2_+} W_{Y}^u(\wc p)
 \label{fib prod case 3}
\end{equation}
with the fibre sum orientation, as in \eqref{fib prod3} (fix orientations on unstable manifolds of $(-f_W,-Z_W)$, which induces 
orientations on the stable manifolds via \eqref{orient stable and unstable}). 
Despite the subscript (which reminds us that we haven't quotiented by domain automorphisms), the space 
$\M^*_{H}(B;J_W)$ consists of pseudoholomorphic cylinders in $W$ (recall that $H=0$ in $W$ after stretching the neck),
which by Lemma \ref{planes = spheres} can be identified with a space of pseudoholomorphic spheres in $X$.
Recall from Lemma \ref{planes = spheres} that the map $\psi \colon W \to X$ is a diffeomorphism onto its image 
$X\setminus \Sigma$. Denote by $\ev\colon \M^*_{H}(B;J_W)/\C^* \to X\times \Sigma$ the evaluation map at $(0,\infty)$
(where we quotiented by the space $\C^*$ of domain automorphisms). Write $W^u_{X}(x)$ for $\psi\big(W^u_{W}(x)\big)$.

\begin{definition}
 The preimage orientation on $\ev^{-1}\left(W^u_{X}(x) \times W_{\Sigma}^u(p)\right)$ is called 
 its {\em Gromov--Witten orientation}. The corresponding sign associated to an element in this 
 zero-dimensional manifold is its {\em Gromov--Witten sign}.
\label{GW orient case 3}
\end{definition}

If $W^u_{X}(x) \subset X$ and $W_{\Sigma}^u(p)\subset \Sigma$ represent homology classes, then the Gromov--Witten orientation determines the signs with which elements in $\ev^{-1}\left(W^u_{X}(x) \times W_{\Sigma}^u(p)\right)$ contribute to a relative Gromov--Witten invariant. 

\begin{proposition}
The sign with which an element of \eqref{fib prod case 3} contributes to the symplectic homology differential agrees with 
the Gromov--Witten sign of the underlying $J_W$-holomorphic plane in $W$ (thought of as a sphere in $X$ 
with one point constraint at $W^u_{X}(x)\subset X$ and one point constraint in $\Sigma$ of order of contact 
$k=B\bullet \Sigma$ at $W_{\Sigma}^u(p)\subset \Sigma$).
\label{case 3 and gw}
\end{proposition}
\begin{proof}
Denoting all evaluation maps by $\ev$, we have the oriented diffeomorphisms
\begin{align*}
W^u_{W}&(x) \times_{\ev} \left(\M^*_{H}(B;J_W) \times_{\ev} \M^*_{H,k_+}(0;J_Y) \right) \times_{\ev} W_{Y}^u(\wc p) \\
&\cong W^u_{W}(x) \times_{\ev} \M^*_{H}(B;J_W) \times_{\ev} (\M^*_{H,k_+}(0;J_Y) \times_{\ev} W_{Y}^u(\wc p))) \\
&\cong (-1)^{\dim W_{Y}^u(\wc p)} \big(W^u_{W}(x) \times_{\ev}  \M^*_{H}(B;J_W) \times_{\ev} W_{Y}^u(\wc p)\big) \times \R_2\\
&\cong (-1)^{\dim W_{Y}^u(\wc p) + \dim W^u_{W}(x)} \big(\M^*_{H}(B;J_W) \times_{\ev} \left(W^u_{W}(x) \times W_{Y}^u(\wc p)\right)\big)\times \R_2
\end{align*}
where we denote by $\R_2$ the domain of the $s$-coordinate of a Floer cylinder in $\M^*_{H,k_+}(0;J_Y)$. We used 
the associativity of the fibre sum orientation, the oriented diffeomorphism $\M^*_{H,k_+}(0;J_Y)\cong \R_2 \times Y$ 
(see \cite{DiogoLisiSplit}*{Lemma 5.22}), the fact that $W$ and $\M^*_{H}(B;J_W)$ are 
even dimensional and \cite{JoyceCorners}*{Proposition 7.5.(a)\&(c)}. 
Now, if we think of $\M^*_{H}(B;J_W)$ 
as a space of pseudoholomorphic spheres in $X$, then using the fibre sum orientation convention, we also have the oriented diffeomorphisms
\begin{align*}
&\M^*_{H}(B;J_W) \times_{\ev} \left(W^u_{W}(x) \times W_{Y}^u(\wc p)\right) \\
 &\quad \cong (-1)^{\dim W^u_{W}(x) + \dim W_{\Sigma}^u(p)} \left((\M^*_{H}(B;J_W)/\C^*) \times_{\ev} \left(W^u_{X}(x) \times W_{\Sigma}^u(p)\right)\right)\times \R_1
\end{align*}
where $\R_1$ is the domain of the $s$-coordinate of a Floer cylinder in $\M^*_{H}(B;J_W)$. By Lemma 
\ref{fibre sum vs preimage}, the fibre sum and preimage orientations agree on 
$(\M^*_{H}(B;J_W)/\C^*) \times_{\ev} \left(W^u_{X}(x) \times W_{\Sigma}^u(p)\right)$. Combining the two 
calculations, we get an oriented diffeomorphism between 
$$
 W^u_{W}(x) \times_{\ev}  \M^*_{H}(B;J_W) \times_{\ev} \M^*_{H,k_+}(0;J_Y) \times_{\ev} W_{Y}^u(\wc p)
$$
with the fibre sum orientation (which gives rigid contributions to the symplectic homology differential when we quotient by $\R_1\times \R_2$) and 
$$
\left((\M^*_{H}(B;J_W)/\C^*) \times_{\ev} \left(W^u_{X}(x) \times W_{\Sigma}^u(p)\right)\right )\times \R_1 \times \R_2
$$
with the preimage orientation 
(which gives the Gromov--Witten sign when we quotient by $\R_1\times \R_2$, according to Definition \ref{GW orient case 3}). 
The result now follows. 
\end{proof}

We now define some curve counts that will be useful when we write
a formula for the symplectic homology differential.  Each of the
Cases 1 through 3 motivates one of the following definitions.

\begin{definition} \label{D:n_A}
 Let $q,p\in \Crit(f_\Sigma)$ and $A \in H_2(\Sigma;\Z)$. Define
 $$
 n_A(q,p) := \# \M_A(q,p),
 $$
 where (using the notation of \eqref{fib prod 2})
 $$
 \M_A(q,p) := \M(q,p;A) / \C^*.
 $$
Here, $\M(q,p;A)$ is given the Gromov--Witten orientation (see Definition \ref{GW orient case 1}). 
\end{definition} 
 
This makes sense when $ \dim \M_A(q,p) = 2\, \langle c_1(T\Sigma),A\rangle + M(p) - M(q)-2 = 0$. 

\begin{definition} \label{D:n_B}
Let $B\in H_2(X;\Z)$ and define
$$
n_B := \# \M_B,
$$
where
 \begin{align*}
 \M_B :=\ev^{-1}(pt).
 \end{align*}
Here, $\ev^{-1}(pt)$ is as in Definition \ref{GW orient case 2}, and it is given its Gromov--Witten orientation. 
\footnote{We thank Viktor Fromm for observing that the correct constraint is a generic point in $\Sigma$, rather than a submanifold of possibly higher dimension.}
\end{definition}

This makes sense when $ \dim \M_B = 2\, \langle c_1(TX),B\rangle - 2 (B\bullet\Sigma) - 2 = 0$. 

\begin{definition} \label{D:n_B(x,p)}
 Let $p\in \Crit(f_\Sigma)$, $x\in \Crit(f_W)$ and $B \in H_2(X;\Z)$. Define
 $$
 n_B(x,p) := \# \M_B(x,p)
 $$
 where, using the notation from Definition \ref{GW orient case 3},
 \begin{align*}
 \M_B(x,p) := \ev^{-1}\left(W^u_{X}(x) \times W_{\Sigma}^u(p)\right).
 \end{align*}
\end{definition}

The formula makes sense when $ \dim \M_B(x,p) = 2\, \langle c_1(TX),B\rangle - 2(B\bullet\Sigma) + M(p) + M(x) -2n = 0$ (see the index formula
for relative Gromov--Witten invariants, for instance in \cite{IonelParkerRelative}*{Lemma 4.2}).

\begin{remark} \label{index bounds}
The coefficients $n_A(q,p)$ above can only be non-trivial if 
$$
\langle c_1(T\Sigma),A\rangle = (\tau_X - K) \, \omega(A) \leq n.
$$

by the dimension formula for $\M_A(q,p)$. Similarly, the coefficients $n_B$ can 
only be non-trivial if 
$$
\, \langle c_1(TX),B\rangle - B\bullet\Sigma = (1 - K/\tau_X) \, \langle c_1(TX),B\rangle = (\tau_X - K) \, \omega(B) = 1.
$$
The coefficients $n_B(x,p)$ can only be non-zero if
$$
\, \langle c_1(TX),B\rangle - B\bullet\Sigma = (1 - K/\tau_X) \, \langle c_1(TX),B\rangle = (\tau_X - K) \, \omega(B) \leq n.
$$
These inequalities are useful when computing the symplectic homology differential, as we will see in an example below. 
\end{remark}

\section{A formula for the symplectic homology differential}

Recall that the symplectic chain complex \eqref{eqn:chain complex} is 
$$
SC_*(W,H) = \left(\bigoplus_{k>0} \bigoplus_{p_k\in \text{Crit}(f_\Sigma)}\Z\langle \widecheck p_k, \widehat p_k \rangle\right) \oplus \left(\bigoplus_{x\in \text{Crit}(f_W)}\Z\langle x \rangle\right)
$$
as an abelian group, where $f_\Sigma: \Sigma \to \R$ and $f_W : W \to \R$ are Morse functions, with associated 
Morse--Smale gradient-like vector fields $Z_\Sigma$ and $Z_W$, respectively. 
We have assigned orientations to all critical manifolds in $\Sigma$, $Y$ and $W$, in Section  \ref{compute signs}.
The orientation conventions for the Morse differential were specified in \cite{DiogoLisiSplit}*{Section 7}. 

We are now ready to write the differential on $SC_*(W,H)$.
\begin{theorem} \label{T:differential}
 The differential on \eqref{eqn:chain complex}, denoted $\partial$, is as follows. Given $p\neq q \in \Crit(f_\Sigma)$, the coefficient of $\widehat q_{k_-}$ in $\partial \widecheck p_{k_+}$ is 
 \begin{align*}
 \langle \partial (\widecheck p_{k_+}) , \widehat q_{k_-} \rangle = (-1)^{M(q)} (k_+-k_-) &\sum_{A \in H_2(\Sigma;\Z)} \delta_{k_+-k_-,\langle c_1(N\Sigma),A \rangle} \,\, n_A(q,p) + \\
 & + \delta_{k_+,k_-} \, \langle  \partial_{f_Y}(\widecheck p_{k_+}),  \widehat q_{k_-} \rangle
 \end{align*}
 where $2\langle c_1(T\Sigma),A \rangle + M(p) - M(q) -2 = 0$, $\partial_{f_Y}$ is the Morse differential in $Y$ and the $\delta_{*,*}$ are Kronecker deltas. 
 
 If $p\in \Crit (f_\Sigma)$, then 
 $$
 \langle \partial (\widecheck p_{k_+}) , \widehat p_{k_-} \rangle =  (k_+-k_-) \, \sum_{B \in H_2(X;\Z)} \delta_{k_+-k_-,(B\bullet\Sigma)} \, n_B
 $$
 where $(1-K/\tau_X) \, \langle c_1(T X),B \rangle - 1 = 0$. If $W$ is Weinstein, this term is trivial if $n\geq 3$ and the minimal Chern number of spheres in $\Sigma$ is at least 2.

 Given $p\in \Crit(f_\Sigma)$ and $x\in \Crit(f_W)$, 
 $$
 \langle \partial (\widecheck p_{k_+}) , x \rangle = k_+ \sum_{B \in H_2(X;\Z)} \delta_{k_+,(B\bullet\Sigma)} \, n_B(x,p)
 $$
 where $2(1-K/\tau_X) \, \langle c_1(TX),B\rangle + M(p) + M(x) -2n = 0$. 
 
 Given $x\in \Crit(f_W)$,
 $$
 \partial(x) = \partial_{-f_W}(x),
 $$
 where $\partial_{-f_W}$ is the Morse differential in $W$ with respect to the function $-f_W$.
\end{theorem}

\begin{proof}
The split Floer differential on \eqref{eqn:chain complex} was introduced at the end of Section \ref{cascades after splitting}. Proposition \ref{P:split SH cases} describe those split Floer cylinders with cascades that can contribute to the differential.  

The two types of contributions in $\langle \partial (\widecheck p_{k_+}) , \widehat q_{k_-} \rangle$ are (1) and (0), respectively, in Proposition \ref{P:split SH cases}. The former are given by Floer cylinders in $\R \times Y$ and correspond to Case 1 in Sections \ref{Floer_to_holo} and \ref{holo_to_Floer}. According to Propositions \ref{Ansatz1} and \ref{Ansatz2}, counts of Floer cylinders in $\R\times Y$ are equivalent to counts of pseudoholomorphic cylinders in $\R \times Y$. By Lemma \ref{L:cylinders to spheres} and the discussion that follows, the latter are equivalent to counts of rigid pseudoholomorphic spheres in $\Sigma$ (given by the $n_A(q,p)$), with additional factors $k_+-k_-$ in accordance with Lemma \ref{number of lifts}. 
The sign comes from Proposition \ref{case 1 and GW}. 

The contributions in $\langle \partial (\widecheck p_{k_+}) , \widehat p_{k_-} \rangle$ correspond to (2) in Proposition \ref{P:split SH cases}, namely split Floer cylinders with one augmentation. They correspond to Case 2 in Sections \ref{Floer_to_holo} and \ref{holo_to_Floer}. We can appeal again to Propositions \ref{Ansatz1} and \ref{Ansatz2}, to relate punctured Floer cylinders in $\R\times Y$ to counts of punctured pseudoholomorphic cylinders in $\R \times Y$. The latter are branched covers of trivial cylinders in this case, and project to constant spheres in $\Sigma$. 
The number of lifts of these constant spheres is $k_+-k_-$, again by Lemma \ref{number of lifts} (in the notation of that Lemma, $k^- + k^-_P = k^+$ in this case, since the spheres in $\Sigma$ are constant). 
The counts of augmentation planes in $W$, capping the punctures in these cylinders, are given by the numbers $n_B$, by Lemma \ref{planes = spheres}. 
The comment about $W$ Weinstein follows from Lemma \cite{DiogoLisiSplit}*{Lemma 6.6} %
(and, by Lemma \ref{augmentation rel GW} below, it can be interpreted as saying that a certain relative Gromov--Witten invariant of the triple $(X,\Sigma,\omega)$ vanishes).  

The contributions in $\langle \partial (\widecheck p_{k_+}) , x \rangle$ are the ones in Proposition \ref{P:split SH cases}, 
and consist of a split Floer cylinder in $\R\times Y$ with one end asymptotic to a Reeb orbit in $Y$, matching the asymptotics 
of a pseudohomorphic plane in $W$, which is in turn connected to a critical point of $f_W$ by a flow line of $Z_W$. They correspond 
to Case 3 in Sections \ref{Floer_to_holo} and \ref{holo_to_Floer}. We can appeal to Propositions \ref{Ansatz1} and \ref{Ansatz2}, 
to relate the Floer cylinders in $\R\times Y$ to trivial pseudoholomorphic cylinders in $\R \times Y$. 
The component in $W$ is a pseudoholomorphic cylinder $\tilde v_0 \colon \R\times S^1 \to W$ with a removable singularity at 
$-\infty$. At $+\infty$, the cylinder converges to a Reeb orbit of multiplicity $k$. Therefore, precomposing $\tilde v_0$ with 
constant rotations on the domain (as in the discussion around \eqref{rot v0}), we get $k_+$ many maps 
with the correct asymptotic marker condition at $+\infty$. By Lemma \ref{planes = spheres}, such $\tilde v_0$ correspond to 
pseudoholomorphic spheres in $X$ with appropriate tangency condition to $\Sigma$. The contribution of these configurations to 
the differential will then be $k_+ \, n_B(x,p)$, as wanted.

The only contributions to the differential of $x\in \Crit(f_W)$ are from the Morse differential, because $W$ is exact and hence 
has no non-constant pseudoholomorphic spheres. They correspond to the second part of Case 0 in Proposition \ref{P:split SH cases}. 
\end{proof}

\begin{remark}
For every $p\in \Crit(f_\Sigma)$ and $k_+>0$, we have $\partial_{f_Y}(\widehat{p}_{k_+}) = 0$. There are two flow lines of $-Z_Y$ from $\widehat{p}_{k_+}$, which cancel one another.  
\end{remark}

\begin{remark}
At least if there were no contributions $\langle \partial (\widecheck p_{k_+}) , x \rangle$, it would be immediate from Theorem \ref{T:differential} that $\partial^2 = 0$.
\end{remark}

\subsection{Relation to Gromov--Witten invariants}

All the coefficients contributing to the differential in Theorem \ref{T:differential} are either combinatorial or topological, except for the  $n_A(q,p)$, the $n_B$ and the $n_B(x,p)$. These are harder to determine, but can sometimes be related to Gromov--Witten invariants, absolute or relative. 

We denote by $\gw^\Sigma_{0,k,A}(C_1,\ldots,C_k)$ the genus 0 Gromov--Witten invariant of $(\Sigma,\omega_\Sigma)$, counting pseudohomolorphic spheres in $\Sigma$ in class $A\in H_2(\Sigma;\Z)$ with $k$ marked points constrained to go through representatives of the classes $C_i\in H_*(\Sigma;\Z)$. To see a definition of these invariants when $(\Sigma,\omega_\Sigma)$ is monotone, see \cite{McDuffSalamon}.  
We also denote denote by $\gw^{(X,\Sigma)}_{0,k,(s_1,\ldots,s_l),B}(C_1,\ldots,C_k;D_1,\ldots,D_l)$ the genus 0 relative Gromov--Witten invariant of $(X,\Sigma,\omega)$, counting pseudohomolorphic spheres in $X$ class $B\in H_2(X;\Z)$ with $k$ marked points constrained to go through representatives of the classes $C_i\in H_*(X;\Z)$, and $l$ additional marked points constrained to be tangent to $\Sigma$ with order of contact $s_j$ and go through representatives of the classes $D_j\in H_*(\Sigma;\Z)$. For details on how to define these invariants, see \cites{IonelParkerRelative,LiRuan,TehraniZingerSSF} (for an algebro-geometric approach, see \cite{LiRelative}). 
Note that our assumptions that $(X,\omega)$ is monotone with monotonicity constant $\tau_X$, that $\Sigma$ is Poincar\'e-dual to $K\omega$ and that $\tau_X - K > 0$ imply that the tuple $(X,\Sigma,\omega)$ is {\em strongly semi-positive}, as in \cite{TehraniZingerSSF}*{Definition 4.7}.

Recall that a Morse function is {\em perfect} if the corresponding Morse differential vanishes. Furthermore, it is {\em lacunary} if it does not have critical points of consecutive Morse index.
\begin{lemma}\label{lacunary GW}
 If the Morse function $f_\Sigma$ is perfect, then $W^s_\Sigma(q)$ and $W^u_\Sigma(p)$ represent classes in $H_*(\Sigma;\Z)$ and 
 $$
 n_A(q,p) = \gw^\Sigma_{0,2,A}\big([W^s_\Sigma(q)],[W^u_\Sigma(p)]\big).
 $$

 If the Morse function $f_W$ is also perfect, then $W^u_{X}(x)\coloneq\psi\big(W^u_W(x)\big)$ represents a class in the image of the map $H_*(X\setminus \Sigma;\Z) \to H_*(X;\Z)$ and 
 $$
 n_B(x,p) = \gw^{X,\Sigma}_{0,1,(B\bullet\Sigma),B}\big([W^u_{X}(x)];[W^u_\Sigma(p)]\big).
 $$

 If $f_\Sigma$ is lacunary, then 
 $$
 \langle \partial_{f_Y}(\widecheck p_{k_+}), \widehat q_{k_+}\rangle = \, \left\langle c_1(N\Sigma),\big(W^s_{\Sigma}(q) \cap W^u_{\Sigma}(p)\big)\right\rangle 
 $$
 where $N\Sigma$ is the normal bundle to $\Sigma$ in $X$. 
 
\end{lemma}

\begin{proof}
For the first statement, the fact that the stable and unstable submanifolds of $f_\Sigma$ define homology classes in $\Sigma$ implies that the $n_A(q,p)$ are the Gromov--Witten invariants in the statement, and not merely numbers defined at the (Morse) chain level. 
The orientation conventions in Definition \ref{D:n_A} are precisely the ones used in Gromov--Witten theory (recall Remark \ref{GW signs}). 
The second statement is proven along similar lines. Note that the $W^u_W(x)$ define homology classes in $W$, but that the $W^s_W(x)$ (which we do not consider) would define relative classes in $H_*(W,Y;\Z)$ instead. 
On a technical note, an admissible $J_X$ as in Definition \ref{admissible J} satisfies the vanishing normal Nijenhuis tensor condition \cite{TehraniZingerSSF}*{(4.6)}, and is hence suitable for defining relative Gromov--Witten invariants. 

We now discuss the last statement. For index reasons, the coefficients $\langle \partial_{f_Y}(\widecheck p_{k_+}),\widehat q_{k_+}\rangle$ 
can only be non-zero if $\ind_{f_\Sigma} (p) - \ind_{f_\Sigma}(q) = 2$. Since $f_\Sigma$ is assumed to be lacunary, 
$W^s_{\Sigma}(q) \cap W^u_{\Sigma}(p)$, together with $p$ and $q$, form a 2-sphere $S$. The restriction of $Y \to \Sigma$ to $S$ 
is an $S^1$-bundle $E$ over $S$, and $(f_Y,Z_Y)$ restricts to a Morse--Smale pair $(f_E,Z_E)$ on $E$. By the description of the 
Gysin sequence in Morse homology \cite{OanceaThesis}*{(3.26)}, $\langle \partial_{f_E}(\widecheck p), \widehat q\rangle$ is the 
Euler class of $E\to S$ integrated over $S$, which is precisely what we wanted to show. 
\end{proof}

The following result does not need any hypotheses on the auxiliary Morse functions. 
\begin{lemma} \label{augmentation rel GW}
The count of augmentation planes
 $$
 n_B = \gw^{X,\Sigma}_{0,0,(B\bullet\Sigma),B}(\emptyset;[pt])
 $$ 
is the genus 0 relative Gromov--Witten invariant of $(X,\Sigma,\omega)$, in class $B\in H_2(X,\Z)$, with no point constraints in $X$ and order of contact $B\bullet\Sigma$ to a fixed (generic) point in $\Sigma$.
\end{lemma}

The proof is as in the second case of Lemma \ref{lacunary GW}.

\begin{remark}
When $f_\Sigma$ is not a perfect Morse function, the stable and unstable manifolds of its critical points may not represent classes in $H_*(\Sigma;\Z)$. In such cases, the $n_A(q,p)$ can still be thought of as {\em chain level Gromov--Witten numbers} of $\Sigma$. 
These are not invariants, and may thus be harder to compute. Similarly, if $f_W$ is not perfect, the $n_B(x,p)$ are only chain level relative Gromov--Witten numbers. 

We will see below an example in which we compute Gromov--Witten invariants of $(\CP^1\times \CP^1,\Delta,\omega_{FS}\oplus \omega_{FS})$ using the integral $J$, even though it is not `cylindrical near $\Sigma$', as we require of admissible $J_X$ in Definition \ref{admissible J}. That is not a problem, though, since the numbers we are computing are invariants. 

\end{remark}

\begin{remark} \label{R:computations}

Computing Gromov--Witten invariants can be a very hard problem, both in
the absolute and relative case, but it has been done for a large class
of important examples, especially using tools from algebraic geometry.
We implicitly assume then that the definitions of Gromov--Witten
invariants in algebraic geometry and in symplectic geometry agree in
settings where they are both defined.  For more details about the relation
between symplectic and algebraic absolute Gromov--Witten invariants,
see \cites{SiebertGW,LiTian}.

For computations of absolute invariants, see for instance
\cites{Beauville,ZingerCompleteIntersections}.  Gathmann wrote
a computer program called \texttt{GROWI} to compute absolute and
relative Gromov--Witten invariants of hypersurfaces in projective spaces
\cite{growi}.

Relative invariants are often harder to compute than absolute
ones. Nevertheless, under certain assumptions, a relative
invariant where all intersections with $\Sigma$ are transverse
(i.e.~have order of tangency 1) can be related to an absolute
invariant where  one forgets $\Sigma$ (in the spirit of
the divisor axiom for absolute Gromov--Witten invariants)
\cites{McDuffAbsoluteRelative,HuRuan,TehraniZingerAbsoluteRelative}.

It is also worth pointing out that the relative invariants
of a tuple $(X,\Sigma,\omega)$ can in principle be obtained
from the absolute invariants of $X$ and of $\Sigma$ (including
more complicated insertions like $\psi$-classes, in general)
\cite{MaulikPandharipande}. Nevertheless, it can be computationally
quite challenging to use this fact to compute relative invariants of
$(X,\Sigma,\omega)$.

Depending on what part of symplectic homology one wants to compute, it may also
be possible to ignore the contributions coming from relative Gromov--Witten
invariants. 
By Theorem \ref{T:differential}, if 
$X$ is at least 6 dimensional, the minimal Chern number of $\Sigma$ is at least 2 
and $W$ is Weinstein, then the symplectic homology differential does not have
any $n_A$ contributions (i.e.~corresponding to augmentation planes).
In this case then, the 
{\em positive} (or {\em large-action} or {\em high energy}) 
symplectic homology differential of $W$ \cite{BOExactSequence} then ignores
terms of the form $n_B(x,p)$ and thus 
only has
contributions from the (absolute) Gromov--Witten invariants.
\end{remark}

\begin{remark} \label{Cone}
The following was pointed out to us by D. Pomerleano. We also thank R. Leclercq and M. Sandon for helpful conversations about their related work in progress.
 Theorem \ref{T:differential} may seem unsatisfactory since it relates the differential $\partial$, which is a chain level object depending on the model we use to compute symplectic homology, to counts of pseudoholomorphic spheres which, at least under the assumption of Lemma \ref{lacunary GW}, are given by Gromov--Witten invariants, which are homological. 
 The chain complex $SC_*$ in \eqref{eqn:chain complex} has a quotient complex, referred to in \cite{BOExactSequence} as $SC_*^+$, obtained by modding out generators coming from critical points of $f_W$. 
 We can write $SC_*^+$ as a complex of the form 
 \begin{equation}
 C_*(\Sigma)[t]t \oplus \big(C_*(\Sigma)[t]t \big) [1]
\label{SC+}
 \end{equation}
 where $C_*(\Sigma)$ denotes the Morse complex of $f_\Sigma$ and $t^k$ encodes generators associated to Reeb orbits of multiplicity $k$. The variable $t$ has degree $2\frac{\tau_X - K}{K}$, by \eqref{eqn:grading Y} (which should also force us to take a global shift of $C_*(\Sigma)$ by $1-n$; we don't do this to avoid making the notation even heavier). The first summand contains the orbits decorated with $\widecheck{} \,$ and the second summand, with the degree shift up by 1, contains the orbits decorated with $\widehat{} \,$. 
 Theorem \ref{T:differential} implies that the differential in \eqref{SC+} is given by a matrix 
 $$
 \partial^+ = \begin{pmatrix}
  d_\Sigma & \Delta \\
  0 & d_\Sigma
 \end{pmatrix} 
 $$
 where $d_\Sigma$ is the Morse differential in $\Sigma$.
 The off-diagonal term $\Delta\colon C_*(\Sigma)[t]t \to C_*(\Sigma)[t]t$ is obtained by counting flow lines in $\Sigma$ connecting critical points of index difference 2, pseudohomolorphic spheres in $\Sigma$ (the terms involving $n_A(p,q)$), and 
 pseudoholomorphic spheres in $X$ (the terms involving $n_B$). 
This means that $SC_*^+$ is a cone complex on $\Delta$. There is an induced long exact sequence
$$
\ldots \to H_*(\Sigma)[t]t[1] \to SH_{*}^+(W,H) \to H_{*}(\Sigma)[t]t \stackrel{[\Delta]}{\to} H_{*}(\Sigma)[t]t[1] \to \ldots
$$
with connecting homomorphism $[\Delta]$ (we are not being careful keeping track of the degrees $*$ in the sequence). 
This can be thought of as the long exact sequence in \cite{BOExactSequence}, in the special case when the Liouville manifold is the completion of $X\setminus \Sigma$.  
It is also related to the spectral sequence in \cite{SeidelBiasedView}*{(3.2)}.

We can think of this exact sequence as a more invariant consequence of Theorem \ref{T:differential}. Note that if we were interested in computing $SH_*^+(W)$ with field coefficients, then the exact sequence suggests that it should be enough to compute Gromov--Witten invariants, instead of chain level counts of pseudoholomorphic spheres, even if one does not start with a perfect Morse function in $\Sigma$. 
Observe that $SC_*(W)$ is also a cone complex on a chain map $SC_*^+(W) \to C_*(W)$, counting pseudoholomorphic spheres in $X$ with tangency conditions in $\Sigma$ (the terms involving $n_B(p,x)$). The same argument as above should enable us to compute $SH_*(W)$ with the additional information of relative Gromov--Witten invariants of the triple $(X,\Sigma,\omega)$, instead of the non-invariant counts $n_B(p,x)$, at least over field coefficients.

\end{remark}
 
\part{Example: \texorpdfstring{$T^*S^2$}{T*S2}}

We now illustrate the results in this paper with the computation of the symplectic homology of the completion $W$ of $X\setminus \Sigma$, where $(X,\Sigma, \omega) = (\CP^1 \times \CP^1,\Delta,\omega_{FS}\oplus \omega_{FS})$, where $\Delta$ is the diagonal and $\omega$ restricts to the Fubini-Study form of area 1 in each factor. The manifold $W$ is symplectomorphic to $T^*S^2$ (see Exercise 6.20 in \cite{McDuffSalamonIntro}), so its symplectic homology is isomorphic to the homology of the based loop space of $S^2$. We will see that our computation recovers this result.  
$X$ and $\Sigma$ are both monotone, with $\tau_{X} = 2$ and $\tau_\Sigma = 1$. Also, $[\Sigma]$ is Poincar\'e-dual to $[\omega]$, so $K=1$. The $S^1$-bundle $Y\to \Sigma$ is $\R P^3$.

\section{The coefficients in the differential}
The manifold $\Sigma = \CP^1$ admits a lacunary Morse function $f_\Sigma: \Sigma
\to \R$, with one critical point of Morse index 0, denoted by $m$ (for {\em minimum}),
and one critical point of index 2, denoted by $M$ (for {\em maximum}).
Lemma \ref{lacunary GW} implies that all the numbers $n_A(q,p)$ (giving the Case 1 contributions to the differential) are absolute Gromov--Witten invariants of $\Sigma = \CP^1$. Definition \ref{D:n_A} and Remark \ref{index bounds} indicate that the only potentially non-trivial invariant is 
$$
n_L(M,m) = \gw_{L,2}^{\CP^1}\left([pt],[pt] \right) = 1
$$
where $L\in H_2(\CP^1;\Z)$ is the class of a line. This invariant counts the
number of lines in $\CP^1$ through two generic points, which is 1 (the integral
complex structure on $\CP^1$ is regular). 

The manifold $W = T^*S^2$ also admits a lacunary Morse function $f_W: W \to \R$
growing at infinity, with one critical point of index 0, denoted by $e$, and one
critical point of index 2, denoted by $c$. 
The manifold $W^u(c)$ represents the zero section in $T^*S^2$. We orient it as
$L_1 - L_2$, where $L_1,L_2 \in H_2(\CP^1 \times \CP^1;\Z)$ are homology classes
representing the factors $\CP^1 \times \{ pt \}$ and $\{ pt \} \times \CP^1$,
respectively. We will refer to a sphere representing the classes $L_1$ and $L_2$
as \textit{horizontal} and \textit{vertical}, respectively.

Definition \ref{D:n_A} and Remark \ref{index bounds} again indicate that the only potentially non-trivial coefficients in Cases 2 and 3 are 
$$
n_{L_i}, \quad n_{L_i}(e,M), \quad n_{L_i}(c,m), \quad n_{2L_i}(e,m) \quad \text{and} \quad n_{L_1+L_2}(e,m).
$$

Lemmas \ref{lacunary GW} and \ref{augmentation rel GW} imply that these numbers are relative Gromov--Witten invariants of $(\CP^1 \times \CP^1, \Delta,\omega)$. We will now compute these invariants.

\begin{proposition} 
 \begin{enumerate}
  \item $n_{L_i} = \gw_{L_i,0,(1)}^{\CP^1\times \CP^1,\Delta}\left(\emptyset; pt \right) = 1$,
  
\noindent 
for $i=1,2$ (the point constraint is in $\Delta$, not in $\CP^1 \times \CP^1$);
  
  \item $n_{L_i}(e,M) = \gw_{L_i,1,(1)}^{\CP^1 \times \CP^1, \Delta}\left([pt]; [\Delta] \right) = 1$;
  
  \item $n_{L_i}(c,m) = \gw_{L_i,1,(1)}^{\CP^1\times \CP^1, \Delta}([S^2];[pt]) = (-1)^i$
 
 \noindent 
  where $S^2 \subset T^* S^2$ is the zero section, oriented so as to represent the homology class $L_1 - L_2$; 
  
  \item $n_{2L_i}(e,m) = \gw_{2L_i,1,(2)}^{\CP^1 \times \CP^1, \Delta}\left([pt]; [pt] \right) = 0$;
  
  \item $n_{L_1+L_2}(e,m) = \gw_{L_1+L_2,1,(2)}^{\CP^1 \times \CP^1, \Delta}\left([pt]; [pt] \right) = 1$.
   \end{enumerate}
   \label{P:rel GW computations}
\end{proposition}

We need an auxiliary result. 

\begin{lemma} \label{L:integrable J}
 The product complex structure on $X = \CP^1 \times \CP^1$ is regular for all
 immersed holomorphic spheres $U \colon \CP^1 \to \CP^1 \times \CP^1$.
In particular, this is true when $U_*[\CP^1] \in \{L_1,L_2,L_1 + L_2\}$. 
\end{lemma}
\begin{proof}
 Observe that the symplectic form is integral and so for any holomorphic sphere
 $U$, $\int_{\CP^1} U^*\omega \geq 1$. 
 By \cite{McDuffSalamon}*{Lemma 3.3.3}, we just need to make sure that $\langle c_1(TX),U_*[\CP^1]\rangle \geq 1$, but $\langle c_1(TX),U_*[\CP^1]\rangle = 2 \int_{\CP^1} U^*\omega \geq 2$.
\end{proof}

For the proof of cases (1) through (3) of Proposition \ref{P:rel GW
computations}, one could appeal to the relation between relative and
absolute Gromov--Witten invariants when all the tangencies to $\Sigma$
are of order 1, as mentioned in Remark \ref{R:computations}.  This is
not necessary, though, since Lemma \ref{L:integrable J} allows us to
do the computations explicitly using the product complex structure on
$\CP^1 \times \CP^1$. This satisfies the vanishing normal Nijenhuis
tensor condition \cite{TehraniZingerSSF}*{(4.6)}, and can thus be used
for computing relative Gromov--Witten invariants.

\begin{proof}[Proof of Proposition \ref{P:rel GW computations}]

  \begin{enumerate}
  \item Given $i=1,2$,
$$
\gw_{L_i,0,(1)}^{\CP^1\times \CP^1,\Delta}\left(\emptyset; [pt] \right) = 1
$$
is the fact that if one fixes any point $p\in\Delta$, then there is exactly one holomorphic sphere in class $L_i$ that goes through $p$.
  \item Given $i=1,2$, 
$$
\gw_{L_i,1,(1)}^{\CP^1 \times \CP^1, \Delta}\left([pt]; [\Delta] \right) = 1
$$
expresses the fact that, if we fix any point $p \in \CP^1 \times \CP^1 \setminus \Delta$, there is a unique holomorphic sphere in class $L_i$ that goes through $p$ and intersects $\Delta$. 
  \item Given $i,j=1,2$,
$$
\gw_{L_i,1,(1)}^{\CP^1\times \CP^1, \Delta}(L_j;[pt]) = 1-\delta_{i,j}.
$$
This means, for instance, that there is a unique vertical sphere in $\CP^1 \times \CP^1$ intersecting a horizontal sphere and a generic point $p \in\CP^1\times \CP^1$, but there is no horizontal sphere intersecting another horizontal sphere and a generic $p$. 
This is because two different horizontal spheres do not intersect. Since $S^2 = L_1 - L_2$,
$$
\gw_{L_i,1,(1)}^{\CP^1\times \CP^1, \Delta}([S^2];[pt]) = \gw_{L_i,1,(1)}^{\CP^1\times \CP^1, \Delta}(L_1;[pt]) - \gw_{L_i,1,(1)}^{\CP^1\times \CP^1, \Delta}(L_2;[pt]) = (-1)^i
$$
 \item Given $i = 1,2$,  
 $$
 \gw_{2L_i,1,(2)}^{\CP^1 \times \CP^1,\Delta}([pt];[pt]) = 0
 $$
encodes the fact that holomorphic curves in classes $2 L_i \in H_2(\CP^1 \times \CP^1; \Z)$ are covers of either vertical or horizontal spheres, and therefore cannot go through one generic point in $\CP^1 \times \CP^1$ and another generic point in $\Delta$.

 \item To prove that 
 $$
 \gw_{L_1+L_2,1,(2)}^{\CP^1 \times \CP^1, \Delta}\left(pt; pt \right) = 1
 $$
we show that, up to domain automorphism, there is a unique holomorphic map $U: \CP^1 \to \CP^1 \times \CP^1$, such that $U_*[\CP^1] = L_1+L_2 \in H_2(\CP^1 \times \CP^1;\Z)$, $U(\infty) = (\infty, 1)$, $U(0) = (0,0)$, and such that $U$ intersects the diagonal $\Delta \subset \CP^1 \times \CP^1$ in a non-transverse way. Since 
$$
U(\infty) = U \big([1;0]\big) = (\infty, 1) = \big([1;0],[1;1]\big),
$$
we can write in homogeneous coordinates 
$$
U\big([z;1]\big) = \big([a z + b; 1], [z + c; z + d]\big)
$$ 
which we will abbreviate as 
$$
U(z) = \left(a z + b, \frac{z + c}{z + d}\right).
$$
Since $U(0) = (0,0)$, we get $b=c=0$, and so $U(z) = \big(az, z/(z+d)\big)$. Now, for the tangency condition, note that
$$
U'(z) = \left(a, \frac{d}{(z + d)^2}\right)
$$
and so $U'(0) = \big(a, 1/d\big)$. So $U$ is tangent to the diagonal at $(0,0)$ precisely when $a = 1/d$. Therefore, the space of maps $U$ can be identified with the space of $a\in \C ^*$. Taking a quotient by the group $\C^*$ of automorphisms of the domain $(\CP^1, \{0,\infty\})$, we get the uniqueness of $U$.
\end{enumerate}
\end{proof}

\section{The group \texorpdfstring{$SH_*(T^*S^2)$}{SH*(T*S2)}}

We can now compute the symplectic homology chain complex for $T^*S^2$. Using the functions $f_\Sigma$ and $f_W$ introduced above, we have
$$
SC_*(T^*S^2) = \Z \left\langle e,c\right\rangle \oplus \bigoplus_{k>0} \Z\left\langle \wc m_k, \wh m_k, \wc M_k, \wh M_k \right\rangle
$$
We can use Theorem \ref{T:differential}, Lemma \ref{lacunary GW} and Proposition \ref{P:rel GW computations} to compute the differential. Given $k\geq 2$,
\begin{align*}
\partial \wc m_{k+1} &=  2 \, n_L(M,m) \, \wh M_{k-1} \,+ \,\big( n_{L_1} \, + \, n_{L_2} \big)\, \wh m_k = \\
&= 2 \times 1 \, \wh M_{k-1} \, + \, \big(1 + 1 \big)\, \wh m_k = 2 \, \wh M_{k-1} + 2 \, \wh m_k 
\end{align*}
and
\begin{align*}
\partial \wc M_{k} &= \big( n_{L_1} \, + \, n_{L_2} \big)\, \wh M_{k-1} \,+\langle c_1(\R P^3 \to \CP^1),\CP^1\rangle \, \wh m_{k} = 2 \, \wh M_{k-1} + 2 \, \wh m_k  
\end{align*}

The remaining differentials are
\begin{align*}
\partial \wc m_{2} &= \big(n_{L_1} \, + \, n_{L_2}\big)\, \wh m_1 + 2 \, n_{L_1 + L_2}(e,m) \, e = 2 \, \wh m_1 + 2 \, e ,
\end{align*}
\begin{align*}
\partial \wc M_{1} &= \big(n_{L_1}(e,M) + n_{L_2}(e,M) \big) \, e \,+\, \langle c_1(\R P^3 \to \CP^1),\CP^1\rangle \, \wh m_1 = \\
&= 2 \, e + 2 \, \wh m_1 
\end{align*}
and
\begin{align*}
\partial \wc m_{1} &= \,\big(n_{L_1}(c,m) + n_{L_2}(c,m) \big) \, c = (-1+1)\, c = 0. 
\end{align*}

Summing up, the non-trivial contributions to the differential are
\[\left\{
\begin{array}{l}
\partial \wc{m}_{k+1} = 2 \, \wh{m}_{k} + 2 \, \wh{M}_{k-1} \\
\partial \wc{M}_{k} = 2 \wh{m}_{k} + 2 \, \wh{M}_{k-1} \\
\partial \wc{m}_2 = 2 \, \wh{m}_1 + 2 \, e \\
\partial \wc{M}_1 = 2 \, \wh{m}_1 + 2 \, e
\end{array}
\right.
\]
for $k\geq 2$. We can conclude the following.
\begin{proposition}
\[
SH_*(T^*S^2 ; \Z) = \Z \left \langle c, \wc{m}_{1}, e, \wc{M}_{k} - \wc{m}_{k+1}, \wh{M}_{k} \right \rangle \oplus \Z / 2 \left \langle e + \wh{m}_{1}, \wh{M}_{k} + \wh{m}_{k+1} \right \rangle
\]
where we take all $k \geq 1$. 
\end{proposition}

We can compare these results with earlier computations of the symplectic homology of $T^*S^2$. By Viterbo's Theorem, it is isomorphic as a ring to the homology of the free loop space of $S^2$ \cite{AbbondandoloSchwarzRing}. As a graded abelian group, this was computed by McCleary \cite{McClearyGeodesics}. 
As a ring, it was computed in \cite{CohenJonesYan} to be 
\begin{equation} \label{CJY}
 H_*(LS^2;\Z) \cong ( \Lambda(b) \tensor \Z[a,v] ) / (a^2 , ab, 2av)
\end{equation}
for some $a\in H_0(LS^2;\Z)$, $b\in H_1(LS^2;\Z)$ and $v\in H_4(LS^2;\Z)$. The following tables show how our computations match these, at least as graded abelian groups. For the free part, we get
\[
\begin{array}{c|c|c|c|c|c}
SH_d(T^*S^2) & c & -\wc{m}_{1} & e & \wc{M}_{k} - \wc{m}_{k+1} & \wh{M}_{k} \\
\hline
H_d(L S^2) & a & b & 1 & bv^k & v^k \\
\hline
d & 0 & 1 & 2 & 2k+1 & 2k+2
\end{array}
\]
for $k\geq 1$. 
For the $\Z / 2$-torsion part:
\[
\begin{array}{c|c|c}
SH_d(T^*S^2) & e + \wh{m}_{1} & \wh{M}_{k} + \wh{m}_{k+1} \\
\hline
H_d(L S^2) & av & av^{k+1} \\
\hline
d & 2 & 2k+2
\end{array}
\]
for $k\geq 1$. One can also adapt the arguments presented above to describe the ring structure on symplectic homology and show that it matches \eqref{CJY}, but that is beyond the scope of this paper, and will be addressed in future work.

\def\cprime{$'$}

\end{document}